\documentclass[a4paper,centertags,oneside,12pt]{amsart}
\usepackage{mathrsfs}
\usepackage{appendix}
\usepackage{amssymb}
\usepackage{fancyhdr}
\usepackage{charter}
\usepackage{typearea}
\usepackage{pdfsync}
\usepackage{mathrsfs}
\usepackage{dsfont}
\usepackage{appendix}
\usepackage{amssymb}
\usepackage{fancyhdr}
\usepackage{charter}
\usepackage{typearea}
\usepackage{color}
\usepackage{hyperref}
\usepackage{amsmath,amstext,amsthm,amscd}
\usepackage{mathrsfs}
\usepackage[a4paper,top=2.5cm,bottom=2.5cm,left=2.5cm,
right=2.5cm]{geometry}

\usepackage{enumitem}
\setlist[1]{itemsep=5pt}
\newcommand{\comment}[1]{}

\makeatletter
      \def\@setcopyright{}
      \def\serieslogo@{}
      \makeatother

\newcommand{\C}{\mathbb{C}}
\newcommand{\R}{\mathbb{R}}
\newcommand{\N}{\mathbb N}
\newcommand{\Z}{\mathbb Z}
\newcommand{\T}{\mathbb T}
\newcommand{\ddbar}{\overline\partial}
\newcommand{\pr}{\partial}
\newcommand{\ol}{\overline}
\newcommand{\Td}{\widetilde}
\newcommand{\norm}[1]{\left\Vert#1\right\Vert}
\newcommand{\abs}[1]{\left\vert#1\right\vert}

\newcommand{\set}[1]{\left\{#1\right\}}
\newcommand{\To}{\rightarrow}
\DeclareMathOperator{\Ker}{Ker}
\DeclareMathOperator{\supp}{supp}
\DeclareMathOperator{\Dom}{Dom}

\newcommand{\bb}{\boldsymbol{b}}
\newcommand{\cC}{\mathscr{C}}

\newtheorem{theorem}{Theorem}[section]
\newtheorem{lemma}[theorem]{Lemma}
\newtheorem{corollary}[theorem]{Corollary}
\newtheorem{proposition}[theorem]{Proposition}

\newtheorem{question}[theorem]{Question}
\theoremstyle{definition}
\newtheorem{definition}[theorem]{Definition}

\newtheorem{remark}[theorem]{Remark}
\newtheorem{assumption}[theorem]{Assumption}
\numberwithin{equation}{section}
\begin{document}
\title[Semi-classical asymptotics of partial Bergman kernels]
{Semi-classical asymptotics of partial Bergman kernels
on $\R$-symmetric complex manifolds with boundary}
\author[Chin-Yu Hsiao]{Chin-Yu Hsiao}
\address{Institute of Mathematics, Academia Sinica, 
Astronomy-Mathematics Building, No. 1, Sec. 4, Roosevelt Road, Taipei 10617, Taiwan}
\thanks{Chin-Yu Hsiao was partially supported by Taiwan Ministry of Science 
and Technology projects  
108-2115-M-001-012-MY5, 109-2923-M-001-010-MY4 and Academia Sinica investigator
Award. }
\email{chsiao@math.sinica.edu.tw or chinyu.hsiao@gmail.com}
\author[Xiaoshan Li]{Xiaoshan Li}
\address{School of Mathematics
and Statistics, Wuhan University, Wuhan 430072, Hubei, China}
\thanks{Xiaoshan Li was supported by NSFC 12271411 and NSFC 11871380.}
\email{xiaoshanli@whu.edu.cn}
\author[George Marinescu]{George Marinescu}
\address{Department of Mathematics and Computer Science, 
Universit{\"a}t zu K{\"o}ln\\ 
Weyertal 86-90, \newline
    \mbox{\quad}\,50931 K{\"o}ln, Germany\\
    \newline
\mbox{\quad}\,Institute of Mathematics `Simion Stoilow', Romanian Academy,
Bucharest, Romania}
\thanks{George Marinescu was partially supported by the 
DFG funded project SFB TRR 191 `Symplectic Structures in 
Geometry, Algebra and Dynamics' (Project-ID 281071066 TRR 191) 
and ANR-DFG funded project `Quantisierung, Singularit\"aten und Holomorphe Dynamik' (Project-ID 490843120 MA 2469/6-1)}
\email{gmarines@math.uni-koeln.de}
\begin{abstract}
Let $M$ be a relatively compact connected open subset 
with smooth connected boundary of a complex manifold $M'$. 
Let $(L, h^L)\rightarrow M'$ be a positive line bundle over $M'$. 
Suppose that $M'$ admits a holomorphic $\R$-action which 
preserves the boundary of $M$ and 
lifts to $L$. We establish 
the asymptotic expansion of a partial Bergman kernel 
associated to a package of Fourier modes of high frequency
with respect to the $\R$-action
in the high powers of $L$.
As an application, we establish 
an $\R$-equivariant analogue of Fefferman's and Bell-Ligocka's
result about smooth extension up to the boundary of biholomorphic maps 
between weakly pseudoconvex domains in $\C^n$. 
Another application concerns the embedding of pseudoconcave 
manifolds.
\end{abstract} 
\maketitle \tableofcontents
\section{ Introduction}\label{sec1}
Let $M$ be a relatively compact connected open 
subset with a smooth connected boundary $X=\partial{M}$ of a 
complex manifold $M'$ of complex dimension $n$, with complex structure $J$. 
Let $(L,h^L)\To M'$ be a holomorphic line bundle over $M'$. Let 
\[H^0(\ol M,L):=\set{u\in\cC^\infty(\ol M,L);\, \ddbar u=0},\]
where $\ddbar$ is the Cauchy-Riemann operator acting on sections of $L$. 
Let $H^0_{(2)}(M,L)\subset L^2(M,L)$ be the $L^2$-closure 
of $H^0(\ol M,L)$ and let 
\[\Pi: L^2(M,L)\To H^0_{(2)}(M,L)\]
be the orthogonal projection (Bergman projection). 
The study of boundary behavior of $\Pi$ is a classical subject 
in several complex variables. 
When $M$ is strictly pseudoconvex and $L$ is trivial, 
Fefferman~\cite{Fer74} obtained an asymptotic expansion of 
$\Pi$ on the diagonal of the boundary. 
Subsequently, Boutet de Monvel-Sj\"{o}strand~\cite{BS76} 
described the singularities of the distribution kernel of $\Pi$ 
by showing that it is a complex Fourier integral operator (see also~\cite{Hsiao18},~\cite{HM19}). 
When $M$ is not strictly pseudoconvex, for example, if $M$ is strictly 
pseudoconcave or weakly pseudoconvex, 
it is difficult to study $\Pi$. When $M$ is a weakly pseudocovex 
finite type domain in $\mathbb C^2$, the first author and Savale~\cite{HS22} 
obtained a pointwise Fefferman type asymptotic expansion for $\Pi$. 

In general, it is very difficult to study $\Pi$ when $M$ 
is not strictly pseudoconvex, for example for the weakly pseudoconvex
domain
\[M:=\set{(z_1,z_2)\in\mathbb C^2;\, 
\abs{z_1}^2+e^{-\frac{1}{\abs{z_2}^2}}<1}.
\]  
which is not of finite type. 
However, $M$ admits a holomorphic $S^1$-action: 
$e^{i\theta}\circ(z_1,z_2)=(e^{i\theta}z_1,e^{i\theta}z_2)$ 
and admits an $S^1$-invariant strictly plurisubharmonic weight function 
$\abs{z}^2$ (so that the trivial line bundle $M\times\C$ 
endowed with the metric $e^{-2\abs{z}^2}$ is a positive line bundle). 
It is quite interesting to 
study $\R$-equivariant (or $S^1$-equivariant) 
semi-classical boundary behavior of Bergman kernel asymptotics 
with respect to $e^{-2k\abs{z}^2}$ as $k\To+\infty$ 
(high power of a positive line bundle). This is the motivation of this work. 
In this paper, we will study semi-classical boundary behavior of 
$\mathbb R$-equivariant Bergman kernel with respect to a high 
power of positive line bundle. The following general question 
is closely related to the problem of extending 
the Kodaira embedding theorem to complex manifolds with boundary and 
are fundamental questions in the study of complex manifolds with boundary.


\begin{question}\label{q-gue220512yyd}
Let $M$ be a relatively compact connected open subset with smooth 
boundary of a complex manifold $M'$ of complex dimension $n$. 
Suppose that there is a positive holomorphic line bundle 
$(L, h^L)\rightarrow M'$ over $M'$. Let 
\(\Pi_k: L^2(M,L^k)\To H^0_{(2)}(M,L^k)\)
be the Bergman projection associated to the tensor power $L^k$ of $L$. 
Can we understand the large $k$ behavior of the Bergman kernel 
or partial Bergman kernels near the boundary of $M$? 
\end{question}

Question~\ref{q-gue220512yyd} is related to
problem of compactification of pseudoconcave
manifolds to compact projective or Moishezon manifolds
and was studied by many authors \cite{A63,AS,AT,Be05,DM,EH,M96,M16,NT,ST82}.
In \cite[p.\,29]{M96} and in discussions with Gennadi Henkin 
\cite{He} it was conjectured
that a characterization of Moishezon 1-concave manifolds
in terms of positive sheaves analogous to the 
Grauert-Riemenschneider characterization is possible. 
The boundary behavior of $\Pi_k$ is closely related to the problem 
of extending the Kodaira embedding theorem to complex  manifolds 
with boundary and are fundamental questions in the study of 
complex manifolds with boundary. The difficulty of 
Question~\ref{q-gue220512yyd} comes form the fact 
that we do not know if the associated $\ddbar$-Neumann Laplacian 
has a spectral gap. The spectral gap property plays an important 
role in the study of Bergman kernel asymptotic expansions on 
complex manifolds without boundary (see~\cite{MM07}). 
The boundary $X$ is a compact CR manifold and $L$ is a CR line bundle 
over $X$. In~\cite{HLM},~\cite{HHL17}, it was shown that if $X$ admits a 
transversal and CR $\R$-action and $(L,h^L)$ is a $\R$-equivariant 
positive CR line bundle, then the associated Kohn Laplacian has 
\emph{partial spectral gap} and the 
associated Szeg\H{o} kernel admits a full asymptotic expansion. 
Therefore, it is natural to study Question~\ref{q-gue220512yyd} 
when $M'$ admits a holomorphic $\R$-action. 
A significant difference between the CR case and the case of complex manifolds 
with boundary is that even when $M'$ admits a holomorphic $\R$-action, 
it is still very difficult to see if the associated $\ddbar$-Neumann Laplacian 
has a spectral gap or \emph{partial spectral gap}. 
From the Szeg\H{o} kernel asymptotic expansion obtained 
in~\cite{HHL17,HLM}and by carefully studying the semi-classical Poisson 
operator and using a reduction to the boundary technique, 
we establish an asymptotic expansion for a partial Bergman kernel of 
on $M$ with respect to high powers 
of a positive line bundle $L$ under certain natural assumptions.

We now formulate the main results. 
We refer to Section~\ref{sec2} for some notations and terminology used here.
Let $(M',J)$ be a complex manifold of dimension $n$ with complex structure $J$. 
Let $M$ be a relatively compact domain with smooth connected boundary 
of $M'$. We denote by $X=\partial M$ 
the boundary of $M$. We 
assume that $M'$ admits a holomorphic $\R$-action $\eta$ 
(see Definition~\ref{d-gue220512yyd}), $\eta\in\R$: $\eta: M'\to M'$, 
$x\mapsto\eta\circ x$.  Let $T\in\cC^\infty(M', TM')$ 
be the infinitesimal generator of the $\R$-action 
(see \eqref{e-gue150808}).  We 
work under the following assumption.

\begin{assumption}\label{a-gue170929I}
(a) The $\R$-action is holomorphic, preserves the boundary of $M$ and
is CR-transversal to the boundary, that is,
$$\C T(x)\oplus T^{1,0}_xX\oplus T^{0,1}_xX=
\C T_xX,\ \ \mbox{for all $x\in X$},$$
where $T^{1,0}X:=\mathbb CTX\cap T^{1,0}M'$, $T^{0,1}X:=
{\C}TX\cap T^{0,1}M'$.

\noindent
(b) There exists an $\R$-equivariant holomorphic 
line bundle $L\rightarrow M'$ endowed with an $\R$-invariant Hermitian fiber 
metric $h^L$ such that $iR^L$ is positive on $M'$, where 
$R^L$ is the Chern curvature of $L$ induced by $h^L$. 
\end{assumption}

Let us consider an $\R$-invariant defining function
$\rho\in\cC^\infty(M')$ of $M$.
Without loss of generality, we assume that 
\begin{equation}\label{e-gue220711yyd}
d\rho(JT)<0\:\:\text{on $X$}.
\end{equation}
If $d\rho(JT)>0$ on $X$, 
we can modify the $\R$-action so that $T$ becomes to $-T$. 

Let $s$ be a local $\R$-equivariant holomorphic trivializing 
section of $L$, $\abs{s}^2_{h^L}=e^{-2\phi}$. 
Since $h^L$ is $\R$-equivariant, it is not difficult to see that 
$(JT)(\phi)$ is globally defined (see Proposition \ref{2023-8-11-p1}), 
where $J$ denotes the complex structure map on $TM'$. 
We work under the following assumption.

\begin{assumption}\label{a-gue220517yyd}
There is a constant $C_0>0$ so that 
\begin{equation}\label{e-gue220710yyd}
\begin{split}
&\mbox{$R^L_x+(C_0+2(JT)(\phi)(x))\mathcal{L}_x$ 
is positive definite on $T^{1,0}_xX$ at every $x\in X$},\\
&C_0+2(JT)(\phi)(x)>0,\ \ \mbox{for all $x\in X$},
\end{split}
\end{equation}
where $\mathcal{L}_x$ denotes the Levi form of $X$ 
at $x\in X$ (see Definition~\ref{d-gue220507yyd} below).
\end{assumption}  

\begin{remark}\label{r-gue220710yyd}
(i) Let $R^L_X$ be the CR curvature of $L$ on the boundary $X$ (see Definition~\ref{d-gue220507yydI}). By Lemma \ref{L:RX} we have 
\begin{equation}\label{e-gue220711yydI}
R^L_{X,x}=R^L|_{T^{1,0}_xX}+2(JT)(\phi)(x)\mathcal{L}_x,\ \ x\in X.
\end{equation}
Thus, Assumption \ref{a-gue220517yyd} requires
that the CR curvature $R^L_{X}$ dominates the Levi form 
$\mathcal{L}$ in the negative directions of $\mathcal{L}$. 
Even if $R^L$ is positive, the CR curvature $R^L_X$ might not be positive on $X$. 
It follows from \cite{HHL17,HLM} that if there exists $\delta\in\R$ such that
$R^L_X+\delta\mathcal{L}$ is everywhere positive, the  
asymptotic expansion of the Szeg\H{o} kernel holds. 
To obtain the Bergman kernel asymptotics, 
we need the conditions \eqref{e-gue220710yyd}.
 
(ii) Assumption \ref{a-gue220517yyd} always holds if $\mathcal{L}_x$
is semi-positive at every $x\in X$, that is, if $M$ is weakly pseudoconvex. 
    
(iii) In Section \ref{S:psconc} we give examples of pseudoconcave manifolds
which satisfy Assumption \ref{a-gue220517yyd}. 
\end{remark}

We denote by $T^{1,0}M'$ and the holomorphic and anti-holomorphic
tangent bundles of $M'$, and by $T^{*p,q}M'$ the bundle
of $(p,q)$-forms. 
We fix a $J$-invariant and $\R$-invariant Riemannian metric $g^{TM'}$ on $TM'$
such that $|T|=1$.
We denote by $\langle\,\cdot\,|\,\cdot\,\rangle$ the
pointwise Hermitian product induced by $g^{TM'}$
on the fibers of $\C TM'$ and by duality on 
$\C T^{*}M'$, hence on $T^{*p,q}M'$.
We also fix a $\R$-invariant Hermitian 
fiber metric $h^L$ such that $R^L$ is positive on $T^{1,0}M'$.
 
Let $\rho\in\cC^\infty(M',\R)$ be a defining function of $X$, 
that is, $\rho=0$ on $X$, $\rho<0$ on $M$ and $d\rho\neq0$ near $X$. 
From now on, we take a defining function $\rho$ 
so that $\rho$ is $\R$-invariant and $|d\rho|=1$ on $X$. 

We will consider $(L^k,h^{L^k})$ the $k$-th power of $(L,h^L)$. 
Let $dv_{M'}$ be the volume form on $M'$ induced by the Hermitian 
metric $\langle\,\cdot\,|\,\cdot\,\rangle$ on $\C TM'$ and 
let $(\,\cdot\,|\,\cdot\,)_{\ol M}$  be the inner product on 
$\cC^\infty(\ol M)$ induced by $\langle\,\cdot\,|\,\cdot\,\rangle$.  
Similarly, let $(\,\cdot\,|\,\cdot\,)_{\ol M,k}$  be the inner product on $\cC^\infty(\ol M,L^k)$ induced by $h^{L^k}$ and $\langle\,\cdot\,|\,\cdot\,\rangle$.  
Let $L^2(M,L^k)$ be the completion of $\cC^\infty(\ol M,L^k)$ with respect to $(\,\cdot\,|\,\cdot\,)_{\ol M,k}$ and we extend $(\,\cdot\,|\,\cdot\,)_{\ol M,k}$ to $L^2(M,L^k)$
in the standard way. Let $\norm{\cdot}_{\ol M}$ and $\norm{\cdot}_{\ol M,k}$ be the corresponding norms with respect to $(\,\cdot\,|\,\cdot\,)_{\ol M}$ and $(\,\cdot\,|\,\cdot\,)_{\ol M,k}$ respectively.  Let 
\[\ddbar_k: \cC^\infty(\ol M,L^k)\To\Omega^{0,1}(\ol M,L^k)\]
be the Cauchy-Riemann operator with values in $L^k$.
We shall also use the notation $\ddbar_k$ for the closure in $L^2$ 
of the $\ddbar_k$ operator, initially defined on $\cC^\infty(\ol M,L^k)$. 
Note that for this extension we have $\Ker\ddbar_k=H^0_{(2)}(M,L^k)$.

We shall use the notation $-iT_M$ for the closure in $L^2$ of the $-iT$ 
operator, initially defined on
$\cC^\infty(\ol M,L^k)$. Since $T\rho=0$, we can check that 
$$-iT_M: \Dom(-iT_M)\subset L^2(M,L^k)\To L^2(M,L^k)$$ 
is self-adjoint and its spectrum consists of a countable set of
eigenvalues of finite multiplicity (cf.\ Theorem \ref{t-gue220317yyd}).
The Fourier (spectral) decomposition \eqref{e:FD} of $L^2(M,L^k)$
determined by $T$, will provide us with a convenient
space of Fourier modes (eigensections of $T$), whose weighted projection
turns out to have an asymptotic expansion.

We fix $\delta>0$ so that:
\begin{equation}\label{eq:deltadef}
\begin{split}
&R^L_{X,x}+2t\mathcal{L}_x \:\:\text{is positive definite at every $x\in X$, 
for every $t\in[\delta/4,2\delta]$},\\
&t+(JT)(\phi)(x)>0\:\:\text{for every $t\in(\delta/4,2\delta)$ 
and every $x\in X$},
\end{split}
\end{equation}
where 
$R^L_{X,x}$ and $\mathcal{L}_x$ denote the CR curvature of $L$ 
and the Levi form of $X$ at $x\in X$ respectively 
(see Definitions~\ref{d-gue220507yyd} and \ref{d-gue220507yydI} ). 
From \eqref{e-gue220710yyd}, this is always possible. 
We fix a function 
\begin{equation}\label{e-21-5-16-a2z}
\tau_\delta\in\cC^\infty_c((\delta/4, 2\delta)), \:\:
0\leq\tau_{\delta}\leq 1, \:\:\tau_{\delta}\equiv1~\text{on}~[\delta/2, \delta].
\end{equation}
We set
\begin{equation}\label{e-gue220322ycdsz}
\tau_{k\delta}\in\cC^\infty_c(\R),\:\:
\tau_{k\delta}(t):=\tau_\delta(t/k),\ \ t\in\R.
\end{equation}
We consider the following spectral operator which is a 
weighted projection on the Fourier modes (i.\,e.\ eigenfunctions of $T$)
\begin{equation}\label{eq:FProjk}
F^{(0)}_{\tau_{k\delta},M}:=
\tau_{k\delta}(-iT_M): L^2(M,L^k)\To L^2(M,L^k),
\end{equation}
where $\tau_{k\delta}(-iT_M)$ denotes the functional calculus of 
$-iT_M$ with respect to $\tau_{k\delta}$ 
(see \eqref{e-gue220317yyda}).  
Let now
\begin{equation}\label{eq:BergProjk}
\Pi_k: L^2(M,L^k)\To H^0_{(2)}(M,L^k)
\end{equation}
be the orthogonal projection with respect to $(\,\cdot\,|\,\cdot\,)_{\ol M,k}$. 
One of our main objets of study is the weighted \emph{partial Bergman projection} 
\begin{equation}\label{e-gue220512yydb}
\Pi_{\tau_{k\delta}}:=\Pi_k\circ F^{(0)}_{\tau_{k\delta},M}: 
L^2(M,L^k)\To H^0_{(2)}(M,L^k).
\end{equation}

Let $s$ be an $\R$-equivariant holomorphic frame of $L$ 
defined on an open set $U$ of $M'$, $\abs{s}^2_{h^L}=e^{-2\phi}$. 
The localization of $\Pi_{\tau_{k\delta}}$ with respect to $s$ is given by 
\[
\begin{split}
&\Pi_{\tau_{k\delta},s}:
\cC^\infty_c(U\cap\ol M)\To L^2(U\cap M)\\
&\Pi_{\tau_{k\delta}}(f e^{k\phi} s^{k})=
\Pi_{\tau_{k\delta},s}(f)e^{k\phi}s^{k},\:\: f\in\cC^\infty_c(U\cap\ol M).
\end{split}
\]
Let $\Pi_{\tau_{k\delta},s}(x,y)\in\mathscr D'((U\times U)\cap(M\times M))$ 
be the distribution kernel of $\Pi_{\tau_{k\delta},s}$. 
The main result of this work is the following.

\begin{theorem}\label{t-gue220512yyd}
Let $M$ be a relatively compact domain with smooth connected boundary 
of a complex manifold $M'$ of complex dimension $n$
and let $L\rightarrow M'$ be an $\R$-equivariant holomorphic 
line bundle such that Assumptions 
\ref{a-gue170929I} and \ref{a-gue220517yyd} are
satisfied. Let $\tilde\chi, \hat\chi\in\cC^\infty(\ol M)$ 
with $\supp\tilde\chi\cap\supp\hat\chi=\emptyset$. Then, 
\begin{equation}\label{e-gue220510yydsz}
\mbox{$\tilde\chi\Pi_{\tau_{k\delta}}\hat\chi\equiv0\mod 
O(k^{-\infty})$ on $\ol M\times\ol M$}. 
\end{equation} 
Let $s$ be a $\R$-equivariant holomorphic frame of $L$ on an open set 
$U$ of $M'$, $\abs{s}^2_{h^L}=e^{-2\phi}$. If $U\cap X=\emptyset$
we have 
\[\mbox{$\Pi_{\tau_{k\delta}}\equiv0\mod O(k^{-\infty})$ on 
$(U\times U)\cap(\ol M\times\ol M)$}.\]
Suppose that $D:=U\cap X\neq\emptyset$. 
Let $\chi, \chi_1\in\cC^\infty_c(U\cap\ol M)$, $\chi_1\equiv1$ near $\supp\chi$. Then, 
\begin{equation}\label{e-gue220510yydtz}
(\chi_1\Pi_{\tau_{k\delta},s}\chi)(x,y)\equiv\int_{\R}e^{ik\Psi(x,y,t)}b(x,y,t,k)dt
\mod O(k^{-\infty}),
\end{equation}
on $(U\times U)\cap(\ol M\times\ol M)$,
where $(\chi_1\Pi_{\tau_{k\delta},s}\chi)(x,y)$ 
is the distribution kernel of $\chi_1\Pi_{\tau_{k\delta},s}\chi$, 
$\Psi(x,y, t)\in\cC^\infty(((U\times U)\cap(\ol M\times\ol M))
\times(\delta/4,2\delta))$, and for all $(x,y,t)\in((U\times U)\cap(\ol M\times\ol M))
\times(\delta/4,2\delta)$, 
\begin{equation}\label{e-gue220505yyduz}
\begin{split}
&{\rm Im\,}\Psi\geq0,\quad
\mbox{${\rm Im\,}\Psi>0$ if $(x,y)\notin (U\times U)\cap(X\times X)$},\\
&\Psi(x,x,t)=0,\ \ x\in U\cap X,\ \ \Psi(x,y,t)\neq0\ \ \mbox{if $(x,y)\notin{\rm diag\,}((U\times U)\cap(X\times X))$},\\
&d_1\Psi(x,x,t)\neq0,\ \ d_2\Psi(x,x,t)\neq0,\ \ x\in U\cap\ol M,
\end{split}
\end{equation} 
\begin{equation}\label{e-gue220510vz}
\begin{split}
&b(x,y,t,k)\in (S^{n+1}(1)\cap \cC^\infty_c)(((U\times U)\cap
(\ol M\times\ol M))\times(\delta/4,2\delta))
,\\
&\mbox{$b(x,y,t,k)\sim\sum^{+\infty}_{j=0}b_j(x,y,t)k^{n+1-j}$ 
in $S^{n+1}(1;((U\times U)\cap(\ol M\times\ol M))\times(\delta/4,2\delta))$},\\
 &b_j(x,y,t)\in (S^0(1)\cap\cC^\infty_c)(((U\times U)\cap(\ol M\times\ol M))
 \times(\delta/4,2\delta)),
 \:\:j=0,1,\ldots,
\\&b_0(x,x,t)>0. 
 \end{split}\end{equation}
Moreover,
\begin{equation}\label{e-gue220510wz}
 b_0(x,x,t)=(2\pi)^{-n}\chi(x)\abs{\det(R^L_{X,x}+2t\mathcal{L}_x)}
 \abs{\tau_\delta(t)}^2(2i)\varphi_1(x,d_x\Phi(x,x,t)),
\end{equation} 
for all $(x,x,t)\in D\times D\times(\delta/4,2\delta)$, where 
$\varphi_1$ is given in \eqref{e-gue220426yydz} below.
\end{theorem}

We refer the reader to \eqref{e-gue190813yyd} and the discussion in the 
end of Section~\ref{s-gue150808} for the notations used in 
Theorem~\ref{t-gue220512yyd}.
The phase function $\Phi:=\Psi|_{D\times D}$ is the same as the 
phase function appearing in the description of the Szeg\H{o} kernels 
in~\cite[Theorem 1.1]{HLM},~\cite[Theorem 1.1]{HHL17} 
(see also \eqref{e-gue220430yyd}). Moreover, we have 
\begin{equation}\label{e-gue220505yyduzz}
\begin{split}
&\Psi(\tilde x,\tilde y, t)=\Phi(x,y,t)+x_{2n}\varphi_1(x,d_x\Phi(x,y,t))-
y_{2n}\ol{\varphi_1(y,-d_y\ol\Phi(x,y,t))}\\
&\quad+O(\abs{(x_{2n},y_{2n})}^2),\\
&{\rm Im\,}\Psi(\tilde x,\tilde y, t)\geq C(\abs{x_{2n}}+\abs{y_{2n}}),
\end{split}
\end{equation}
for every $(\tilde x, \tilde y,t)\in((U\times U)\cap(\ol M\times\ol M))
\times(\delta/4,2\delta)$, where $C>0$ is a constant and we 
take the branch of the functions 
$$\varphi_1(x,d_1\Phi(x,y,t)), \:\:-\ol{\varphi_1(y,-d_2\ol\Phi(x,y,t))}$$
so that $i\varphi_1(x,d_1\Phi(x,x,t))>0$, $-i\ol{\varphi_1(y,-d_2\ol\Phi(y,y,t))}>0$, 
for every $(x,x,t)\in D\times D\times(\delta/4,2\delta)$.



\subsection{Applications}\label{s-gue230718yyd}
One of the main applications of the asymptotic 
expansion of the Bergman kernel at the boundary
is the following fundamental theorem of Fefferman \cite[Theorem 1]{Fer74}:
\emph{A biholomorphic map $f:M_1\to M_2$ of smoothly bounded strictly
pseudoconvex domains in $\C^n$ is smooth up to the boundary.}
It is a several complex variables generalization of the classical 
theorem of Painlev\'e \cite{Pain}
for smooth domains in the complex plane. 
Fefferman's theorem was further generalized by Bell and Ligocka \cite{BeLi80}. 
They proved that if the Bergman projections associated to 
smoothly bounded weakly pseudoconvex domains $M_1$
and $M_2$ preserve the space of functions which are smooth up to the
boundary, then $f$ extends to be a diffeomorphism of $\overline{M}_1$
onto $\overline{M}_2$.

As an application of Theorem~\ref{t-gue220512yyd}, 
we have the following $\R$-equivariant analogue of Fefferman's  and Bell-Ligocka's
extension result for weakly pseudoconvex domains in $\C^n$. 

\begin{theorem}\label{t-gue230718yyd}
Let $M_j$, $j=1,2$, be bounded weakly pseudoconvex domains of $\mathbb C^n$. 
Let $\eta_j$, $j=1,2$, be holomorphic $\mathbb R$-actions on $\mathbb C^n$
such that that $\eta_j$ acts on $M_j$ and is transversal to the CR 
structure of the boundary of $M_j$.
Suppose that there is a $\eta_j$-invariant strictly plurisubharmonic 
function $\phi\in\cC^\infty(\mathbb C^n)$, $j=1,2$.
Let $F: M_1\To M_2$ be an $\mathbb R$-equivariant 
biholomorphic map. 
Then $F$ extends smoothly to the boundary of $M_1$.
\end{theorem}

We say that the action $\eta_j$ is transversal to the CR structure 
of the boundary of $M_j$ if $\C T X_j=T^{1,0}X_j\oplus T^{0,1}X_j\oplus\C T_j$, 
where $X_j=\pr M_j$, $T_j$ is the vector field on $X_j$ 
induced by $\eta_j$, $T^{1,0}X_j:=T^{1,0}\C^n\cap\C TX_j$,  $T^{0,1}X_j:=T^{0,1}\C^n\cap\C TX_j$. 
We say that $F$ is $\R$-equivariant if
$F(\eta_1\circ x)=\eta_2\circ F(x)$, for every $x\in M_1$. 
We will prove Theorem~\ref{t-gue230718yyd} in Section~\ref{s-gue230718yydI}.

Let $I_k:=[\frac{k\delta}{2},k\delta]$. Put 
\begin{equation}\label{e-gue220513yyd}
H^0_{I_k}(\ol M,L^k):=\set{u\in H^0(\ol M,L^k);\, 
\mathds{1}_{[k\delta/2,k\delta]}(-iT_M)u=u}.
\end{equation}
This is the space spanned by eigensections of $-iT_M$
on $H^0(\ol M,L^k)$ whose eigenvalues lie in the interval $[k\delta/2,k\delta]$. 
By using Theorem~\ref{t-gue220512yyd} we will prove in 
Section~\ref{s-gue220503yyd} the following.
\begin{theorem}\label{t-gue220513yyd}
Let $M$ and $L$ be as Theorem \ref{t-gue220512yyd}.
We fix $\delta>0$ such that \eqref{eq:deltadef} is satisfied.
Then as $k\To+\infty$, 
\begin{equation}\label{e-gue220513yydI}
\dim H^0_{I_k}(\ol M,L^k)=k^n(2\pi)^{-n}\int_X\int^{\delta}_{\frac{\delta}{2}}
\det(R^L_{X,x}+2t\mathcal{L}_x)\,dt\,dv_X(x)+o(k^n),
\end{equation}
where 
$dv_X(x)$ is the volume form on $X$ induced by the Hermitian 
metric $\langle\,\cdot\,|\,\cdot\,\rangle$ on $\mathbb CTM'$.
\end{theorem}
Note that the integral in \eqref{e-gue220513yydI} is positive, 
so \eqref{e-gue220513yydI}
shows that $\dim H^0_{I_k}(\ol M,L^k)$ grows as $k^n$ as
$k\to+\infty$. 
We obtain from Theorem \ref{t-gue220513yyd}
the holomorphic Morse inequality
\begin{equation}\label{e-gue220513yydII}
\dim H^0(M,L^k)\geq k^n(2\pi)^{-n}\int_X\int^{\delta}_{\frac{\delta}{2}}
\det(R^L_{X,x}+2t\mathcal{L}_x)\,dt\,dv_X(x)+o(k^n),\:\:k\to+\infty.
\end{equation}
Holomorphic Morse inequalities for
manifolds with boundary were studied in \cite{Be05,M96,M16},
see \cite{MM07} for a comprehensive study. 

We further enhance the dimension estimate \eqref{e-gue220513yydI}
by showing the asymptotic expansion of the partial Bergman 
kernel $\Pi_{\tau_{k\delta}}(\cdot,\cdot)$ on the diagonal.
This implies that the sections of $H^0_{I_k}(\ol M,L^k)$ holomorphically
embed a neighborhood of the boundary, and the embedding
is smooth up to boundary.
\begin{corollary}
Let $M$ and $L$ be as Theorem \ref{t-gue220512yyd}.
There exists a neighborhood $V$ of $X=\partial M$
such that we have the following asymptotics in any $\cC^{\ell}$-topology 
on $\overline{V\cap M}$,
\begin{equation}\label{eq:dgexp}
\Pi_{\tau_{k\delta}}(x,x)=
\sum_{j=0}^\infty \bb_j(x)k^{n-j}+O(k^{-\infty}),\quad \bb_0(x)>0, \:\:
x\in\overline{V\cap M}.
\end{equation}
Moreover, the Kodaira map 
$\overline{V\cap M}\to\mathbb{P}(H^0_{I_k}(\ol M,L^k)^*)$ 
associated to $H^0_{I_k}(\ol M,L^k)^*$ 
is a smooth embedding, holomorphic on $V\cap M$.
\end{corollary}

The notions of Moishezon manifold and big line bundle, well known
for compact manifolds, make also sense for
pseudoconcave manifolds. 
A connected complex manifold $M$ is called pseudoconcave in the sense of
Andreotti \cite{A63}
if there exists a non-empty open set $D\subset M$ with smooth boundary
$\partial D$ such that the Levi form of $D$ restricted to the analytic 
tangent space of $\partial D$
has at least one negative eigenvalue at each point.
Examples are provided by $q$-concave manifolds ($q<\dim M$)
in the sense of Andreotti-Grauert \cite{AG}.
It is well known \cite{A63} that for  
a pseudoconcave manifold $M$ the transcendence degree  
$\operatorname{deg\,tr}\mathcal{K}(X)$ of the meromorphic function 
field $\mathcal{K}(M)$ is at most the complex dimension 
$n$ of $M$. 
In analogy to the corresponding notion for compact manifolds, we say that 
a pseudoconcave manifold is {\it Moishezon} if 
$\operatorname{deg\,tr}\mathcal K(X)=\dim M$. 

For any pseudoconcave manifold $M$ and
any holomorphic vector bundle $E\to M$, the space
of holomorphic sections $H^0(M,E)$ is finite dimensional.
A holomorphic line bundle  $L\to M$ is called big if 
$\liminf_{k\rightarrow\infty}k^{-n}\dim H^0(M, L^k)>0$.
If $L$ is big, then for $k$ big enough the sections
of $H^0(M, L^k)$ give local coordinates
on a Zariski open set of $M$.
 

\begin{corollary}\label{C:Moi}
Let $M$ and $L$ be as in Theorem \ref{t-gue220513yyd}.

(a) Assume that $M$ is pseudoconcave in the sense of Andreotti. 
Then $L$ is big and $M$ is Moishezon.

(b) Assume that
$\dim M\geq3$ and $M$ is strictly pseudoconcave, that is, the Levi form 
of $\partial M$ is negative definite on $\partial M$. 
Then there exists a compact
Moishezon manifold $\widetilde{M}$ and an open set $M_0\subset\widetilde{M}$
such that $M$ is biholomorphic to $M_0$.
\end{corollary}
\noindent
In case (b) we say that $M$ can be compactified to a compact
Moishezon manifold $\widetilde{M}$.


\section{Set up}\label{sec2}
\subsection{Some standard notations}\label{s-gue150508b}
We use the following notations: $\mathbb Z$ 
is the set of integers, $\mathbb N=\{1, 2, \ldots, \}$, $\mathbb N_{0}=\mathbb N\cup\{0\}$. 
$\R$ 
is the set of real numbers, $\ol\R_+:=\set{x\in\R;\, x\geq0}$. 
For a multiindex $\alpha=(\alpha_1,\ldots,\alpha_n)\in\mathbb N_0^n$
we denote by $\abs{\alpha}=\alpha_1+\ldots+\alpha_n$ its norm and by $l(\alpha)=n$ its length.
For $m\in\mathbb N$, write $\alpha\in\set{1,\ldots,m}^n$ if $\alpha_j\in\set{1,\ldots,m}$, 
$j=1,\ldots,n$. $\alpha$ is strictly increasing if $\alpha_1<\alpha_2<\ldots<\alpha_n$. For $x=(x_1,\ldots,x_n)$ we write
\[
\begin{split}
&x^\alpha=x_1^{\alpha_1}\ldots x^{\alpha_n}_n,\\
& \pr_{x_j}=\frac{\pr}{\pr x_j}\,,\quad
\pr^\alpha_x=\pr^{\alpha_1}_{x_1}\ldots\pr^{\alpha_n}_{x_n}=\frac{\pr^{\abs{\alpha}}}{\pr x^\alpha}\,,\\
&D_{x_j}=\frac{1}{i}\pr_{x_j}\,,\quad D^\alpha_x=D^{\alpha_1}_{x_1}\ldots D^{\alpha_n}_{x_n}\,,
\quad D_x=\frac{1}{i}\pr_x\,.
\end{split}
\]
Let $z=(z_1,\ldots,z_n)$, $z_j=x_{2j-1}+ix_{2j}$, $j=1,\ldots,n$, be coordinates of $\C^n$.
We write
\[
\begin{split}
&z^\alpha=z_1^{\alpha_1}\ldots z^{\alpha_n}_n\,,\quad\ol z^\alpha=\ol z_1^{\alpha_1}\ldots\ol z^{\alpha_n}_n\,,\\
&\pr_{z_j}=\frac{\pr}{\pr z_j}=
\frac{1}{2}\Big(\frac{\pr}{\pr x_{2j-1}}-i\frac{\pr}{\pr x_{2j}}\Big)\,,\quad\pr_{\ol z_j}=
\frac{\pr}{\pr\ol z_j}=\frac{1}{2}\Big(\frac{\pr}{\pr x_{2j-1}}+i\frac{\pr}{\pr x_{2j}}\Big),\\
&\pr^\alpha_z=\pr^{\alpha_1}_{z_1}\ldots\pr^{\alpha_n}_{z_n}=\frac{\pr^{\abs{\alpha}}}{\pr z^\alpha}\,,\quad
\pr^\alpha_{\ol z}=\pr^{\alpha_1}_{\ol z_1}\ldots\pr^{\alpha_n}_{\ol z_n}=
\frac{\pr^{\abs{\alpha}}}{\pr\ol z^\alpha}\,.
\end{split}
\]
For $j, s\in\mathbb Z$, set $\delta_{j,s}=1$ if $j=s$, $\delta_{j,s}=0$ if $j\neq s$.
Let $X$ be a $\cC^\infty$ orientable paracompact manifold.
We let $TX$ and $T^*X$ denote the tangent bundle of $X$ and the cotangent bundle of $X$ respectively.
The complexified tangent bundle of $X$ and the complexified cotangent bundle of $X$
will be denoted by $\C TX$ and $\C T^*X$ respectively. We write $\langle\,\cdot\,,\cdot\,\rangle$
to denote the pointwise duality between $TX$ and $T^*X$.
We extend $\langle\,\cdot\,,\cdot\,\rangle$ bilinearly to $\C TX\times\C T^*X$.

Let $E$ be a $\cC^\infty$ vector bundle over $X$. The fiber of $E$ at $x\in X$ will be denoted by $E_x$.
Let $F$ be another vector bundle over $X$. We write
$F\boxtimes E^*$ to denote the vector bundle over $X\times X$ with fiber over $(x, y)\in X\times X$
consisting of the linear maps from $E_y$ to $F_x$.

The spaces of
smooth sections of $E$ over $X$ and distribution sections of $E$ over $X$ will be denoted by $\cC^\infty(X, E)$ and $\mathscr D'(X, E)$ respectively.
Let $\mathscr E'(X, E)$ be the subspace of $\mathscr D'(X, E)$ whose elements have compact support in $X$. 
Put $\cC^\infty_c(X,E):=\cC^\infty(X,E)\bigcap\mathscr E'(X, E)$.
For $m\in\R$, we let $W^m(X,E)$ denote the Sobolev space of order $m$ of sections of $E$ and let $W_{\rm comp}^m(X, E)$ denote the Sobolev space
of order $m$ of sections of $E$ with compact support of $X$. Let $Y\Subset X$ be a relatively compact open set of $X$. Put
\begin{gather*}
W^m_{\rm loc\,}(Y, E)=\big\{u\in\mathscr D'(Y, E):\, \varphi u\in W_{\rm comp}^m(Y, E),
      \,\forall\varphi\in\cC^\infty_c(Y)\big\}\,,\\
      W^m(\overline Y, E)=W^m_{\rm comp}(X, E)|_{Y}\,.
\end{gather*}
Let $W_{1}, W_{2}$ be bounded open subsets of $\R^{n_{1}}$ and $\R^{n_{2}}$, respectively. Let $E$ and $F$ be complex or real vector bundles over $W_{1}$ and $W_{2}$, respectively. Let $s_{1}, s_{2}\in\R$ and $n_{0}\in\mathbb Z.$ For a $k$-dependent continuous function $F_{k}: W^{s_{1}}_{\rm comp}(W_{1}, E)\rightarrow W^{s_{2}}_{\rm loc}(W_{2}, F)$, we write 
\[F_k=O(k^{n_0}):  W^{s_{1}}_{\rm comp}(W_{1}, E)\rightarrow W^{s_{2}}_{\rm loc}(W_{2}, F)\]
if for any $\chi_j\in\cC_{c}^{\infty}(W_{j})$, $j=1, 2$, $\Td\chi_1\in\cC^\infty_c(W_1)$, $\Td\chi_1\equiv1$ on $\supp\chi_1$, there is a positive constant $c>0$ independent of $k$ such that 
$$\|\chi_2F_{k}\chi_1u\|_{s_{2}}\leq ck^{n_{0}}\|\Td\chi_1u\|_{s_{1}}, \forall u\in W^{s_{1}}_{\rm loc}(W_{1}, E).$$
We write $$F_{k}=O(k^{-\infty}): W^{s_1}_{\rm comp}(W_{1}, E)\rightarrow W^{s_{2}}_{{\rm loc}}(W_{2}, F)$$ if $F_{k}=O(k^{-N}): W^{s_{1}}_{\rm comp}(W_{1}, E)\rightarrow W^{s_{2}}_{\rm loc}(W_{2}, F)$ for any $N\in\mathbb N_{0}$. Let $G_k: \cC^\infty_c(W_1,E)\To\mathscr D'(W_2, F)$ be a $k$-dependent continuous operator. We write 
$G_k=O(k^{-\infty})$ if 
\[G_{k}=O(k^{-\infty}): W^{s_1}_{\rm comp}(W_{1}, E)\rightarrow W^{s_{2}}_{{\rm loc}}(W_{2}, F)\]
for every $s_1, s_2\in\mathbb\R$.  
We recall the definition of the semi-classical symbol spaces. 

\begin{definition} \label{d-gue140826}
Let $W$ be an open set in $\R^N$. 
We say that $f: W\To[0,+\infty[$ is an order function if there are constants $C_0>0$, $N_0>0$, 
such that $f(x)\leq C_0\abs{x-y}^{N_0}f(y)$, for every $x, y\in W$. We let $S(f;W)$ be the set of $a\in\cC^\infty(W)$ such that for every $\alpha\in\mathbb N^N_0$, there exists $C_\alpha>0$ such that 
\[\abs{\pr^\alpha a(x)}\leq C_\alpha f(x),\] for every $x\in W$.
For $m\in\R$, let
\[S^m(f)=S^m(f;W):=\Big\{(a(\cdot,k))_{k\in\R}\,|\,(k^{-m}a(\cdot,k))\in S^0(f;W)\Big\}.
\]
Hence $a(\cdot,k)\in S^m(f;W)$ if for every $\alpha\in\mathbb N^N_0$, there
exists $C_\alpha>0$ independent of $k$, such that $\abs{\pr^\alpha_xa(x,k))}\leq C_\alpha k^{m}f(x)$ on $W$.

Consider a sequence $a_j\in S^{m_j}(f)$, $j\in\N_0$, where $m_j\searrow-\infty$,
and let $a\in S^{m_0}(f)$. We say that
\[
a(\cdot,k)\sim
\sum\limits^\infty_{j=0}a_j(\cdot,k)\:\:\text{in $S^{m_0}(f)$},
\]
if for every
$\ell\in\N_0$ we have $a-\sum^{\ell}_{j=0}a_j\in S^{m_{\ell+1}}(f)$. 
For a given sequence $a_j$ as above, we can always find such an asymptotic sum
$a$, which is unique up to an element in $S^{-\infty}(f)=S^{-\infty}(f;W):=\cap _mS^m(f)$.

We say that $a(\cdot,k)\in S^{m}(f)$ is a classical symbol on $W$ of order $m$ if
\begin{equation} \label{e-gue13628I}
a(\cdot,k)\sim\sum\limits^\infty_{j=0}k^{m-j}a_j\: \text{in $S^{m}(f)$},\ \ a_j(x)\in
S^0(f),\ j=0,1\ldots.
\end{equation}
The set of all classical symbols on $W$ of order $m$ is denoted by
$S^{m}_{{\rm cl\,}}(f)=S^{m}_{{\rm cl\,}}(f;W)$. 

Similarly, we define $S^m(f;Y,E)$, $S^{m}_{{\rm cl\,}}(f;Y,E)$ in the standard way, where $Y$ is a smooth manifold and $E$ is a vector bundle over $Y$ and $f$ is an order function on $Y$. 
\end{definition}
\begin{definition} \label{d-gue13628I}
Let $W$ be an open set in $\R^N$ and let $E$ be a vector bundle over $W$. A semi-classical pseudodifferential
operator on $W$ of order $m$ from sections of $E$ to sections of $E$ with classical symbol is a $k$-dependent continuous operator
$A_k: \cC^\infty_c(W,E)\To\cC^\infty(W,E)$ such that the distribution kernel $A_k(x,y)$ of $A_k$
is given by the oscillatory integral
\begin{equation}\label{psk-def}
\begin{split}
A_k(x,y)=&\frac{k^N}{(2\pi)^N}\int e^{ik\langle x-y,\xi\rangle}a(x,y,\xi,k)d\xi
+O(k^{-\infty}),\\
&a(x,y,\xi,k)\in S^m_{{\rm cl\,}}((1+\abs{\xi})^\ell;W\times W\times\R^N,E\boxtimes E^*),\end{split}
\end{equation}
where $\ell\in\mathbb Z$. 
We shall identify $A_k$ with $A_k(x,y)$. 
We write $A={\rm Op\,}_k(a)$. We let $\sigma^0_a$ to denote the leading term of $a$. 
\end{definition}

\subsection{Complex manifolds with boundary}\label{s-gue220215yyd} 

Let $M$ be a relatively compact connected open subset with $\cC^\infty$ boundary $X$ of a
complex manifold $M'$ of complex dimension $n$. Let $T^{1,0}M'$ and 
$T^{0,1}M'$ be the holomorphic tangent bundle of $M'$ and the anti-holomorphic 
tangent bundle of $M'$ respectively. 
For $p, q\in\mathbb N$, let $T^{*p,q}M'$ 
be the vector bundle of $(p,q)$ forms on $M'$. The Hermitian metric 
$\langle\,\cdot\,|\,\cdot\,\rangle$ on $\C TM'$ induces by duality, a Hermitian metric 
$\langle\,\cdot\,|\,\cdot\,\rangle$ on 
$\oplus^{p,q=n}_{p,q=0}T^{*p,q}M'$. Put $T^{1,0}X:=\mathbb CTX\cap T^{1,0}M'$. 
Then, $(X,T^{1,0}X)$ is a CR manifold of real dimension $2n-1$. For $p, q=0,\ldots,n-1$, the bundle of $(p,q)$ forms of $X$ is given by 
$T^{*p,q}X:=\Bigr(\Lambda^p(\mathbb CT^*X)\bigcap T^{*p,0}M\Bigr)\wedge\Bigr(\Lambda^q(\mathbb CT^*X)\bigcap T^{*0,q}M'\Bigr)$. 

Let $E$ be a $\cC^\infty$ vector bundle over $M'$. 
Let $U$ be an open set in $M'$. Let 
\[
\begin{split}
&\cC^\infty(U\cap \ol M,E),\ \ \mathscr D'(U\cap \ol M,E),\ \ \cC^\infty_c(U\cap \ol M,E),\ \ 
\mathscr E'(U\cap \ol M,E),\\ 
W^s&(U\cap \ol M,E),\ \ W^s_{{\rm comp\,}}(U\cap \ol M,E),\ \ 
W^s_{{\rm loc\,}}(U\cap \ol M,E),
\end{split}
\]
(where $\ s\in\R$)
denote the spaces of restrictions to $U\cap\ol M$ of elements in 
\[
\begin{split}
\cC^\infty&(U,E),\ \ \mathscr D'(U,E),\ \ \cC^\infty_c(U,E),\ \ 
\mathscr E'(U,E),\\  
&W^s(U,E),\ \  W^s_{{\rm comp\,}}(U,E),\ \  
W^s_{{\rm loc\,}}(U,E),
\end{split}
\] 
respectively. Write 
\[
\begin{split}
L^2(U\cap\ol M,E):=&W^0(U\cap \ol M,E),\ \ 
L^2_{{\rm comp\,}}(U\cap\ol M,E):=W^0_{{\rm comp\,}}(U\cap \ol M,E),\\ 
&L^2_{{\rm loc\,}}(U\cap\ol M,E):=W^0_{{\rm loc\,}}(U\cap \ol M,E).
\end{split}
\] 

For $s\in\R$, let $\norm{\cdot}_{s,M'}$ and $\norm{\cdot}_{s,X}$ denote the Sobolev norm of $W^s_{{\rm comp\,}}(U,E)$ and $W^s_{{\rm comp\,}}(U\cap X,E)$ respectively.
For $s\in\R$  and let $u\in W^s_{{\rm comp\,}}(U\cap \ol M,E)$. We let 
\begin{equation}\label{e-gue220301yyd}
\norm{u}_{s,\ol M}:=\inf\{\norm{\Td u}_{s,M'};\, \Td u\in W^s_{{\rm comp\,}}(U,E), \Td u|_M=u\}. 
\end{equation}

For every $p, q=1,\ldots,n$, we denote 
\[
\begin{split}
\Omega^{p,q}&(U\cap\ol M):=\cC^\infty(U\cap\ol M,T^{*p,q}M'),\ \ 
\Omega^{p,q}(M'):=\cC^\infty(M',T^{*p,q}M'),\\ 
\Omega^{p,q}_c&(U\cap\ol M):=\cC^\infty_c(U\cap\ol M,T^{*p,q}M'),\\ 
&\Omega^{p,q}_c(M'):=\cC^\infty_c(M',T^{*p,q}M'),\ \ 
\Omega^{p,q}_c(M):=\cC^\infty_c(M,T^{*p,q}M'),\\
\Omega^{p,q}&(U\cap\ol M,E):=\cC^\infty(U\cap\ol M,E\otimes T^{*p,q}M'),\ \ 
\Omega^{p,q}(M',E):=\cC^\infty(M',E\otimes T^{*p,q}M'),\\ 
\Omega^{p,q}_c&(U\cap\ol M,E):=\cC^\infty_c(U\cap\ol M,E\otimes T^{*p,q}M'),\\
&\Omega^{p,q}_c(M',E):=\cC^\infty_c(M',E\otimes T^{*p,q}M'),\ \ 
\Omega^{p,q}_c(M,E):=\cC^\infty_c(M,E\otimes T^{*p,q}M').
\end{split}
\]
Similarly, let $D$ be an open set of $X$. Let $\Omega^{p,q}(D):=\cC^\infty(D,T^{*p,q}X)$, $\Omega^{p,q}_c(D):=\cC^\infty_c(D,T^{*p,q}X)$, $\Omega^{p,q}(D,E):=\cC^\infty(D,T^{*p,q}X\otimes E)$, $\Omega^{p,q}_c(D,E):=\cC^\infty_c(D,T^{*p,q}X\otimes E)$.

Let $A$ and $B$ be $\cC^\infty$ vector bundles over $M'$. 
Let $U$ be an open set in $M'$. Let 
$$F_1, F_2: \cC^\infty_c(U\cap M,A)\To\mathscr D'(U\cap M,B)$$ 
be continuous operators. Let 
$F_1(x,y), F_2(x,y)\in\mathscr D'((U\times U)\cap(M\times M), A\boxtimes B^*)$ 
be the distribution kernels of $F_1$ and $F_2$ respectively. 
We write 
$$F_1\equiv F_2\!\!\mod\cC^\infty((U\times U)\cap(\ol M\times\ol M))$$ 
or $F_1(x,y)\equiv F_2(x,y)\!\!\mod\cC^\infty((U\times U)\cap(\ol M\times\ol M))$ 
if $F_1(x,y)=F_2(x,y)+r(x,y)$, where 
$r(x,y)\in\cC^\infty((U\times U)\cap(\ol M\times\ol M),A\boxtimes B^*)$. Similarly, 
let $\hat F_1, \hat F_2: \cC^\infty_c(U\cap M,A)\To\mathscr D'(U\cap X,B)$ be continuous operators. Let 
$\hat F_1(x,y), \hat F_2(x,y)\in\mathscr D'((U\times U)\cap(X\times M), A\boxtimes B^*)$ be the distribution kernels of $\hat F_1$ and $\hat F_2$ respectively. We write 
$\hat F_1\equiv\hat F_2\!\!\mod\cC^\infty((U\times U)\cap(X\times\ol M))$ or $\hat F_1(x,y)\equiv\hat F_2(x,y)\!\!\mod\cC^\infty((U\times U)\cap(X\times\ol M))$ if $\hat F_1(x,y)=\hat F_2(x,y)+\hat r(x,y)$, where 
$\hat r(x,y)\in\cC^\infty((U\times U)\cap(X\times\ol M),A\boxtimes B^*)$. Similarly, let $\tilde F_1, \tilde F_2: \cC^\infty_c(U\cap X,A)\To\mathscr D'(U\cap M,B)$ be continuous operators. Let 
\[\tilde F_1(x,y), \tilde F_2(x,y)\in\mathscr D'((U\times U)\cap(M\times X), A\boxtimes B^*)\] 
be the distribution kernels of $\tilde F_1$ and $\tilde F_2$ respectively. We write 
$\tilde F_1\equiv\tilde F_2\!\!\mod\cC^\infty((U\times U)\cap(\ol M\times X))$ or $\tilde F_1(x,y)\equiv\tilde F_2(x,y)\!\!\mod\cC^\infty((U\times U)\cap(\ol M\times X))$ if $\tilde F_1(x,y)=\tilde F_2(x,y)+\tilde r(x,y)$, where 
$\tilde r(x,y)\in\cC^\infty((U\times U)\cap(\ol M\times X),A\boxtimes B^*)$. 

Let $F_k, G_k: \cC^\infty_c(U\cap M,A)\To\mathscr D'(U\cap M,B)$ be $k$-dependent continuous operators. Let $F_k(x,y), G_k(x,y)\in\mathscr D'((U\times U)\cap(M\times M), A\boxtimes B^*)$ be the distribution kernels of $F_k$ and $G_k$ respectively. We write 
\begin{equation}\label{e-gue190813yyd}
F_k\equiv G_k\!\!\mod O(k^{-\infty})\ \ \mbox{ on $(U\times U)\cap(\ol M\times\ol M)$} 
\end{equation}
or 
\begin{equation}\label{e-gue190813yydz}
F_k(x,y)\equiv G_k(x,y)\!\!\mod O(k^{-\infty})\ \ \mbox{ on $(U\times U)\cap(\ol M\times\ol M)$} 
\end{equation}
if there is a $r_k(x,y)\in \cC^\infty(U\times U, A\boxtimes B^*)$ with $r_k(x,y)=O(k^{-\infty})$ on $U\times U$ such that  $r_k(x,y)|_{(U\times U)\cap(\ol M\times\ol M)}=F_k(x,y)-G_k(x,y)$, for $k\gg1$.  Similarly, 
let $\hat F_k, \hat G_k: \cC^\infty_c(U\cap M,A)\To\mathscr D'(U\cap X,B)$ be $k$-dependent continuous operators. Let 
$\hat F_k(x,y), \hat G_k(x,y)\in\mathscr D'((U\times U)\cap(X\times M), A\boxtimes B^*)$ be the distribution kernels of $\hat F_k$ and $\hat G_k$ respectively. We write 
$\hat F_k\equiv\hat G_k\mod O(k^{-\infty})$ on $(U\times U)\cap(X\times\ol M)$ or 
$\hat F_k(x,y)\equiv\hat G_k(x,y)\mod O(k^{-\infty})$ on $(U\times U)\cap(X\times\ol M)$
if there is a $\hat r_k(x,y)\in \cC^\infty((U\times U)\cap(X\times\ol M), A\boxtimes B^*)$ with $\hat r_k(x,y)=O(k^{-\infty})$ on $(U\cap X)\times U$ such that 
$\hat F_k(x,y)-\hat G_k(x,y)=\hat r_k(x,y)|_{(U\times U)\cap(X\times\ol M)}$, for $k\gg1$. 
 Similarly, let $\tilde F_k, \tilde G_k: \cC^\infty_c(U\cap X,A)\To\mathscr D'(U\cap M,B)$ be $k$-dependent continuous operators. Let 
\[\tilde F_k(x,y), \tilde G_k(x,y)\in\mathscr D'((U\times U)\cap(M\times X), A\boxtimes B^*)\] 
be the distribution kernels of $\tilde F_k$ and $\tilde G_k$ respectively. We write 
$\tilde F_k\equiv\tilde G_k\mod O(k^{-\infty})$ on $(U\times U)\cap(\ol M\times X)$ or $\tilde F_k(x,y)\equiv\tilde G_k(x,y)\mod O(k^{-\infty})$ on $(U\times U)\cap(\ol M\times X)$ 
if there is a $\tilde r_k(x,y)\in \cC^\infty((U\times U)\cap(\ol M\times X), A\boxtimes B^*)$ with $\tilde r_k(x,y)=O(k^{-\infty})$ on $U\times(U\cap X)$ such that 
$\tilde F_k(x,y)-\tilde G_k(x,y)=\tilde r_k(x,y)|_{(U\times U)\cap(\ol M\times X)}$, for $k\gg1$. 

Let $s_{1}, s_{2}\in\R$ and $n_{0}\in\mathbb Z$. For $k$-dependent continuous functions 
\[\begin{split}
&F_{k}: W^{s_{1}}_{\rm comp}(U\cap\ol M, A)\rightarrow W^{s_{2}}_{\rm loc}(U\cap\ol M,B),\\
&\hat F_{k}: W^{s_{1}}_{\rm comp}(U\cap X, A)\rightarrow W^{s_{2}}_{\rm loc}(U\cap\ol M,B),\\
&\tilde F_{k}: W^{s_{1}}_{\rm comp}(U\cap\ol M, A)\rightarrow W^{s_{2}}_{\rm loc}(U\cap X,B),
\end{split}\]
 we write 
\begin{equation}\label{e-gue220301yyda}
\begin{split}
&F_k=O(k^{n_0}):  W^{s_{1}}_{\rm comp}(U\cap\ol M, A)\rightarrow W^{s_{2}}_{\rm loc}(U\cap\ol M, B),\\
&\hat F_k=O(k^{n_0}):  W^{s_{1}}_{\rm comp}(U\cap X, A)\rightarrow W^{s_{2}}_{\rm loc}(U\cap\ol M, B),\\
&\tilde F_k=O(k^{n_0}):  W^{s_{1}}_{\rm comp}(U\cap\ol M, A)\rightarrow W^{s_{2}}_{\rm loc}(U\cap X, B),
\end{split}
\end{equation}
if for any $\chi_j\in\cC_{c}^{\infty}(U\cap\ol M)$, $j=1, 2$, $\Td\chi_1\in\cC^\infty_c(U\cap\ol M)$, $\Td\chi_1\equiv1$ on $\supp\chi_1$, $\tau_1, \tilde\tau_1\in\cC^\infty_c(U\cap X)$, $\tilde\tau_1\equiv1$ on $\supp\tau_1$, there is a positive constant $c>0$ independent of $k$ such that
\[\begin{split}
&\|\chi_2F_{k}\chi_1u\|_{s_{2},\ol M}\leq ck^{n_{0}}\|\Td\chi_1u\|_{s_{1},\ol M}, \forall u\in W^{s_{1}}_{\rm loc}(U\cap\ol M, A),\\
&\|\chi_2\hat F_{k}\tau_1u\|_{s_{2},\ol M}\leq ck^{n_{0}}\|\Td\tau_1u\|_{s_{1},X}, \forall u\in W^{s_{1}}_{\rm loc}(U\cap X, A),\\
&\|\tau_1\tilde F_{k}\chi_1u\|_{s_{2},X}\leq ck^{n_{0}}\|\Td\chi_1u\|_{s_{1},\ol M}, \forall u\in W^{s_{1}}_{\rm loc}(U\cap\ol M, A).
\end{split}\]

\subsection{Complex manifolds with $\R$-action} \label{s-gue150808}

Let $M$ be a relatively compact connected open subset with
smooth connected boundary $X$ of a
complex manifold $M'$ of complex dimension $n$.  
We will use the same notations as in Section~\ref{s-gue220215yyd}. 
From now on, we 
assume that $M'$ admits a $\R$-action $\eta:\R\times M'\to M'$.
$(\eta,x)\mapsto\eta\circ x$.  
Let $T\in\cC^\infty(M', TM')$ be the infinitesimal generator of
the $\R$-action which  is given by
\begin{equation}\label{e-gue150808}
(Tu)(x)=\frac{\partial}{\partial \eta}
\left(u(\eta\circ x)\right)\big|_{\eta=0},\ \ u\in\cC^\infty(M').
\end{equation}
A smooth function $u$ defined on an open set of 
$M'$ is called $\R$-invariant if $Tu=0$.

\begin{definition}\label{d-gue220512yyd}
We say that the $\R$-action $\eta$ is holomorphic if
$$[T,\cC^\infty(M', T^{1,0}M')]\subset\cC^\infty(M', T^{1,0}M').$$ 
\end{definition}
Recall that, in this work, we work with Assumption~\ref{a-gue170929I}. 
\begin{definition} \label{d-gue150508dII}
Let $M'$ be a complex manifold endowed with a holomorphic $\R$-action,
Let $F$ be complex (resp.\ holomorphic) vector bundle on $M'$.
If the $\R$-action lifts to a smooth (resp.\ holomorphic) action on $F$,
we say that the bundle $F$ is $\R$-equivariant. 
A similar definition applies for CR bundles over the CR manifold $X$
(cf.\ \cite[Definition 2.6]{HLM}).
\end{definition} 

If $F$ is an $\R$-equivariant complex (resp.\ holomorphic) vector bundle
of rank $r$, then there exists an open cover $(U_j)_j$ of $M'$ 
and frames $\{f^1_j,f^2_j,\dots,f^r_j\}$ on each $U_j$
such that the corresponding transition matrices are $\R$-invariant
(resp.\ holomorphic and $\R$-invariant).
For a section $s\in\cC^\infty(M',F)$,
by using the fixed frames $\{f^1_j,f^2_j,\dots,f^r_j\}$, 
we can define $Ts$ in the standard way. We say that $s$ is 
$\R$-equivariant if $Ts=0$. 


If $F$ is an $\R$-equivariant CR vector bundle over $X$, then there
exists an open cover and CR frames such that the transition matrices
are CR and $\R$-invariant \cite[Proposition 2.7]{HLM} and we can define
as above the notion of $\R$-equivariant CR section.

\begin{proposition}\label{2023-8-11-p1}
Let $(L, h^L)\rightarrow M'$ be an $\R$-equivariant holomorphic
line bundle and $s$ an $\R$-equivariant holomorphic frame with norm 
$|s|^2_{h^L}=e^{-2\phi}$. 
Then $JT(\phi)=-JT(\log|s|_{h^L})$ does not depend on the
choice of the frame $s$, so it defines
a global function on $M'$.
\end{proposition}
\begin{proof}
Let $s, \tilde s$ be $\mathbb R$-invariant holomorphic frames defined 
on an open set $U\subset M'$. Then $s=\tilde s g$ 
for some $\mathbb R$-invariant holomorphic function 
$g$ on $U$. The norm of $s$ and $\tilde s$ are
$e^{-2\phi}$ and $e^{-2\tilde\phi}$, respectively. It follows that 
$\tilde\phi-\phi=\frac12\log|g|^2.$
Since $T+iJT\in T^{0, 1}M'$ we have $(T+iJT)g=0$ and 
thus $JT(g)=0$. It follows that 
$$JT(\tilde\phi)-JT(\phi)=\frac12 JT(\log|g|^2)=0.$$
\end{proof}
On $X$, we will use the coordinates introduced by
Baouendi-Rothschild-Treves~\cite{BRT85}. 
\begin{definition}\label{D:BRT}
For each point $p\in X$
there exist a coordinate neighborhood $V$ with coordinates $(x_1,\ldots,x_{2n-1})$,
centered at $p=0$, and $\varepsilon>0$, $\varepsilon_0>0$,
such that, by setting $$z_j=x_{2j-1}+ix_{2j}, j=1,\ldots,n-1, x_{2n-1}=\eta$$
and $$D=\{(z, \eta)\in V: \abs{z}<\varepsilon, |\theta|<\varepsilon_0\},$$ we have
\begin{equation}\label{e-can1}
T=\frac{\partial}{\partial\eta}\:\:\text{on $U$},\\
\end{equation}
and the vector fields
\begin{equation}\label{e-can2}
Z_j=\frac{\partial}{\partial z_j}-i\frac{\partial\psi}{\partial z_j}(z)
\frac{\partial}{\partial\eta},\:\:j=1,\ldots,n-1,
\end{equation}
form a basis of $T_x^{1,0}X$ for each $x\in D$, where $\psi\in\cC^\infty(D,\R)$
is independent of $\eta$.
We call $(x_1,\ldots,x_{2n-1})$ canonical coordinates, 
$D$ canonical coordinate patch and
$(D,(z,\eta),\psi)$ a BRT trivialization. 
The frames \eqref{e-can2} are called BRT frames. 
We can also define BRT frames on the bundle $T^{\ast 0, q}X$. 
The dual frames of $\{Z_j\}_{j=1}^{n-1}$ are $\{dz_j\}_{j=1}^{n-1}$.  
We sometime write $(D,x=(x_1,\ldots,x_{2n-1}))$ 
to denote canonical coordinates. 
We will always fix a BRT covering of $T^{*0,q}X$ with BRT frames. 
Then, $T^{*0,q}X$ is a $\R$-equivariant vector bundle over $X$.
\end{definition}
Let $L\rightarrow M'$ be an $\R$-equivariant holomorphic line bundle. Recall that
we work under Assumption~\ref{a-gue170929I}.
From now on, we fix a $\R$-invariant Hermitian fiber metric 
$h^L$ such that $R^L$ is positive on $M'$,
and we fix a Hermitian metric $\langle\,\cdot\,|\,\cdot\,\rangle$ on 
$\C TM'$ so that $T^{1,0}M'\perp T^{0,1}M'$ and 
$|T|=1$ on $X$.
We can always find such a $\R$-invariant Hermitian metric 
on $M'$. For example, the curvature of $L$ induces a $\R$-invariant Hermitian metric 
on $M'$. The Hermitian metric 
$\langle\,\cdot\,|\,\cdot\,\rangle$ on $\C TM'$ induces by duality, a Hermitian metric 
$\langle\,\cdot\,|\,\cdot\,\rangle$ on 
$\oplus^{p,q=n}_{p,q=0}T^{*p,q}M'$. Let $\abs{\cdot}$ 
be the corresponding pointwise norm with respect to $\langle\,\cdot\,|\,\cdot\,\rangle$. 
For $v\in T^{*0,q}M'$ and for any $p=0,1,2,\ldots,n$, let $(v\wedge)^*:T^{*0,q+p}M'\To T^{*0,p}M'$ be the adjoint of $v\wedge:T^{*0,p}M'\To T^{*0,p+q}M'$ with respect to $\langle\,\cdot\,|\,\cdot\,\rangle$. That is, $\langle\,v\wedge u\,|\,g\,\rangle=\langle\,u\,|\,(v\wedge)^*g\,\rangle$, $\forall u\in T^{*0,p}M'$, $g\in T^{*0,p+q}M'$. 
Let $\rho\in\cC^\infty(M',\R)$ be a defining function of $X$, 
that is, $\rho=0$ on $X$, $\rho<0$ on $M$ and $d\rho\neq0$ near $X$. 
From now on, we take a defining function $\rho$ so that $\rho$
is $\R$-invariant and $|d\rho|=1$ on $X$. 
The Hermitian metric $\langle\,\cdot\,|\,\cdot\,\rangle$ on 
$\mathbb CTM'$ induces a Hermitian metric 
$\langle\,\cdot\,|\,\cdot\,\rangle$ on $\mathbb CTX$.  


Let $dv_{M'}$ be the volume form on $M'$ induced by the Hermitian metric $\langle\,\cdot\,|\,\cdot\,\rangle$ on $\C TM'$ and 
and let $(\,\cdot\,|\,\cdot\,)_{\ol M}$ and $(\,\cdot\,|\,\cdot\,)_{M'}$ be the inner products on $\Omega^{0,q}(\ol M)$ and $\Omega^{0,q}_c(M')$
defined by
\begin{equation} \label{e-gue190312}
\begin{split}
&(\,f\,|\,h\,)_{\ol M}=\int_M\langle\,f\,|\,h\,\rangle dv_{M'},\ \ f, h\in\Omega^{0,q}(\ol M),\\
&(\,f\,|\,h\,)_{M'}=\int_{M'}\langle\,f\,|\,h\,\rangle dv_{M'},\ \ f, h\in\Omega^{0,q}_c(M').
\end{split}
\end{equation}  
Let $\norm{\cdot}_{\ol M}$ and $\norm{\cdot}_{M'}$ be the corresponding norms with respect to $(\,\cdot\,|\,\cdot\,)_{\ol M}$ and $(\,\cdot\,|\,\cdot\,)_{M'}$ respectively.  
Let $L^2_{(0,q)}(M)$ be the completion of $\Omega^{0,q}(\ol M)$ with respect to $(\,\cdot\,|\,\cdot\,)_{\ol M}$. We extend $(\,\cdot\,|\,\cdot\,)_{\ol M}$ to $L^2_{(0,q)}(M)$
in the standard way. We write $L^2(M):=L^2_{(0,0)}(M)$.

Let $dv_{X}$ be the volume form on $X$ induced by the Hermitian metric $\langle\,\cdot\,|\,\cdot\,\rangle$ on $\C TX$  
and let $(\,\cdot\,|\,\cdot\,)_{X}$ be the inner product on $\Omega^{0,q}(X)$ induced by $dv_{X}$ and $\langle\,\cdot\,|\,\cdot\,\rangle$. 
Let $\norm{\cdot}_{X}$  the corresponding norm with respect to $(\,\cdot\,|\,\cdot\,)_{X}$. Let $L^2_{(0,q)}(X)$ be the completion of $\Omega^{0,q}(X)$ with respect to 
$(\,\cdot\,|\,\cdot\,)_{X}$. We extend $(\,\cdot\,|\,\cdot\,)_{X}$ to $L^2_{(0,q)}(X)$
in the standard way. We write $L^2(X):=L^2_{(0,0)}(X)$.

For $k\in\mathbb N$, let $(L^k,h^{L^k})$ be the $k$-th power of $(L,h^L)$. Let $\langle\,\cdot\,|\,\cdot\,\rangle_{h^{L^k}}$ be the Hermitian metric of $L^k\otimes T^{*0,q}M'$ induced by $h^{L^k}$ and  
$\langle\,\cdot\,|\,\cdot\,\rangle$. 
Let $(\,\cdot\,|\,\cdot\,)_{\ol M,k}$ and $(\,\cdot\,|\,\cdot\,)_{M',k}$ be the inner products on $\Omega^{0,q}(\ol M,L^k)$ and $\Omega^{0,q}_c(M',L^k)$
defined by
\begin{equation} \label{e-gue190312z}
\begin{split}
&(\,f\,|\,h\,)_{\ol M,k}=\int_M\langle\,f\,|\,h\,\rangle_{h^{L^k}}dv_{M'},\ \ f, h\in\Omega^{0,q}(\ol M,L^k),\\
&(\,f\,|\,h\,)_{M',k}=\int_{M'}\langle\,f\,|\,h\,\rangle_{h^{L^k}}dv_{M'},\ \ f, h\in\Omega^{0,q}_c(M',L^k).
\end{split}
\end{equation}
Let $\norm{\cdot}_{\ol M,k}$ and $\norm{\cdot}_{M',k}$ be the corresponding norms with respect to $(\,\cdot\,|\,\cdot\,)_{\ol M,k}$ and $(\,\cdot\,|\,\cdot\,)_{M',k}$ respectively. 
Let $L^2_{(0,q)}(M,L^k)$ be the completion of $\Omega^{0,q}(\ol M,L^k)$ with respect to $(\,\cdot\,|\,\cdot\,)_{\ol M,k}$. We extend $(\,\cdot\,|\,\cdot\,)_{\ol M,k}$ to $L^2_{(0,q)}(M,L^k)$
in the standard way. We write $L^2(M,L^k):=L^2_{(0,0)}(M,L^k)$.  Similarly,  let $(\,\cdot\,|\,\cdot\,)_{X,k}$ be the inner product on $\Omega^{0,q}(X,L^k)$ induced by $dv_{X}$ and $\langle\,\cdot\,|\,\cdot\,\rangle_{h^{L^k}}$. 
Let $\norm{\cdot}_{X,k}$  the corresponding norm with respect to $(\,\cdot\,|\,\cdot\,)_{X,k}$. Let $L^2_{(0,q)}(X,L^k)$ be the completion of $\Omega^{0,q}(X,L^k)$ with respect to $(\,\cdot\,|\,\cdot\,)_{X,k}$. We extend $(\,\cdot\,|\,\cdot\,)_{X,k}$ to $L^2_{(0,q)}(X,L^k)$
in the standard way. We write $L^2(X,L^k):=L^2_{(0,0)}(X,L^k)$. 

From~\cite[Theorem 3.5, Corollary 3.9]{HHL17}, we see that the 
$\R$-action $\eta$ comes from a torus action 
$\T^d=(e^{i\theta_1},\ldots,e^{i\theta_d})$ on $X$. 
In view of~\cite[Theorem 3.12, Lemma 3.14]{HHL17}, 
we see that $L$, $h^L$, $R^L$ and $\rho$ are torus invariant 
and hence the Hermitian metric $\langle\,\cdot\,|\,\cdot\,\rangle$ 
and the $L^2$ inner products $(\,\cdot\,|\,\cdot\,)_{\ol M,k}$, 
$(\,\cdot\,|\,\cdot\,)_{X,k}$ are torus invariant. 
For every $j=1,\ldots,d$, let $T_j$ be the operator on 
$\Omega^{0,q}(M',L^k)$, resp.\ $\cC^\infty(X,T^{*0,q}M'\otimes L^k)$, 
given by 
\[
(T_ju)(x)=\frac{\partial}{\partial\theta_j}
u((1,\ldots,1,e^{i\theta_j},1,\ldots,1)\circ x)\big|_{\theta_j=0},
\]
for all $u\in\Omega^{0,q}(M',L^k)$, resp.\ 
$u\in\cC^\infty(X,T^{*0,q}M'\otimes L^k)$.
Since $L$ is torus invariant, we can also define $T_ju$ in the standard way, 
for every $u\in\Omega^{0,q}(M',L^k)$ 
($u\in\cC^\infty(X,T^{*0,q}M'\otimes L^k$), $j=1,\ldots,d$. 
Note that $T_j$ can be zero at some point of $X$. 
Since the $\R$-action $\eta$ comes from $T^d$, 
there exist real numbers $\beta_j\in\R$, $j=1,\ldots,d$, such that 
\begin{equation}\label{e-gue220330ycd}
T=\beta_1T_1+\ldots+\beta_dT_d. 
\end{equation}

Let $\ddbar_k: \Omega^{0,q}(M',L^k)\To\Omega^{0,q+1}(M',L^k)$
be the part of the exterior differential operator which maps forms of type $(0,q)$ to forms of
type $(0,q+1)$ and we denote by
$\ol{\pr}^*_{f,k}: \Omega^{0,q+1}(M',L^k)\To\Omega^{0,q}(M',L^k)$
the formal adjoint of $\ddbar_k$. That is
\[(\,\ddbar_kf\,|\,h\,)_{M',k}=(f\,|\,\ol{\pr}^*_{f,k}h\,)_{M',k},\]
$f\in\Omega^{0,q}_c(M',L^k)$, $h\in\Omega^{0,q+1}(M',L^k)$.
We shall also use the notation $\ddbar_k$ for the closure in $L^2$ of the $\ddbar_k$ operator, initially defined on
$\Omega^{0,q}(\ol M,L^k)$ and $\ddbar^*_k$  for the Hilbert space adjoint 
of $\ddbar_k$. Recall that for $u\in L^2_{(0,q}(M)$, we say that 
$u\in\Dom\ddbar_k$ if we can find a sequence 
$u_j\in\Omega^{0,q}(\ol M,L^k)$, $j=1,2,\ldots$\,, 
with $\lim_{j\To\infty}\norm{u_j-u}_{\ol M,k}=0$ such that 
$\lim_{j\To\infty}\norm{\ddbar_ku_j-v}_{\ol M,k}=0$, for some 
$v\in L^2_{(0,q+1)}(M,L^k)$.  
 
Let $h$ be the two form induced by the Hermitian metric $\langle\,\cdot\,|\,\cdot\,\rangle$. Let $p\in X$.  We can take coordinates $x=(x_1, \ldots, x_{2n})$ defined on an open set  $U$ of $p$ in $M'$ such that $(x_1, \ldots, x_{2n-1})$ are the BRT coordinates on $D:=U\cap X$, $x_{2n}=\rho+O(\abs{\rho}^2)$, $x(p)=0$ and 
\begin{equation}\label{e-gue220301ycd}
\begin{split}
&h=\sum_{j, k=1}^{2n}h_{jk}dx_jdx_k,\ \ h_{jk}=h_{kj},\ \ j,k=1,\ldots,2n,\\
&h_{2n, 2n}=1,\ \ h_{2n, j}|_X=0, \ \ j=1,\ldots,2n-1. 
\end{split}
  \end{equation}
For a given point $z\in M'$, let $A$ be the anti-linear 
map $$A: \mathbb CT_{z}M'\rightarrow \mathbb CT^{\ast}_{z}M',
\quad Av(u)=\langle\,u\,|\,v\,\rangle,\ \  u, v\in \mathbb CTM'.$$
$A$ is an isomorphism for all $z\in M'$ and 
$A(T^{1, 0}M')=T^{\ast 1, 0}M'$ and $A(T^{0, 1}M')=T^{\ast 0, 1}M'$.
Let $(U, z=(z_{1}, \ldots, z_{n}))$ be any local holomorphic chart on $M'$.  
Write 
\begin{equation}\label{e-gue220302yyd}
\begin{split}
&\left\langle\,\frac{\pr}{\pr z_j}\,|\,\frac{\pr}{\pr z_s}\,\right\rangle
=g_{js},\ \ j, s=1,\ldots,n,\\
&g:=\left(g_{js}\right)^n_{j,s=1},\ \ (g^{-1})^t=\left(g^{js}\right)^n_{j,s=1},
\end{split}
\end{equation}
where $(g^{-1})^t$ denotes the transpose of $g^{-1}$. 
Let $A^{-1}$ be the inverse of $A$. Then by definition, 
we can check that for every $j=1,\ldots,n$, 
$$A\left(\frac{\partial}{\partial z_{j}}\right)=
\sum_{j=1}^{n}g_{sj}dz_{s}, ~~A^{-1}(dz_{j})=
\sum_{k=1}^{n}g^{sj}\frac{\pr}{\pr z_s}.$$
 Define \begin{equation}
 \frac{\partial}{\partial\rho}:=A^{-1}(d\rho).
 \end{equation}
 It is easy to check that $\frac{\partial}{\partial\rho} (\rho)|_{X}\equiv1$. 
 It is easy to see that on $X$ we have $T=J(\frac{\pr}{\pr\rho})$
 or $T=-J(\frac{\pr}{\pr\rho})$, where $J$ 
 is the complex structure map on $TM'$. Without loss of generality, 
 from now on, we assume that 
 \begin{equation}\label{e-gue220311yyd}
 \mbox{$T=J(\frac{\pr}{\pr\rho})$ on $X$}. 
 \end{equation}
 
 Let $U_1$, $U_2$ be open sets of $M'$ with $U_1\cap X\neq\emptyset$, $U_2\cap X\neq\emptyset$. Let $s_1$, $s_2$ be local $\R$-equivariant holomorphic frames of $L$ over $U_1$ and $U_2$ respectively. 
 For $q\geq 0$, let 
 \begin{equation}\label{e-gue220302yyds}
 \begin{split}
 &A_{k}: \cC^\infty_c(U_1\cap M, T^{\ast 0, q}M'\otimes L^{k})\rightarrow\mathscr D'(U_2\cap M, T^{\ast 0, q}M'\otimes L^{k}),\\
 &\hat A_{k}: \cC^\infty_c(U_1\cap X, T^{\ast 0, q}M'\otimes L^{k})\rightarrow\mathscr D'(U_2\cap M, T^{\ast 0, q}M'\otimes L^{k}),\\
  &\tilde A_{k}: \cC^\infty_c(U_1\cap M, T^{\ast 0, q}M'\otimes L^{k})\rightarrow\mathscr D'(U_2\cap X, T^{\ast 0, q}M'\otimes L^{k}),\\
  &A^\dagger_{k}: \cC^\infty_c(U_1\cap X, T^{\ast 0, q}M'\otimes L^{k})\rightarrow\mathscr D'(U_2\cap X, T^{\ast 0, q}M'\otimes L^{k}),
  \end{split}
  \end{equation}
 be continuous operators. 
 Let  $|s_j|_{h}^{2}=e^{-2\phi_j}$, $j=1, 2$. The localization of $A_{k}$, $\hat A_k$, $\tilde A_k$ and $A^\dagger_k$ with respect to $s_1$ and $s_2$ are given by
\begin{equation*}
\begin{split}
&A_{k, s_1,s_2}: \cC^\infty_c(U_1\cap M, T^{\ast 0, q}M')\rightarrow\mathscr D'(U_2\cap M, T^{\ast 0, q}M'),\\
&A_{k, s_1,s_2}(u):=e^{-k\phi_2}s^{-k}_2A_{k}(s^{k}_1e^{k\phi_1}u), \forall u\in\cC^\infty_c(U_1\cap M, T^{\ast 0, q}M'),
\end{split}
\end{equation*}
\begin{equation*}
\begin{split}
&\hat A_{k, s_1,s_2}: \cC^\infty_c(U_1\cap X, T^{\ast 0, q}M')\rightarrow\mathscr D'(U_2\cap\ol M, T^{\ast 0, q}M'),\\
&\hat A_{k, s_1,s_2}(u):=e^{-k\phi_2}s^{-k}_2\hat A_{k}(s^{k}_1e^{k\phi_1}u), \forall u\in\cC^\infty_c(U_1\cap X, T^{\ast 0, q}M'),
\end{split}
\end{equation*}
\begin{equation*}
\begin{split}
&\tilde A_{k, s_1,s_2}: \cC^\infty_c(U_1\cap M, T^{\ast 0, q}M')
\rightarrow\mathscr D'(U_2\cap X, T^{\ast 0, q}M'),\\
&\tilde A_{k, s_1,s_2}(u):=e^{-k\phi_2}s^{-k}_2
\tilde A_{k}(s^{k}_1e^{k\phi_1}u), \forall 
u\in\cC^\infty_c(U_1\cap M, T^{\ast 0, q}M'), \\
&A^\dagger_{k, s_1,s_2}: \cC^\infty_c(U_1\cap X, 
T^{\ast 0, q}M')\rightarrow\mathscr D'(U_2\cap X, T^{\ast 0, q}M'),\\
&A^\dagger_{k, s_1,s_2}(u):=
e^{-k\phi_2}s^{-k}_2A^\dagger_{k}(s^{k}_1e^{k\phi_1}u), 
\forall u\in\cC^\infty_c(U_1\cap X, T^{\ast 0, q}M').
\end{split}
\end{equation*}
When $U_1=U_2$ and $s=s_1=s_2$, we write $A_{k,s}:=
A_{k, s_1,s_2}$, $\hat A_{k,s}:=\hat A_{k,s_1,s_2}$, 
$\tilde A_{k,s}:=\tilde A_{k,s_1,s_2}$, 
$A^\dagger_{k,s}:=A^\dagger_{k,s_1,s_2}$.

Let $u\in W^{\ell}_{{\rm comp\,}}(U_1\cap\ol M, 
T^{\ast 0, q}M'\otimes L^{k})$, 
for some $\ell\in\R$. Write $u=s^k_1\otimes\tilde u$, where
$\tilde u\in W^{\ell}_{{\rm comp\,}}(U\cap\ol M, T^{\ast 0, q}M')$. 
Let $\hat u:=e^{-k\phi}\tilde u$. 
We define $\norm{u}_{\ell,\ol M,k,s_1}:=\norm{\hat u}_{\ell,\ol M}$. 
We will sometimes write $\norm{u}_{\ell,\ol M,k}$
to denote $\norm{u}_{\ell,\ol M,k,s_1}$. We will fix an open
trivialization covering $\set{U_j}^N_{j=1}$ of $\ol M$, for $N\in\mathbb N$, 
and for each $j$, we will fix an $\R$-equivariant
holomorphic frame $s_j$ on $U_j$.
For $u\in W^{\ell}(\ol M, T^{\ast 0, q}M'\otimes L^{k})$, 
we define $\norm{u}_{\ell,\ol M,k}$ by using partition of unity. 
Similarly, for $u\in W^{\ell}_{{\rm comp\,}}(U\cap X, 
T^{\ast 0, q}M'\otimes L^{k})$, we define 
$\norm{u}_{\ell,X,k,s_1}$ in the same way. 
We will also sometimes write $\norm{u}_{\ell,X,k}$ to denote 
$\norm{u}_{\ell,X,k,s_1}$. 
We will fix an open  trivialization covering $\set{D_j}^N_{j=1}$ of $X$, 
$N\in\mathbb N$, and for each $j$, we will fix a $\R$-equivariant
CR frame $s_j$ on $D_j$.
For $u\in W^{\ell}(X, T^{\ast 0, q}M'\otimes L^{k})$, 
we define $\norm{u}_{\ell,X,k}$ by using partition of unity. 

Let $\ell_{1}, \ell_{2}\in\R, n_{0}\in\mathbb Z\bigcup\{-\infty\}$. 
We write
\begin{equation*}
\begin{split}
&A_{k}=O(k^{n_{0}}): W^{\ell_{1}}_{{\rm comp\,}}
(U_1\cap\ol M, T^{*0,q}M'\otimes L^{k})\rightarrow 
W^{\ell_{2}}_{{\rm loc\,}}(U_2\cap \ol M, 
T^{\ast 0, q}M'\otimes L^{k}),\\
&\hat A_{k}=O(k^{n_{0}}): W^{\ell_{1}}_{{\rm comp\,}}
(U_1\cap X, T^{*0,q}M'\otimes L^{k})\rightarrow
W^{\ell_{2}}_{{\rm loc\,}}(U_2\cap \ol M, T^{\ast 0, q}M'\otimes L^{k}),\\
&\tilde A_{k}=O(k^{n_{0}}): W^{\ell_{1}}_{{\rm comp\,}}
(U_1\cap\ol M, T^{\ast 0, q}M'\otimes L^{k})\rightarrow 
W^{\ell_{2}}_{{\rm loc\,}}(U_2\cap X, T^{\ast 0, q}M'\otimes L^{k}),\\
&A^\dagger_{k}=O(k^{n_{0}}): W^{\ell_{1}}_{{\rm comp\,}}
(U_1\cap X, T^{\ast 0, q}M'\otimes L^{k})\rightarrow 
W^{\ell_{2}}_{{\rm loc\,}}(U_2\cap X, T^{\ast 0, q}M'\otimes L^{k}),
\end{split}
\end{equation*}
if $A_k$, $\hat A_k$, $\tilde A_k$ and $A^\dagger_k$ can be extended to continuous operators: 
\begin{equation*}
\begin{split}
&A_{k}: W^{\ell_{1}}_{{\rm comp\,}}(U_1\cap\ol M, T^{\ast 0, q}M'\otimes L^{k})\rightarrow W^{\ell_{2}}_{{\rm loc\,}}(U_2\cap \ol M, T^{\ast 0, q}M'\otimes L^{k}),\\
&\hat A_{k}: W^{\ell_{1}}_{{\rm comp\,}}(U_1\cap X, T^{\ast 0, q}M'\otimes L^{k})\rightarrow W^{\ell_{2}}_{{\rm loc\,}}(U_2\cap \ol M, T^{\ast 0, q}M'\otimes L^{k}),\\
&\tilde A_{k}: W^{\ell_{1}}_{{\rm comp\,}}(U\cap\ol M, T^{\ast 0, q}M'\otimes L^{k})\rightarrow W^{\ell_{2}}_{{\rm loc\,}}(U\cap X, T^{\ast 0, q}M'\otimes L^{k}),\\
&A^\dagger_{k}: W^{\ell_{1}}_{{\rm comp\,}}(U\cap X, T^{\ast 0, q}M'\otimes L^{k})\rightarrow W^{\ell_{2}}_{{\rm loc\,}}(U\cap X, T^{\ast 0, q}M'\otimes L^{k})
\end{split}
\end{equation*}
and 
\begin{equation*}
\begin{split}
&A_{k, s_1,s_2}=O(k^{n_{0}}): W^{\ell_{1}}_{\rm comp}(U_1\cap\ol M, T^{\ast 0, q}M')\rightarrow W^{\ell_{2}}_{\rm loc}(U_2\cap\ol M, T^{\ast 0, q}M'),\\
&\hat A_{k,s_1,s_2}=O(k^{n_{0}}): W^{\ell_{1}}_{{\rm comp\,}}(U_1\cap X, T^{\ast 0, q}M')\rightarrow W^{\ell_{2}}_{{\rm loc\,}}(U_2\cap \ol M, T^{\ast 0, q}M'),\\
&\tilde A_{k,s_1,s_2}=O(k^{n_{0}}): W^{\ell_{1}}_{{\rm comp\,}}(U_1\cap\ol M, T^{\ast 0, q}M')\rightarrow W^{\ell_{2}}_{{\rm loc\,}}(U_2\cap X, T^{\ast 0, q}M'),\\
&A^\dagger_{k,s_1,s_2}=O(k^{n_{0}}): W^{\ell_{1}}_{{\rm comp\,}}(U_1\cap X, T^{\ast 0, q}M')\rightarrow W^{\ell_{2}}_{{\rm loc\,}}(U_2\cap X, T^{\ast 0, q}M').
\end{split}
\end{equation*}
We write 
\begin{equation*}
\begin{split}
&\mbox{$A_{k}\equiv0\mod O(k^{-\infty})$ on $(U_1\times U_2)\cap(\ol M\times\ol M)$},\\
&\mbox{$\hat A_{k}\equiv0\mod O(k^{-\infty})$ on $(U_1\times U_2)\cap(X\times\ol M)$},\\
&\mbox{$\tilde A_{k}\equiv0\mod O(k^{-\infty})$ on $(U_1\times U_2)\cap(\ol M\times X)$},\\
&\mbox{$A^\dagger_{k}\equiv0\mod O(k^{-\infty})$ on $(U_1\times U_2)\cap(X\times X)$}
\end{split}
\end{equation*}
if 
\begin{equation*}
\begin{split}
&\mbox{$A_{k,s_1,s_2}\equiv0
\mod O(k^{-\infty})$ on $(U_1\times U_2)\cap(\ol M\times\ol M)$},\\
&\mbox{$\hat A_{k,s_1,s_2}\equiv0
\mod O(k^{-\infty})$ on $(U_1\times U_2)\cap(X\times\ol M)$},\\
&\mbox{$\tilde A_{k,s_1,s_2}\equiv0
\mod O(k^{-\infty})$ on $(U_1\times U_2)\cap(\ol M\times X)$},\\
&\mbox{$A^\dagger_{k,s_1,s_2}\equiv0
\mod O(k^{-\infty})$ on $(U_1\times U_2)\cap(X\times X)$}.
\end{split}
\end{equation*}
Let 
 \begin{equation}\label{e-gue220302yydsz}
 \begin{split}
 &B_{k}: \cC^\infty(M, T^{\ast 0, q}M'\otimes L^{k})\rightarrow\mathscr D'(M, T^{\ast 0, q}M'\otimes L^{k}),\\
 &\hat B_{k}: \cC^\infty(X, T^{\ast 0, q}M'\otimes L^{k})\rightarrow\mathscr D'(M, T^{\ast 0, q}M'\otimes L^{k}),\\
  &\tilde B_{k}: \cC^\infty(M, T^{\ast 0, q}M'\otimes L^{k})\rightarrow\mathscr D'(X, T^{\ast 0, q}M'\otimes L^{k}),\\
  &B^\dagger_{k}: \cC^\infty(X, T^{\ast 0, q}M'\otimes L^{k})\rightarrow\mathscr D'(X, T^{\ast 0, q}M'\otimes L^{k}),
  \end{split}
  \end{equation}
 be continuous operators. Let $\ell_{1}, \ell_{2}\in\R, n_{0}\in\mathbb Z\bigcup\{-\infty\}$. We write 
 \begin{equation*}
\begin{split}
&B_{k}=O(k^{n_{0}}): W^{\ell_{1}}(\ol M, T^{\ast 0, q}M'\otimes L^{k})\rightarrow W^{\ell_{2}}(\ol M, T^{\ast 0, q}M'\otimes L^{k}),\\
&\hat B_{k}=O(k^{n_{0}}): W^{\ell_{1}}(X, T^{\ast 0, q}M'\otimes L^{k})\rightarrow W^{\ell_{2}}(\ol M, T^{\ast 0, q}M'\otimes L^{k}),\\
&\tilde B_{k}=O(k^{n_{0}}): W^{\ell_{1}}(\ol M, T^{\ast 0, q}M'\otimes L^{k})\rightarrow W^{\ell_{2}}(X, T^{\ast 0, q}M'\otimes L^{k}),\\
&B^\dagger_{k}=O(k^{n_{0}}): W^{\ell_{1}}(X, T^{\ast 0, q}M'\otimes L^{k})\rightarrow W^{\ell_{2}}(X, T^{\ast 0, q}M'\otimes L^{k}),
\end{split}
\end{equation*} 
if $B_k$, $\hat B_k$, $\tilde B_k$ and $B^\dagger_k$ can be extended to continuous operators: 
\begin{equation*}
\begin{split}
&B_{k}: W^{\ell_{1}}(\ol M, T^{\ast 0, q}M'\otimes L^{k})\rightarrow 
W^{\ell_{2}}(\ol M, T^{\ast 0, q}M'\otimes L^{k}),\\
&\hat B_{k}: W^{\ell_{1}}(X, T^{\ast 0, q}M'\otimes L^{k})\rightarrow 
W^{\ell_{2}}(\ol M, T^{\ast 0, q}M'\otimes L^{k}),\\
&\tilde B_{k}: W^{\ell_{1}}(\ol M, T^{\ast 0, q}M'\otimes L^{k})\rightarrow 
W^{\ell_{2}}(X, T^{\ast 0, q}M'\otimes L^{k}),\\
&B^\dagger_{k}: W^{\ell_{1}}(X, T^{\ast 0, q}M'\otimes L^{k})\rightarrow 
W^{\ell_{2}}(X, T^{\ast 0, q}M'\otimes L^{k})
\end{split}
\end{equation*}
 and for all holomorphic frames $s_1$, $s_2$ 
 defined on open sets $U_1$, $U_2$ of $M'$, respectively, we have
 \begin{equation*}
\begin{split}
&B_{k, s_1,s_2}=O(k^{n_{0}}): 
W^{\ell_{1}}_{\rm comp}(U_1\cap\ol M, 
T^{\ast 0, q}M')\rightarrow W^{\ell_{2}}_{\rm loc}
(U_2\cap\ol M, T^{\ast 0, q}M'),\\
&\hat B_{k,s_1,s_2}=O(k^{n_{0}}):
W^{\ell_{1}}_{{\rm comp\,}}(U_1\cap X, T^{\ast 0, q}M')
\rightarrow W^{\ell_{2}}_{{\rm loc\,}}(U_2\cap \ol M, T^{\ast 0, q}M'),\\
&\tilde B_{k,s_1,s_2}=O(k^{n_{0}}):
W^{\ell_{1}}_{{\rm comp\,}}(U_1\cap\ol M, T^{\ast 0, q}M')
\rightarrow W^{\ell_{2}}_{{\rm loc\,}}(U_2\cap X, T^{\ast 0, q}M'),\\
&B^\dagger_{k,s_1,s_2}=O(k^{n_{0}}):
W^{\ell_{1}}_{{\rm comp\,}}(U_1\cap X, T^{\ast 0, q}M')
\rightarrow W^{\ell_{2}}_{{\rm loc\,}}(U_2\cap X, T^{\ast 0, q}M').
\end{split}
\end{equation*}
 We write 
\begin{equation*}
\begin{split}
&B_{k}\equiv0\mod O(k^{-\infty})\:\: \text{on $\ol M\times\ol M$},\\
&\mbox{$\hat B_{k}\equiv0\mod O(k^{-\infty})$ on $X\times\ol M$},\\
&\mbox{$\tilde B_{k}\equiv0\mod O(k^{-\infty})$ on $\ol M\times X$},\\
&\mbox{$B^\dagger_{k}\equiv0\mod O(k^{-\infty})$ on $X\times X$}
\end{split}
\end{equation*}
 if for all holomorphic frames $s_1$, $s_2$ 
 defined on open sets $U_1$, $U_2$ of $M'$ respectively, we have
 \begin{equation*}
\begin{split}
&\mbox{$B_{k,s_1,s_2}\equiv0
\mod O(k^{-\infty})$ on $(U_1\times U_2)\cap(\ol M\times\ol M)$},\\
&\mbox{$\hat B_{k,s_1,s_2}\equiv0
\mod O(k^{-\infty})$ on $(U_1\times U_2)\cap(X\times\ol M)$},\\
&\mbox{$\tilde B_{k,s_1,s_2}\equiv0
\mod O(k^{-\infty})$ on $(U_1\times U_2)\cap(\ol M\times X)$},\\
&\mbox{$B^\dagger_{k,s_1,s_2}\equiv0
\mod O(k^{-\infty})$ on $(U_1\times U_2)\cap(X\times X)$}.
\end{split}
\end{equation*} 
Let $k\in\R$. Let $U$ be an open set in $M'$ and let $E$ be a vector bundle over $M'\times M'$. Let $f: U\To[0,+\infty[$ be an order function. Let
\begin{equation}\label{e-gue190813yydI}
S^k(f;(U\times U)\cap(\ol M\times\ol M),E)
\end{equation}
denote the space of restrictions to $(U\times U)\cap(\ol M\times\ol M)$ of elements in 
$S^k(f;U\times U,E)$.
Let
\[a_j\in S^{k_j}(f;(U\times U)\cap(\ol M\times\ol M),E),\ \ j=0,1,2,\dots,\] 
with $k_j\searrow -\infty$, $j\To \infty$.
Then there exists
$a\in S^{k_0}(f;(U\times U)\cap(\ol M\times\ol M),E)$
such that for every $\ell\in\N$,
\[a-\sum^{\ell-1}_{j=0}a_j\in S^{k_\ell}(f;(U\times U)\cap
(\ol M\times\ol M),E).\]
If $a$ and $a_j$ have the properties above, we write
\[a\sim\sum^\infty_{j=0}a_j \text{ in }
S^{k_0}(f;(U\times U)\cap(\ol M\times\ol M),E).\] 
We define $S^k_{{\rm cl\,}}(f;(U\times U)\cap(\ol M\times\ol M),E)$, $S^k(f;(U\times U)\cap(\ol M\times X),E)$, $S^k_{{\rm cl\,}}(f;(U\times U)\cap(\ol M\times X),E)$
in the similar way. If $E$ is trivial, we will omit $E$ in the notations. 

\section{Local expression of the Kodaira Laplacian near the boundary}

For $q=0,1,\ldots,n$, let 
\[\Box^{(q)}_{f,k}:=\ol{\pr}^*_{f,k}\,\ddbar_k+\ddbar_k\,\ol{\pr}^*_{f,k}: \Omega^{0,q}(M',L^k)\To\Omega^{0,q}(M',L^k),\]
be the (formal) Kodaira Laplacian on $M'$. Our purpose is to give
an expression of the Kodaira Laplacian near the boundary
in terms of an adapted frame of $TM'$ and an
$\R$-equivariant holomorphic frame 
of $L$. 

Let $s$ be a $\R$-equivariant holomorphic frame 
of $L$ over an open subset $U\Subset M'$ and $\phi\in\cC^\infty(U)$ 
be the local weight of the metric $h^L$ defined by $|s|^2_{h^L}=e^{-2\phi}.$ 
We define the operators $\overline\partial_{k,s}$, 
$\overline\partial^\ast_{k,s}$, $\Box^{(q)}_{k,s}$, which are the 
localizations of the operators $\overline\partial_k$, 
$\overline\partial^\ast_{f,k}$, $\Box^{(q)}_{f,k}$, 
with respect to the local frame $s$ on $U$. 
We set 
\begin{equation}\label{localization operators}
\begin{split}
&\overline\partial_{k,s}:=\overline\partial+
k\overline\partial\phi\wedge: \Omega^{0, q}(U)\rightarrow\Omega^{0, q+1}(U),\\
&\overline\partial^\ast_{k,s}:=\overline\partial^\ast+k(\overline\partial\phi\wedge)^\ast:
\Omega^{0, q+1}(U)\rightarrow\Omega^{0, q}(U),\\
&\Box_{k,s}^{(q)}:=\overline\partial_{k,s}\overline\partial^\ast_{k,s}+\overline\partial^\ast_{k,s}\overline\partial_{k,s}: \Omega^{0, q}(U)\rightarrow\Omega^{0, q}(U),\end{split}
\end{equation}
where $\ddbar: \Omega^{0,q}(U)\To\Omega^{0,q+1}(U)$ is the Cauchy-Riemann operator and $\overline\partial^\ast, \overline\partial_{k,s}^\ast$ are formal adjoint of $\overline\partial$ and $\overline\partial_{k,s}$ with respect to $(\,\cdot\,|\,\cdot\,)_{M'}$ respectively. We have the following unitary identifications:
\begin{equation}\label{e-gue220313ycd}
\begin{split}
\Omega^{0, q}(U, L^k)&\leftrightarrow\Omega^{0, q}(U),\\
g=s^k \tilde g&\leftrightarrow \hat g(z)=e^{-k\phi}s^{-k}g=\tilde g(z)e^{-k\phi}, \tilde g\in\Omega^{0, q}(U),\\
\overline\partial_k&\leftrightarrow\overline\partial_{k,s}, \overline\partial g=s^k e^{k\phi}\overline\partial_{k,s}\hat g,\\
\overline\partial^\ast_{f,k}&\leftrightarrow\overline\partial^\ast_{k,s}, \overline\partial^\ast g=s^k e^{k\phi} \overline\partial^\ast_{k,s}\hat g,\\
\Box^{(q)}_{f,k}&\leftrightarrow\Box^{(q)}_{k,s}, \Box^{(q)}_{f,k}g=s^k e^{k\phi}\Box^{(q)}_{k,s}\hat g.
\end{split}
\end{equation}
On $U$, let  $\{L_j\}_{j=1}^n$ be an orthonormal frame of $T^{1, 0}M'$ and its dual frame $\{e^j\}_{j=1}^n$ with $e^n=\frac{\partial\rho}{|\partial\rho|}$. We denote by $e^{\overline j}=\overline{e^j}$ and $L_{\overline j}=\overline{L_j}, j=1, \ldots, n.$
We can check that
\begin{equation}\label{e-gue220309yyd}
\begin{split}
&\overline\partial_{k,s}=\sum_{j=1}^n(e^{\overline j}\wedge)\circ(L_{\overline j}+k L_{\overline j}\phi)+\sum_{j=1}^n((\overline\partial e^{\overline j})\wedge)\circ (e^{\overline j}\wedge)^\ast, \\
&\overline\partial^\ast_{k,s}=\sum_{j=1}^n(e^{\overline j}\wedge)^\ast\circ(L_{\overline j}^\ast+kL_j\phi)+\sum_{j=1}^n(e^{\overline j}\wedge)\circ((\overline\partial e^{\overline j})\wedge)^\ast,
\end{split}
\end{equation} 
where $L^*_{\ol j}$ is the formal adjoint of $L_{\ol j}$ with respect to $(\,\cdot\,|\,\cdot\,)_{M'}$. 
From \eqref{e-gue220309yyd}, we can repeat the proof of~\cite[Prop. 4.1]{Hsiao18} with minor change and deduce that 

\begin{proposition} \label{s2-pmsmilkI} 
With the notations used before, we have 
\begin{equation} \label{e-gue220309yydI} 
\begin{split}
\Box^{(q)}_{k,s}&=\ddbar_{k,s}\,\ol{\pr}^*_{k,s}+
\ol{\pr}^*_{k,s}\,\ddbar_{k,s} \\
       &=\sum^{n}_{j=1}(-L_j+kL_{j}(\phi))(L_{\ol j}+kL_{\ol j}(\phi))\\
       &+\sum^{n}_{j,t=1}(e^{\ol j})\wedge\circ(e^{\ol t}\wedge)^*\circ[L_{\ol j}+kL_{\ol j}(\phi), -L_{t}+kL_{t}(\phi)]\\
&+\varepsilon(\ol L+k\ol L(\phi))+\varepsilon(-L+kL(\phi))+g,
\end{split}
\end{equation}
where $\varepsilon(\ol L+k\ol L(\phi))$ denotes remainder terms of the form $\sum a_j(L_{\ol j}+kL_{\ol j}(\phi))$ with $a_j$ smooth, matrix-valued and independent of $k$, for all $j$, and similarly for $\varepsilon(-L+kL(\phi))$ and $g$ is a smooth function independent of $k$. 
\end{proposition}

We choose real coordinates $x=(x_1, \ldots, x_{2n})$ on $U$. Write $$L_j=\sum_{\alpha=1}^{2n}a_{j,\alpha}\frac{\partial}{\partial x_{\alpha}}, a_{j,\alpha}\in\cC^\infty(U, \mathbb C),\ \ j=1,\ldots,n,\ \ \alpha=1,\ldots,2n.$$ Define $$C_{\alpha,\overline\beta}=\sum_{j=1}^na_{j,\alpha}\overline{a_{j,\beta}},\ \  \alpha, \beta=1, \ldots, 2n.$$ Then the matrix $[C_{\alpha,\overline\beta}]=[C_{\alpha,\overline\beta}]_{2n\times 2n}$ is  Hermitian self-adjoint. Moreover, we have 

\begin{lemma}
${\rm Re\,}[C_{\alpha,\overline\beta}]$ is positive, where ${\rm Re\,}[C_{\alpha,\overline\beta}]$ is the real part of $[C_{\alpha,\overline\beta}]$.
\end{lemma}
\begin{proof}
We denote by $C=[C_{\alpha,\overline\beta}]$ and write $C=A+\sqrt {-1}B$, where $A={\rm Re\,}C$ and $B$ is a real $2n\times 2n$ matrix. Since $C$ is Hermitian, one has $A^T=A$ and $B^T=-B$, where $A^T$ and $B^T$ denote the transpose of $A$ and $B$ respectively. Since $B^T=-B$, we can check that for any $v=(x_1, \ldots, x_{2n})^T\in\R^{2n}$, we have $v^TBv=0$
and hence $v^TCv=v^TAv$. 
Let $v=(x_1, \ldots, x_{2n})^T\in\R^{2n}$ with $Av=0$. We are going to prove that $v=0$. Write ${\mathcal A}=[a_{j,\alpha}]_{n\times 2n}$. Then the matrix 
$$\left[\begin{array}{c}
\mathcal{A} \\
\overline{\mathcal{A}}
\end{array}\right]_{2n\times 2n}$$ 
is invertible because $\{L_j, L_{\overline j}\}_{j=1}^n$ 
is a basis of $\mathbb CTM'$ over $D$. 
Since $C={\mathcal A}^T\overline {\mathcal A}$ and $v^TCv=0$, 
one has $\overline{\mathcal A}v=0$ and thus $\mathcal Av=0$. 
Hence $v=0$ which implies that $A$ is non-degenerate. 
Since $v^TCv=v^TAv\geq0$, for every 
$v=(x_1, \ldots, x_{2n})^T\in\R^{2n}$, we conclude that $A$ is positive. 
\end{proof} 

Now, we express $\Box_s^{(1)}$ with respect to the real coordinates $(x_1, \ldots, x_{2n})$. As before, write $L_j=\sum_{\alpha=1}^{2n}a_{j,\alpha}\frac{\partial }{\partial x_{\alpha}}, j=1, \ldots, n$. From \eqref{e-gue220309yydI}, one has
\begin{equation}\label{e-gue220311ycd}
\begin{split}
\Box_{k,s}^{(1)}&=\sum_{\alpha, \beta=1}^{2n}
\sum_{j=1}^na_{j,\alpha}\overline{a_{j,\beta}}
\left(-\frac{\partial}{\partial x_\alpha}+
k\frac{\partial\phi}{\partial x_\alpha}\right)
\left(\frac{\partial }{\partial x_\beta}
+k\frac{\partial\phi}{\partial x_\beta}\right)+{\mathcal L}_1+
k{\mathcal L}_{0}\\
&=\sum_{\alpha, \beta=1}^{2n} C_{\alpha,\overline\beta}
\left(-\frac{\partial}{\partial x_\alpha}+
k\frac{\partial\phi}{\partial x_\alpha}\right)
\left(\frac{\partial }{\partial x_\beta}+k\frac{\partial\phi}{\partial x_\beta}\right)+
{\mathcal L}_1+k{\mathcal L}_{0},
\end{split}
\end{equation}
where $\mathcal L_1$ and $\mathcal L_0$ are $k$-independent differential operators of order $1$ and $0$ respectively.  

For $p\in X$, we can choose coordinates $(x_1, \ldots, x_{2n})$ on an open set $U$ of $M'$ such that $(x_1, \ldots, x_{2n-1})$ are the BRT coordinates on $U\cap X$, $T=\frac{\pr}{\pr x_{2n-1}}$ on $U\cap X$, $x_{2n}=\rho$ and 
\begin{equation}\label{e-gue220311ycdI}
\frac{\pr}{\pr x_{2n}}=\frac{\pr}{\pr\rho}+O(\abs{\rho}).
\end{equation}
Since $e^n=\partial\rho/|\partial\rho|$, we have 
\[L_n=\frac{1}{\abs{\frac{\pr}{\pr\rho}-iJ(\frac{\pr}{\pr\rho})}}
\left(\frac{\pr}{\pr\rho}-iJ\Big(\frac{\pr}{\pr\rho}\Big)\right).\]
From this observation and \eqref{e-gue220311yyd}, we have 
\begin{equation}\label{e-gue2203121yyd}
\begin{split}
&L_n=\frac{1}{\sqrt{2}}\left(-i\frac{\pr}{\pr x_{2n-1}}+\frac{\pr}{\pr x_{2n}}\right)+O(\abs{\rho}),\\[2pt]
&L_j=\sum^{2n-1}_{\alpha=1}a_{j\alpha}\frac{\pr}{\pr x_\alpha}+O(\abs{\rho}),\ \ j=1,\ldots,n-1. 
\end{split}
\end{equation}
Thus, 
\begin{equation}\label{e-gue220312yydI}
\begin{split}
&a_{j,2n}|_X=0,\ \ j=1,\ldots,n-1,\\
&a_{n,\alpha}|_X=0,\ \ \alpha=1,\ldots,2n-2,\\
&a_{n,2n-1}|_X=-\frac{i}{\sqrt{2}},\ \ a_{n,2n}|_X=\frac{1}{\sqrt{2}}.
\end{split}
\end{equation}
From \eqref{e-gue220312yydI}, we can check that 
\begin{equation}\label{e-gue220312yydII}
\begin{split}
&C_{\alpha,\ol 2n}|_X=0,\ \ \alpha=1,\ldots,2n-2,\\
&C_{2n-1,\ol{2n}}|_X=-\frac{i}{2},\ \ C_{2n,\ol{2n}}|_X=\frac{1}{2},\\
&C_{2n,\ol{2n-1}}|_X=\frac{i}{2}.
\end{split}
\end{equation}
From \eqref{e-gue220311ycd}, \eqref{e-gue220311ycdI} and \eqref{e-gue220312yydII}, it is straightforward to check that 
\begin{equation}\label{21-10-6-a1}
\begin{split}
\Box_{k,s}^{(1)}=&\frac{1}{2}\Big(-\frac{\partial}{\partial \rho}+
k\frac{\partial\phi}{\partial \rho}\Big)\Big(\frac{\partial}{\partial \rho}
+k\frac{\partial\phi}{\partial \rho}\Big)+
\sum_{\alpha, \beta=1}^{2n-1}C_{\alpha,\overline\beta}
\Big(-\frac{\partial}{\partial x_{\alpha}}+
k\frac{\partial\phi}{\partial x_\beta}\Big)\Big(\frac{\partial}{\partial x_\beta}+k\frac{\partial\phi}{\partial x_\beta}\Big)\\ 
&+\frac{i}{2}\Big(-\frac{\partial}{\partial \rho}+
k\frac{\partial\phi}{\partial \rho}\Big)\Big(\frac{\partial}{\partial x_{2n-1}}+k\frac{\partial\phi}{\partial x_{2n-1}}\Big)
-\frac{i}{2}\Big(-\frac{\partial}{\partial x_{2n-1}}+
k\frac{\partial\phi}{\partial x_{2n-1}}\Big)\Big(\frac{\partial}{\partial\rho}
+k\frac{\partial\phi}{\partial\rho}\Big)\\
&+\sum_{\alpha=1}^{2n-1}r_{n,\alpha}
\Big(-\frac{\partial }{\partial x_{\alpha}}
+k\frac{\partial\phi}{\partial x_{\alpha}}\Big)
\Big(\frac{\partial}{\partial\rho}+
k\frac{\partial\phi}{\partial\rho}\Big)+
\sum_{\beta=1}^{2n-1}\hat r_{n,\ol\beta}
\Big(-\frac{\partial}{\partial\rho}+
k\frac{\partial\phi}{\partial\rho}\Big)
\Big(\frac{\partial}{\partial x_\beta}+
k\frac{\partial\phi}{\partial x_\beta}\Big)\\
&+{\mathcal L_1}(f)+k\mathcal L_0(f), 
\end{split}
\end{equation}
where $r_{n,\ol\alpha}$, $\hat r_{n,\ol\beta}$ are 
$k$-independent smooth functions with 
$r_{n,\ol\alpha}=O(\abs{\rho})$, $\hat r_{n,\ol\beta}=O(\abs{\rho})$, 
for every $\alpha, \beta=1,\ldots,2n-1$,  $\mathcal L_1$ 
and $\mathcal L_0$ are $k$-independent differential operators 
of order $1$ and $0$ respectively. 

At this point we can also define the function $\phi_1$
appearing in the formula \eqref{e-gue220510wz} of the coefficient $b_0$.
Let $s$ be a $\R$-equivariant holomorphic frame of $L$ on an open set 
$U$ of $M'$, with $D:=U\cap X\neq\emptyset$. 
Let $x=(x_1,\ldots,x_{2n-1})$ be canonical coordinates of $X$ on $D$ 
with $T=\frac{\pr}{\pr x_{2n-1}}$ on $D$ (cf.\ Definition \ref{D:BRT})
and $\tilde x=(x_1,\ldots,x_{2n-1},x_{2n})$, $x_{2n}=\rho$, 
are local coordinates of $U$ . 
Write 
\[\phi(\tilde x)=\phi_0(x)+\rho\phi_1(x)+O(\abs{\rho}^2).\]
Put for $(x,\eta)\in D\times\R^{2n-1}$,
\begin{equation}\label{e-gue220426yydz}
\varphi_1(x,\eta)=\frac1i\left(\phi_1^2(x)+2\sum_{\alpha, \beta=1}^{2n-1}
C_{\alpha,\ol\beta}\left(\eta_\alpha\eta_\beta+\frac{\partial\phi_0}{\partial x_\alpha}
\frac{\partial\phi_0}{\partial x_\beta}+i \frac{\partial\phi_0}{\partial x_\alpha}\eta_\beta
-i \eta_\alpha\frac{\partial\phi_0}{\partial x_\beta}\right)
-2\phi_1(x)\eta_{2n-1}\right)^{\!\!1/2}\,.
\end{equation}


\section{Semi-classical Poisson operators}\label{s-gue220313yyd}

\subsection{The Poisson operator for $(0,1)$-forms}\label{s-gue220313yydI}

We first consider $(0,1)$-forms. We will use the same notations and assumptions 
as before. The next Lemma gives a spectral gap for the 
Kodaira Laplacian on $(0,1)$-forms with Dirichlet boundary conditions. 
We use here the fact that the line bundle is 
positive on the whole domain $\overline{M}$.

\begin{lemma}\label{l-gue220313yyd}
There exist positive constants $c$ and $k_0$ such that
\begin{equation}\label{e-gue220313ycdIII}
(\,\Box^{(1)}_{f,k}\,u\,|\,u\,)_{\ol M,k}\geq c k\|u\|^2_{\ol M,k}
\end{equation}
when $k\geq k_{0}$ for all $u\in\Omega^{0, 1}(\overline M, L^k)$ with $u|_{X}=0$.
\end{lemma} 

\begin{proof}
Let $p\in\ol M$. One can choose a neighborhood $U$ of $p$ in $M'$ 
such that we can find an $\R$-equivariant holomorphic frame
$s$ of $L$ on $U$ with $|s|^2=e^{-2\phi}$ and an orthonormal frame 
$\{L_{\overline j}\}_{j=1}^n$ of $T^{0, 1} M'$ over $U$ as in
Proposition~\ref{s2-pmsmilkI} 
and as before, let $\{e^{\overline j}\}_{j=1}^n$ be its dual frame. 
Let $u\in\Omega^{0,1}_c(U\cap\ol M,L^k)$ with $u|_X=0$. 
On $U$, write $u=\tilde u\otimes s^k$, $\tilde u\in \Omega^{0, 1}_c(U\cap\ol M)$ 
with $\tilde u|_X=0$. Let $\hat u:=\tilde ue^{-k\phi}$. From \eqref{e-gue220313ycd}, 
\eqref{e-gue220309yydI} and by using integration by parts, we get 
\begin{equation}\label{e-gue220313ycdI}
\begin{split}
&(\,\Box^{(1)}_{f,k}u\,|\,u\,)_{\ol M,k}=(\,\Box^{(1)}_s\hat u\,|\,\hat u\,)_{\ol M}\\
&=\sum_{j=1}^n\|(L_{\ol j}+kL_{\ol j}(\phi))\hat u\|_{\ol M}^2+(\,\sum^{n}_{j,t=1}(e^{\ol j})\wedge\circ(e^{\ol t}\wedge)^*\circ[L_{\ol j}+kL_{\ol j}(\phi), -L_{t}+kL_{t}(\phi)]\hat u\,|\,\hat u\,)_{\ol M}\\
&+(\,(\varepsilon(\ol L+k\ol L(\phi))+\varepsilon(-L+kL(\phi))+g)\hat u\,|\hat u\,)_{\ol M},
\end{split}
\end{equation}
where $\varepsilon(\ol L+k\ol L(\phi))$, $\varepsilon(-L+kL(\phi))$ and $g$ are as in \eqref{e-gue220309yydI}. We have 
\begin{equation}\label{e-gue220314yyd}
\begin{split}
&[L_{\ol j}+kL_{\ol j}(\phi), -L_{t}+kL_{t}(\phi)]\\
&=[L_{\ol j},-L_t]+k(L_{\ol j}L_t+L_tL_{\ol j})\phi\\
&=\sum^n_{\ell=1}\alpha_\ell(-L_{\ell})+\sum^n_{\ell=1}\beta_\ell L_{\ol\ell}+k(L_{\ol j}L_t+L_tL_{\ol j})\phi\\
&=\sum^n_{\ell=1}\alpha_\ell(-L_{\ell}+kL_{\ell}\phi)+\sum^n_{\ell=1}\beta_\ell(L_{\ol\ell}+kL_{\ol\ell}\phi)\\
&\quad-k\bigr(\sum_{\ell=1}\alpha_\ell L_{\ell}\phi+\sum^n_{\ell=1}\beta_\ell L_{\ol\ell}\phi\bigr)+k(L_{\ol j}L_t+L_tL_{\ol j})\phi\\
&=\varepsilon(\ol L+k\ol L(\phi))+\varepsilon(-L+kL(\phi))
+k\langle\,[L_{\ol j},-L_t], -\ddbar\phi+\pr\phi\,\rangle+k(L_{\ol j}L_t+L_tL_{\ol j})\phi,
\end{split}
\end{equation}
where $\alpha_\ell$, $\beta_\ell$, $\ell=1,\ldots,n$, are $k$-independent smooth functions. 
We can repeat the proof of~\cite[Lemma 4.1]{HM12} and deduce that for every $j, t=1,\ldots,n$, 
\begin{equation}\label{e-gue220314yydI}
\langle\,[L_{\ol j},-L_t], -\ddbar\phi+\pr\phi\,\rangle+k(L_{\ol j}L_t+L_tL_{\ol j})\phi=\langle\,L_t\wedge L_{\ol j},2\pr\ddbar\phi\,\rangle.
\end{equation}
From  \eqref{e-gue220314yyd}, \eqref{e-gue220314yydI} and notice that $R^L$ is positive, there is a constant $c_0>0$ such that 
\begin{equation}\label{e-gue220313ycdII}
\begin{split}
&(\,\sum^{n}_{j,t=1}(e^{\ol j})\wedge\circ(e^{\ol t}\wedge)^*\circ[L_{\ol j}+kL_{\ol j}(\phi), -L_{t}+kL_{t}(\phi)]\hat u\,|\,\hat u\,)_{\ol M}\\
&\geq c_0k\norm{\hat u}^2_{\ol M}-
\abs{(\,(\varepsilon(\ol L+k\ol L(\phi))+\varepsilon(-L+kL(\phi)))\hat u\,|\,\hat u\,)_{\ol M}}.
\end{split}
\end{equation}
From \eqref{e-gue220313ycdI}, \eqref{e-gue220313ycdII} and by using integration by parts again, we deduce that \eqref{e-gue220313ycdIII} holds for $u\in\Omega^{0,1}_c(U\cap\ol M,L^k)$. By using 
partition of unity, we conclude that \eqref{e-gue220313ycdIII} holds for $u\in\Omega^{0,1}(\ol M,L^k)$ with $u|_X=0$. 
\end{proof} 

From Lemma~\ref{e-gue220313ycdIII}, we see that the map
\begin{equation}\label{e-gue220313ycds}
\begin{split}
\Omega^{0, 1}(\overline M, L^k)&\To \Omega^{0, 1}(\overline M, L^k)\oplus\cC^\infty(X,T^{*0,q}M'\otimes L^k),\\
u&\mapsto(\Box^{(1)}_{f,k}u,\gamma u)
\end{split}
\end{equation}
is injective, for $k$ large, where  $\gamma$ denotes the restriction operator to the boundary $X$. From now on, we assume that $k$ is large enough so that the map \eqref{e-gue220313ycds} is injective. Hence, the Poisson operator 
 $$P_k^{(1)}: \cC^\infty(X, T^{\ast 0, q}M'\otimes L^k)\rightarrow\Omega^{0,1}(\ol M,L^k)$$
of $\Box^{(1)}_{f,k}$ is well-defined. That is, if
$u\in\cC^\infty(X,T^{*0,q}M'\otimes L^k)$,
then
\[P^{(1)}_ku\in\Omega^{0,1}(\ol M,L^k),\ \ \Box^{(1)}_{f,k}P^{(1)}_ku=0,\ \
\gamma u=u.\]
Moreover, if $v\in\Omega^{0,1}(\ol M,L^k)$ and $\Box^{(1)}_{f,k}v=0$, then
$v=P^{(1)}_k\gamma v$. The boundary problem $(\Box^{(1)}_{f,k},\gamma)$ on $\overline{M}$
is the Dirichlet boundary problem, which is a regular elliptic boundary problem. 
By Boutet de Monvel~\cite[p.\,29]{B71} the operator $P^{(1)}_k$ 
extends continuously
\begin{equation}\label{e-gue190313ad}
P^{(1)}_k: W^\ell(X, T^{*0,1}M'\otimes L^k)\rightarrow W^{\ell+\frac{1}{2}}(\overline M, T^{*0,1}M'\otimes L^k),\ \ \forall\ell\in\R,
\end{equation}
and there is a continuous operator 
\begin{equation}\label{e-gue190429yyd}
N^{(1)}_k: W^\ell(\ol M, T^{*0,1}M'\otimes L^k)\rightarrow W^{\ell+2}(\overline M, T^{*0,1}M'\otimes L^k),\ \ 
\forall\ell\in\R,
\end{equation}
such that 
\begin{equation}\label{e-gue190418yyda}
\begin{split}
&N^{(1)}_k\Box^{(1)}_{f,k}+P^{(1)}_k\gamma=I\ \ \mbox{on $\Omega^{0,1}(\ol M,L^k)$},\\
&\gamma N^{(1)}_k=0\ \ \mbox{on $\Omega^{0,1}(\ol M,L^k)$},\\
&\Box^{(1)}_{f,k}N^{(1)}_k=I\ \ \mbox{on $\Omega^{0,1}(\ol M,L^k)$}.
\end{split}
\end{equation} 
\begin{lemma}\label{l-gue220314yyd}
For every $\ell\in\R$, there are $n_\ell, \hat n_\ell\in\mathbb N$ such that 
\begin{equation}\label{e-gue220314yydII}
N^{(1)}_k:=O(k^{n_\ell}): W^\ell(\ol M, T^{*0,1}M'\otimes L^k)\rightarrow W^{\ell+2}(\overline M, T^{*0,1}M'\otimes L^k)
\end{equation}
and 
\begin{equation}\label{e-gue220314yydIIa}
P^{(1)}_k:=O(k^{\hat n_\ell}): W^\ell(X, T^{*0,1}M'\otimes L^k)\rightarrow W^{\ell+\frac{1}{2}}(\overline M, T^{*0,1}M'\otimes L^k).
\end{equation}
\end{lemma}

\begin{proof}
Fix $\ell\in\R$. If we go through the work of Boutet de Monvel~\cite{B71}, there are $\tilde n_\ell\in\mathbb N$ and $C_\ell>0$ such that 
\begin{equation}\label{e-gue220314yydm}
\norm{N^{(1)}_ku}_{\ell+2,\ol M,k}\leq C_sk^{\tilde n_\ell}\Bigr(\norm{u}_{\ell,\ol M,k}+\norm{N^{(1)}_ku}_{\ol M,k}\Bigr), 
\end{equation}
for every $u\in\Omega^{0,1}(\ol M,L^k)$. From Lemma~\ref{l-gue220313yyd} and \eqref{e-gue190418yyda}, we have 
\begin{equation}\label{e-gue220314yydn}
ck\norm{N^{(1)}_ku}^2_{\ol M,k}\leq(\,\Box^{(1)}_{f,k}N^{(1)}_ku\,|\,N^{(1)}_ku\,)_{\ol M,k}\leq\norm{N^{(1)}_ku}_{\ol M,k}\norm{u}_{\ol M,k},
\end{equation}
for all  $u\in\Omega^{0,1}(\ol M,L^k)$, where $c>0$ is as in \eqref{e-gue220313ycdIII}. From \eqref{e-gue220314yydm} and \eqref{e-gue220314yydn}, 
we get \eqref{e-gue220314yydII}.  

From the classical construction of elliptic boundary value problems~\cite{B71}, 
there is a continuous operator 
\[\tilde P^{(1)}_k: \cC^\infty(X,{T^{*0,1}}M'\otimes L^k)\To{\Omega^{0,1}}(\ol M,L^k)\]
with $\gamma\tilde P^{(1)}_k=I$ on $\cC^\infty(X,{T^{*0,1}}M'\otimes L^k)$, 
$\Box^{(1)}_{f,k}\tilde P^{(1)}_k$ is smoothing (not $O(k^{-\infty})$) such that $\tilde P^{(1)}_k$ 
can be continuously extended to 
\[\tilde P^{(1)}_k: W^\ell( {X}, 
T^{*0,1}M'\otimes L^k)\rightarrow W^{\ell+\frac{1}{2}}(\overline M, T^{*0,1}M'\otimes L^k),\]
for every $\ell\in\R$ and for every $\ell_1, \ell_2\in\R$, there are 
$n_{\ell_1}, n_{\ell_1,\ell_2}\in\mathbb N$ such that 
\begin{equation}\label{e-gue220315yyd}
\tilde P^{(1)}_k:=O(k^{n_{\ell_1}}): W^{\ell_1}({X}, 
T^{*0,1}M'\otimes L^k)\rightarrow W^{\ell_1+\frac{1}{2}}(\overline M, T^{*0,1}M'\otimes L^k)
\end{equation}
and 
\begin{equation}\label{e-gue220315yydI}
\Box^{(1)}_{f,k}\tilde P^{(1)}_k:=O(k^{n_{\ell_1,\ell_2}}): W^{\ell_1}({X}, T^{*0,1}M'\otimes L^k)\rightarrow W^{\ell_2}(\overline M, T^{*0,1}M'\otimes L^k).
\end{equation}
From \eqref{e-gue190418yyda}, we get 
\begin{equation}\label{e-gue220315yydII}
\tilde P^{(1)}_k-P^{(1)}_k=N^{(1)}_k\Box^{(1)}_{f,k}\tilde P^{(1)}_k. 
\end{equation}
From \eqref{e-gue220314yydII}, \eqref{e-gue220315yyd}, \eqref{e-gue220315yydI} and \eqref{e-gue220315yydII}, we get \eqref{e-gue220314yydIIa}.
\end{proof}

\subsection{The Poisson operator for $F^{(0)}_{[k\delta/4, 2k\delta]}$}
\label{s-gue220315yyd} 

We denote for convenience by $T$ the Lie derivative along the direction $T$. 
Then, $\mathcal L_T: \Omega^{0, q}(\ol M,L^k)
\rightarrow\Omega^{0, q}(\ol M,L^k)$. 
We shall also use the notation $-iT_M$ for the closure in $L^2$ of the $-iT$ operator, initially defined on
$\Omega^{0,q}(\ol M,L^k)$. 
Recall that for $u\in L^2_{(0,q)}(\ol M,L^k)$, we say that 
$u\in\Dom(-iT_M)$ if we can find a sequence 
$u_j\in\Omega^{0,q}(M,L^k)$, $j=1,2,\ldots$\,, 
with $\lim_{j\To\infty}\norm{u_j-u}_{\ol M,k}=0$ such that 
$\lim_{j\To\infty}\norm{-iTu_j-v}_{\ol M,k}=0$, for some 
$v\in L^2_{(0,q)}(M,L^k)$. We can repeat the proofs of~\cite[Theorem 4.1, Theorem 4.5]{HHL17} and deduce 
 
\begin{theorem}\label{t-gue220317yyd}
The operator $-iT_M: \Dom(-iT_M)\subset 
L^2_{(0,q)}(M,L^k)\To L^2_{(0,q)}(M,L^k)$ is self-adjoint, 
and its spectrum consists of a countable set of isolated eigenvalues
of finite multiplicity.
\end{theorem}

We now consider the boundary version of Theorem~\ref{t-gue220317yyd}. 
When restricted on $X$, $T: \cC^\infty(X,T^{*0,q}M'\otimes L^k)\rightarrow \cC^\infty(X,T^{*0,q}M'\otimes L^k)$. 
We extend $-iT$ to $L^2(X,T^{*0,q}M'\otimes L^k)$ in the sense of distribution: 
\[-iT: \Dom(-iT)\subset L^2(X,T^{*0,q}M'\otimes L^k)\To L^2(X,T^{*0,q}M'\otimes L^k),\]
where $\Dom(-iT)=\set{u\in L^2(X,T^{*0,q}M'\otimes L^k);\, -iTu\in L^2(X,T^{*0,q}M'\otimes L^k)}$.  It was established in~\cite[Theorem 4.1, Theorem 4.5]{HHL17} that

\begin{theorem}\label{t-gue220315yyd}
The operator $-iT: \Dom(-iT)\subset L^2(X,T^{*0,q}M'\otimes L^k)\To L^2(X,T^{*0,q}M'\otimes L^k)$ is self-adjoint, 
${\rm Spec\,}(-iT)$ is countable
and every element in ${\rm Spec\,}(-iT)$ is an eigenvalue of $-iT$, where ${\rm Spec\,}(-iT)$ denotes the spectrum of $-iT$.
\end{theorem}

For $\alpha\in{\rm Spec}(-i T_M)$, $\beta\in{\rm Spec\,}(-iT)$, we define 
\[\begin{split}
&\Omega^{0, q}_{\alpha}(\ol M, L^k)=\{u\in \Omega^{0, q}(\ol M, L^k);\, -iT_Mu=\alpha u\},\\
&\Omega^{0, q}_{\beta}(X, L^k)=\{u\in \Omega^{0, q}(X, L^k);\,-iTu=\beta u\},\\
&\cC^\infty_{\beta}(X, T^{*0,q}M'\otimes L^k)=\{u\in\cC^\infty(X, T^{*0,q}M'\otimes L^k);\, -iTu=\beta u\}.
\end{split}\]
We write $\cC^\infty_\alpha(\ol M,L^k):=\Omega^{0, 0}_{\alpha}(\ol M, L^k)$, $\cC^\infty_\beta(X,L^k):=\Omega^{0, 0}_{\beta}(X, L^k)$. 
Let $L^2_{(0, q), \alpha}(M, L^k)$, $L^2_{(0,q),\beta}(X,L^k)$ and $L^2_\beta(X,T^{*0,q}M'\otimes L^k)$ 
be the completions of $\Omega^{0, q}_{\alpha}(\ol M, L^k)$, 
$\Omega^{0, q}_{\beta}(X, L^k)$ and 
$\cC^\infty_{\beta}(X, T^{*0,q}M'\otimes L^k)$ 
with respect to $(\,\cdot\,|\,\cdot\,)_{\ol M,k}$, 
$(\,\cdot\,|\,\cdot\,)_{X,k}$ and $(\,\cdot\,|\,\cdot\,)_{X,k}$ respectively. Let 
\[\begin{split}
&Q^{(q)}_{\alpha,M}: L^2_{(0, q)}(M, L^k)\rightarrow 
L^2_{(0, q), \alpha}(M, L^k),\\
&Q^{(q)}_{\beta}: L^2(X, T^{*0,q}M'\otimes L^k)\rightarrow 
L^2_{\beta}(X, T^{*0,q}M'\otimes L^k)\end{split}\] 
be the orthogonal projections with respect to 
$(\,\cdot\,|\,\cdot\,)_{\ol M,k}$ and 
$(\,\cdot\,|\,\cdot\,)_{X,k}$ respectively. We have the Fourier decomposition
\begin{equation}\label{e:FD}
\begin{split}
L^2_{(0, q)}(M, L^k)=
&\bigoplus \big\{L^2_{(0, q), \alpha}(M, L^k):
\alpha\in{\rm Spec}(-iT_M)\big\},\\
L^2(X, T^{*0,q}M'\otimes L^k)=
&\bigoplus\big\{L^2_{\beta}(X, T^{*0,q}M'\otimes L^k):
\beta\in{\rm Spec}(-iT)\big\}.
\end{split}
\end{equation}
Let us fix $\delta>0$ and consider the function 
$\tau_\delta\in\cC^\infty_c((\delta/4, 2\delta))$ as in
\eqref{e-21-5-16-a2z} and set $\tau_{k\delta}(t):=\tau_\delta(t/k)$
for $k\in\N$
as in \eqref{e-gue220322ycdsz}.
We define 
\begin{equation}\label{e-gue220317yyda}
\begin{split}
&F^{(q)}_{\tau_{k\delta},M}: L^2_{(0, q)}(M, L^k)\rightarrow L^2_{(0, q)}(M, L^k),\\
&F^{(q)}_{\tau_{k\delta},M}(u)=\sum_{\alpha\in{\rm Spec}(-iT_M)}\tau_{k\delta}(\alpha)Q^{(q)}_{\alpha,M}u, \forall u\in L^2_{(0, q)}(M, L^k),
\end{split}
\end{equation}
\begin{equation}\label{e-gue220317yydb}
\begin{split}
&F^{(q)}_{\tau_{k\delta}}: L^2(X, T^{*0,q}M'\otimes L^k)\rightarrow L^2(X, T^{*0,q}M'\otimes L^k),\\
&F^{(q)}_{\tau_{k\delta}}(u)=\sum_{\alpha\in{\rm Spec}(-iT)}\tau_{k\delta}(\alpha)Q^{(q)}_{\alpha, k}u, \forall u\in L^2(X, T^{*0,q}M'\otimes L^k),
\end{split}
\end{equation}
\begin{equation}\label{e-gue220321yyd}
\begin{split}
&F^{(q)}_{[\frac{k\delta}{4}, 2k\delta],M}: L^2_{(0, q)}(M, L^k)\rightarrow L^2_{(0, q)}(M, L^k),\\
&F^{(q)}_{[\frac{k\delta}{4}, 2k\delta],M}(u)=\sum_{\alpha\in{\rm Spec}(-iT_M), \frac{k\delta}{4}\leq\alpha\leq 2k\delta}Q^{(q)}_{\alpha,M}u, \forall u\in L^2_{(0, q)}(M, L^k),
\end{split}
\end{equation}
\begin{equation}\label{e-gue220321yydI}
\begin{split}
&F^{(q)}_{[\frac{k\delta}{4}, 2k\delta]}: L^2(X, T^{*0,q}M'\otimes L^k)\rightarrow L^2(X, T^{*0,q}M'\otimes L^k),\\
&F^{(q)}_{[\frac{k\delta}{4}, 2k\delta]}(u)=
\sum_{\alpha\in{\rm Spec}(-iT),  \frac{k\delta}{4}
\leq\alpha\leq 2k\delta}Q^{(q)}_{\alpha}u,\ \ 
\forall u\in L^2(X, T^{*0,q}M'\otimes L^k).
\end{split}
\end{equation}
We now consider the case $q=0$. 
Recall that Assumption~\ref{a-gue220517yyd} holds. Note that 
\begin{equation}\label{e-gue220517yyd}
C_0+2(JT)(\phi)>0\ \ \mbox{near $X$}. 
\end{equation}

\begin{lemma}\label{l-gue220530yyd}
We have 
\begin{equation}\label{e-gue220517yydzz}
C_0+2(JT)(\phi)\geq0\ \ \mbox{on $\ol M$}. 
\end{equation}
\end{lemma}

\begin{proof}
Let $x_0\in\ol M$ be the point so that $(JT)(\phi))(x_0)={\rm min\,}\set{((JT)(\phi))(x);\, x\in\ol M}$. 
If $x_0\in X$. From \eqref{e-gue220517yyd}, we see that $(C_0+(JT)(\phi))(x_0)>0$ and hence \eqref{e-gue220517yydzz} holds. 
Assume that $x_0\notin X$. If $(JT)(x_0)=0$. We have $C_0+(JT)(\phi)(x_0)>0$ and hence \eqref{e-gue220517yydzz} holds. 
Assume that $(JT)(x_0)\neq0$.
Since $(JT)(\phi)(x_0)={\rm min\,}\set{((JT)(\phi))(x);\, x\in\ol M}$, we have $d_x((JT)(\phi))(x_0)=0$. Thus, $((JT)^2(\phi))(x_0)=0$. 
Since $\pr\ddbar\phi$ is positive definite, we have $((JT)^2(\phi))(x_0)=((T-iJT)(T+iJT)(\phi)))(x_0)>0$, we get a contradiction. Thus, 
$(JT)(x_0)=0$. The lemma follows. 
\end{proof}

From now on, we take $\delta>0$ so that 
\begin{equation}\label{e-gue220712yyd}
\mbox{$t+(JT)(\phi)(x)>0$, for every $x\in\ol M$ and every $t\in(\delta/4,2\delta)$}. 
\end{equation}
From Lemma~\ref{l-gue220530yyd}, this is always possible. 
We need 

\begin{lemma}\label{l-gue220321yyd}
With the notations used above,  
there exist positive constants $c$ and $k_0$ such that
\begin{equation}\label{e-gue220321yydII}
(\,\Box^{(0)}_{f,k}\,u\,|\,u\,)_{\ol M,k}\geq c k^2\|u\|^2_{\ol M,k}
\end{equation}
when $k\geq k_{0}$ for all $u\in\bigoplus_{\frac{k\delta}{4}\leq\alpha\leq 2k\delta}\cC^\infty_{\alpha}(\overline M, L^k)$ with $\gamma u=0$.
\end{lemma}

\begin{proof}
Let $u\in\cC^\infty(\ol M,L^k)$ with $\gamma u=0$ and $Tu=imu$, $m\in[\frac{k\delta}{4},2\delta]$. Fix $p\in\ol M$ and 
let $s$ be an $\R$-equivariant holomorphic frame 
of $L$ over an open subset $U\Subset M'$, $p\in U$, $|s|^2_{h^L}=e^{-2\phi}$. 
Let $L_n=\frac{1}{\sqrt{2}}(T-iJT)\in\cC^\infty(\ol M,T^{0,1}M')$. Let $\chi\in\cC^\infty_c(U\cap\ol M)$. Since $\gamma u=0$, we can integrate by parts and deduce that 
\begin{equation}\label{e-gue220321yydIIIa}
(\,\Box^{(0)}_{f,k}\,(\chi u)\,|\,\chi u\,)_{\ol M,k}\geq\norm{\ol L_n(\chi u)}^2_{\ol M,k}. 
\end{equation}
Now, 
\begin{equation}\label{e-gue220321yydIV}
\begin{split}
&\norm{\ol L_n(\chi u)}^2_{\ol M,k}=\frac{1}{2}(\,(T+iJT)(\chi u)\,|\,(T+iJT)(\chi u)\,)_{\ol M,k}\\
&=\frac{1}{2}\Bigr(\norm{T(\chi u)}^2_{\ol M,k}+\norm{(JT)(\chi u)}^2_{\ol M,k}\\
&\quad-i(\,T(\chi u)\,|\,(JT)(\chi u)\,)_{\ol M,k}
+i(\,(JT)(\chi u)\,|\,T(\chi u)\,)_{\ol M,k}\Bigr).
\end{split}
\end{equation}
Note that $[T,T-iJT]\in\cC^\infty(M',T^{1,0}M')$, $[T,T+iJT]\in\cC^\infty(M',T^{0,1}M')$. From this observation, we can check that 
\begin{equation}\label{e-gue220321yydV}
[T,JT]=0.
\end{equation}
Let $(JT)^*$ be the formal adjoint of $JT$ with respect to $(\,\cdot\,|\,\cdot\,)_{\ol M,k}$, We have 
\begin{equation}\label{e-gue220321yydVI}
(JT)^*=-JT+2k(JT)(\phi)+g,
\end{equation}
where $g$ is a $k$-independent smooth function. 
From \eqref{e-gue220321yydIV}, \eqref{e-gue220321yydV}, \eqref{e-gue220321yydVI} and by using integration by parts, we get 
\begin{equation}\label{e-gue220321yydVII}
\begin{split}
&\norm{\ol L_n({\color{red}\chi}u)}^2_{\ol M,k}\\
&=\frac{1}{2}\Bigr(\norm{T(\chi u)}^2_{\ol M,k}+\norm{(JT)(\chi u)}^2_{\ol M,k}
+i(\,(\chi u)\,|\,(2k(JT)(\phi)+g)T(\chi u)\,)_{\ol M,k}\Bigr)\\
&=\frac{1}{2}\Bigr(\norm{T(\chi u)}^2_{\ol M,k}+\norm{(JT)(\chi u)}^2_{\ol M,k}
+(\,(\chi u)\,|\,2k(JT)(\phi)\chi mu\,)_{\ol M,k}+O(k)\norm{u}^2_{\ol M,k}\Bigr). 
\end{split}
\end{equation} 
From \eqref{e-gue220712yyd} and  \eqref{e-gue220321yydVII}, we conclude that 
\begin{equation}\label{e-gue220517ycda}
\norm{\ol L_n(\chi u)}^2_{\ol M,k}\geq\frac{\delta_0}{2}\Bigr(\norm{T(\chi u)}^2_{\ol M,k}+O(k)\norm{u}^2_{\ol M,k}\Bigr),
\end{equation} 
where $\delta_0>0$ is a $k$-independent universal constant. 
From \eqref{e-gue220321yydIIIa} and \eqref{e-gue220517ycda}, we deduce that 
\begin{equation}\label{e-gue220321ycd}
(\,\Box^{(0)}_{f,k}(\chi u)\,|\,(\chi u)\,)_{\ol M,k}\geq\norm{\ol L_n(\chi u)}^2_{\ol M,k}\geq\frac{\delta_0}{2}\Bigr(\norm{T(\chi u)}^2_{\ol M,k}-C_Uk\norm{(\chi u)}^2_{\ol M,k}\Bigr),
\end{equation}
where $C_U>0$ is a constant independent of $k$, $u$. 

Now, assume that $\ol M\subset\bigcup^N_{j=1}U_j$, $N\in\mathbb N$, $U_j\Subset M'$ is an open small trivialization of $L$, for each $j=1,\ldots,N$. Let $\chi_j\in\cC^\infty_c(U_j\cap\ol M)$, $j=1,\ldots,N$, with $\sum^N_{j=1}\chi_j^2=1$ on $\ol M$. For each $j=1,\ldots,N$, we have 
\begin{equation}\label{e-gue220322yyd}
\begin{split}
\norm{\chi_jTu}^2_{\ol M,k}-C_j\norm{u}_{\ol M,k}\leq \norm{T(\chi_ju)}^2\leq 2\norm{\chi_jTu}^2_{\ol M,k}+C_j\norm{u}_{\ol M,k},\\
\norm{\chi_j\ol L_nu}^2_{\ol M,k}-C_j\norm{u}_{\ol M,k}\leq \norm{\ol L_n(\chi_ju)}^2\leq 2\norm{\chi_j\ol L_nu}^2_{\ol M,k}+C_j\norm{u}_{\ol M,k},
\end{split}
\end{equation}
where $C_j>1$ is a constant independent of $k$ and $u$. 

From \eqref{e-gue220321ycd} and \eqref{e-gue220322yyd}, it is straightforward to see that for each $j=1,\ldots,N$, 
\begin{equation}\label{e-gue220321ycdI}
\begin{split}
\norm{\chi_j\ol L_nu}^2_{\ol M,k}\geq\hat C_j\Bigr(\norm{\chi_jTu}^2_{\ol M,k}-\tilde C_jk\norm{u}^2_{\ol M,k}\Bigr),
\end{split}
\end{equation}
where $\hat C_j>0$, $\tilde C_j>0$ are constants independent of $k$ and $u$. From  \eqref{e-gue220321ycdI}, we get 
\begin{equation}\label{e-gue220322yydI}
\begin{split}
(\,\Box^{(0)}_{f,k}\,u\,|\,u\,)_{\ol M,k}\geq\sum^N_{j=1}\norm{\chi_j\ol L_nu}^2_{\ol M,k}\geq\hat C\Bigr(\norm{Tu}^2_{\ol M,k}-\tilde Ck\norm{u}^2_{\ol M,k}\Bigr),
\end{split}
\end{equation} 
where $\hat C>0$, $\tilde C>0$ are constants independent of $k$ and $u$. Since $u$ is in the space $\oplus_{\frac{k\delta}{4}\leq\alpha\leq 2k\delta}\cC^\infty_{\alpha}(\overline M, L^k)$, we have for $k\gg1$, 
\[\norm{Tu}^2_{\ol M,k}\geq \frac{k^2\delta^2}{16}\norm{u}^2_{\ol M,k}.\]
From this observation and \eqref{e-gue220322yydI}, we get \eqref{e-gue220321yydII}. 
\end{proof} 

In the rest of this section, we assume that $T(x)\neq0$, $(d\rho)(x)\neq0$, for almost every $x\in\ol M$. 
It is clear that the map
\begin{equation}\label{e-gue220313ycdsz}
\begin{split}
\cC^\infty(\overline M, L^k)&\To\cC^\infty(\overline M, L^k)\oplus\cC^\infty(X,L^k),\\
u&\To(\Box^{(0)}_{f,k}u,\gamma u)
\end{split}
\end{equation}
is injective. Hence, the Poisson operator 
 $$P_k^{(0)}: \cC^\infty(X,L^k)\rightarrow\cC^\infty(\ol M,L^k)$$
of $\Box^{(0)}_{f,k}$ is well-defined. The operator $P^{(0)}_k$ 
extends continuously
\begin{equation}\label{e-gue190313adz}
P^{(0)}_k: W^\ell(X,  L^k)\rightarrow W^{\ell+\frac{1}{2}}(\overline M, L^k),\ \ \forall\ell\in\R,
\end{equation}
and there is a continuous operator 
\begin{equation}\label{e-gue190429yydz}
N^{(0)}_k: W^\ell(\ol M, L^k)\rightarrow W^{\ell+2}(\overline M, L^k),\ \ 
\forall\ell\in\R,
\end{equation}
such that 
\begin{equation}\label{e-gue190418yydaz}
\begin{split}
&N^{(0)}_k\Box^{(0)}_{f,k}+P^{(0)}_k\gamma=I\ \ \mbox{on $\cC^\infty(\ol M,L^k)$},\\
&\gamma N^{(0)}_k=0\ \ \mbox{on $\cC^\infty(\ol M,L^k)$},\\
&\Box^{(0)}_{f,k}N^{(0)}_k=I\ \ \mbox{on $\cC^\infty(\ol M,L^k)$}.
\end{split}
\end{equation}
From Lemma~\ref{l-gue220321yyd}, we can repeat the proof of Lemma~\ref{l-gue220314yyd} with minor change and deduce 

\begin{lemma}\label{l-gue220314yydz}
Fix $\delta>0$. For every $\ell\in\R$, there are $n_\ell, \hat n_\ell\in\mathbb N$ such that 
\begin{equation}\label{e-gue220314yydIIz}
N^{(0)}_k\circ F^{(0)}_{[\frac{k\delta}{4}, 2k\delta],M}=O(k^{n_\ell}): W^\ell(\ol M, L^k)\rightarrow W^{\ell+2}(\overline M, L^k)
\end{equation}
and 
\begin{equation}\label{e-gue220314yydIIaz}
P^{(0)}_k\circ F^{(0)}_{[\frac{k\delta}{4}, 2k\delta]}=O(k^{\hat n_\ell}): W^\ell(X, L^k)\rightarrow W^{\ell+\frac{1}{2}}(\overline M, L^k).
\end{equation}
\end{lemma}

\subsection{Semi-classical behaviour of $P^{(1)}_kF^{(1)}_{\tau_{k\delta}}$}\label{s-gue220322yyd} 

As before, fix $\delta>0$ and let 
$\tau_\delta(t)\in\cC^\infty_c(]\frac{\delta}{4}, 2\delta[)$ be as in \eqref{e-21-5-16-a2z} and let $\tau_{k\delta}$ be as in 
\eqref{e-gue220322ycds}.  We will use the same notations as before. We will study semi-classical behaviour of $P^{(q)}_kF^{(q)}_{\tau_{k\delta}}$, where $q=0, 1$. We first consider $q=1$ case. Let $s$ be a $\R$-equivariant holomorphic frame of $L$ over an open subset $U\Subset M'$, $U\cap X\neq\emptyset$, and let $\phi\in\cC^\infty(U)$ be the local weight of the metric $h^L$ defined by $|s|^2_{h^L}=e^{-2\phi}$. Let $D:=U\cap X$. 
We assume $D$ is a canonical coordinate patch and let $x=(x_1,\ldots,x_{2n-1})$ be the BRT coordinates on $D$. Write $z_j=x_{2j-1}+\sqrt{-1}x_{2j}, 1\leq j\leq n-1$. Then $\{dz_j\}_{j=1}^{n-1}$ are smooth frames of $T^{\ast 1, 0}X$ over $D$. Let $\xi=(\xi_1,\ldots,\xi_{2n-1})$ be the dual coordinates of $x$. Let $\eta^{\ol j}$, $j=1,\ldots,n$, be a $T$-invariant orthonormal basis for $T^{*0,1}M'$ on $U$ with $\eta^{\ol n}=\frac{\ddbar\rho}{\abs{\ddbar\rho}}$. 
Put
\begin{equation}\label{e-gue131209}
B^{(1)}_{\tau_{k\delta}, s}(x, y):=\frac{k^{2n-1}}{(2\pi)^{2n-1}}\int_{\R^{2n-1}} e^{ik\langle x-y,~\xi\rangle}\tau_\delta(\xi_{2n-1})\sum_{j=1}^{n}\eta^{\overline j}|_x\otimes (\eta^{\overline j}|_y)^\ast d\xi,
\end{equation}
where for each $j=1,\ldots,n$, $\eta^{\overline j}|_x\otimes (\eta^{\overline j}|_y)^\ast$ denotes the linear map given by 
$(\eta^{\overline j}|_x\otimes (\eta^{\overline j}|_y)^\ast)(\eta^{\overline s}|_y)=\delta_{s,j}\eta^{\overline j}|_x$, $s=1,\ldots,n$. Thus, for $f\in\cC^\infty_c(D,T^{*0,1}M')$, we write $f=\sum_{j=1}^{n}f_{\overline j}(y)\eta^{\overline j}|_y$. Then
\begin{equation}\label{e-gue220322yydp}
B^{(1)}_{\tau_{k\delta}, s}f(x)=\sum_{j=1}^{n}\left[\frac{k^{2n-1}}{(2\pi)^{2n-1}}\int_X\int_{\R^{2n-1}} e^{ik\langle\,x-y\,,\,\xi\,\rangle}\tau_\delta(\xi_{2n-1})f_j(y)d\xi dy\right]\eta^{\overline j}|_x.
\end{equation}
Let $F^{(q)}_{\tau_{k\delta}, s}$ be the localization of $F^{(q)}_{\tau_{k\delta}}$ with respect to $s$ on $U$. The following lemma has been obtained in \cite[Lemma 5.4]{HHL17}. 

\begin{lemma}\label{l-gue131209}
We have
\[F^{(1)}_{\tau_{k\delta},s}-B^{(1)}_{\tau_{k\delta}, s}=O(k^{-\infty}):W^\ell_{{\rm comp\,}}(D,T^{*0,1}M')\To W^\ell_{{\rm loc\,}}(D, T^{*0,1}M'),\ \ \forall\ell\in\mathbb Z.\]
\end{lemma} 

We extend $x_1,\ldots,x_{2n-1}$ to smooth functions on $U$ so that $\tilde x:=(x_1,\ldots,x_{2n-1},\rho)=(x,\rho)$ are local coordinates on $D$ and 
$\frac{\pr}{\pr x_{2n-1}}=T$ on $U$.
Write
\begin{equation}\label{e-gue220323yydp}
\begin{split}
&\varphi_0(x, y, \xi):=\langle x-y, \xi\rangle,\\
&g_0(x, y, \xi, k):=\frac{k^{2n-1}}{(2\pi)^{2n-1}}\tau_\delta(\xi_{2n-1})\sum_{j=1}^{n}\eta^{\overline j}|_x\otimes (\eta^{\overline j}|_y)^\ast.
\end{split}
\end{equation} 
We will extend $\varphi(x, y, \xi)$ and $g(x, y, \xi, k)$ to $U\times D\times \R^{2n-1}$ and we denote the extensions by $\varphi(\tilde x, y, \xi)$ and $g(\tilde x, y, \xi, k)$, respectively. We construct the extensions by the Taylor expansions and using Borel's lemma. Precisely, we will construct $\varphi(\tilde x, y, \xi)\in\cC^{\infty}(U\times D\times \R^{2n-1})$, 
\[g(\tilde x, y, \xi, k)\in S^{2n-1}_{{\rm cl\,}}(1; ((U\times U)\cap(\ol M\times X))\times \R^{2n-1},T^{*0,1}M'\boxtimes(T^{*0,1}M')^*)\] such that 
\begin{equation}\label{21-6-15-a1}
\varphi(\tilde x, y, \xi)\sim\varphi_0(x,y, \xi)+\rho\varphi_1(x,y, \xi)+\rho^2\varphi_2(x, y, \xi)+\ldots\\
\end{equation}
and
\begin{equation}\label{e-gue220323yydq}
\begin{split}
&g(\tilde x, y, \xi ,k)\sim k^{2n-1}g^0(\tilde x, y, \xi)+k^{2n-2}g^1(\tilde x, y, \xi)+\ldots\\
&\mbox{in $S^{2n-1}_{{\rm cl\,}}(1,((U\times U)\cap(\ol M\times X))\times\R^{2n-1},T^{*0,1}M'\boxtimes(T^{*0,1}M')^*)$},\\
&\mbox{$g^j(\tilde x, y, \xi)\in S^{0}(1;((U\times U)\cap(\ol M\times X))\times\R^{2n-1},T^{*0,1}M'\boxtimes(T^{*0,1}M')^*)$, $j=0,1,\ldots$},\\
&g^j(\tilde x, y, \xi)\sim g^j_0(x, y, \xi)+\rho g^j_1(x, y, \xi)+\rho^2 g^j_2(x, y, \xi)+\ldots, \ \ j=0,1,\ldots,\\
&\rho^\ell g^j_\ell( x, y, \xi)\in S^{0}(1; ((U\times U)\cap(\ol M\times X))\times \R^{2n-1},T^{*0,1}M'\boxtimes(T^{*0,1}M')^*),\ \ \ell=0,1,\ldots,
\end{split}
\end{equation}
where for each $(\tilde x, y)\in U\times D$, $g(\tilde x, y, \xi, k), g^j(\tilde x,y,\xi), \rho^\ell g^j_\ell( x, y, \xi, k): T^{*0,1}_yM'\To T^{*0,1}_{\tilde x}M'$, $\tilde x=(x,\rho)$, $j, \ell=0,1,\ldots$,
$\varphi(\tilde x, y, \xi)\sim\varphi_0(x,y, \xi)+\rho\varphi_1(x,y, \xi)+\rho^2\varphi_2(x, y, \xi)+\ldots$ means that for every $N\in\mathbb N$, $\ell\in\mathbb N_0$ and every compact set $K$ of $U\times D\times\R^{2n-1}$,  there is a constant $C_{N,\ell,K}>0$ such that
\[\norm{\rho^{-N-1}\Bigr(\tilde \varphi(\tilde x, y, \xi)-\sum^N_{\ell=0}\rho^\ell\varphi_\ell(x,y, \xi)\Bigr)}_{\cC^\ell(K)}\leq C_{N,\ell,K},\]
and similarly for $g^j(\tilde x, y, \xi)\sim g^j_0(x, y, \xi)+\rho g^j_1(x, y, \xi)+\rho^2 g^j_2(x, y, \xi)+\ldots$ and 
$\{\varphi_\ell(x, y, \xi)\}_{\ell=1}^{\infty}\subset S^0((1+\abs{\xi});D\times D\times\R^{2n-1})$,
\[\{g^j_\ell(x, y, \xi)\}_{\ell=1}^{\infty}\subset S^{0}(1; D\times D\times \R^{2n-1},T^{*0,1}M'\boxtimes(T^{*0,1}M')^*),\ \ j\in\mathbb N_0,\] 
will be determined later. 
It is clear that $$\varphi(\tilde x, y, \xi)|_X=\varphi_0(x, y, \xi), ~~g(\tilde x, y, \xi, k)|_X=g_0(x, y, \xi, k).$$ Then we define
 \begin{equation}\label{e-gue220323yydr}
 \tilde B_{\tau_{k\delta}, s}^{(1)}(\tilde x,y):=\int_{\R^{2n-1}}e^{ik\varphi(\tilde x, y, \xi)}g(\tilde x, y, \xi, k)d\xi
 \end{equation}
 and let $\tilde B^{(1)}_{\tau_{k\delta},s}: \cC^\infty_c(D,T^{*0,1}M')\To\Omega^{0,1}(U\cap\ol M)$ be the continuous map given by 
 \begin{equation}\label{e-gue220521yyd}
 \begin{split}
 \tilde B^{(1)}_{\tau_{k\delta},s}: \cC^\infty_c(D,T^{*0,1}M')&\To\Omega^{0,1}(U\cap\ol M),\\
 f&\To\int_{\R^{2n-1}}e^{ik\varphi(\tilde x, y, \xi)}g(\tilde x, y, \xi, k)f(y)dyd\xi.
 \end{split}\end{equation}

\begin{theorem}\label{21-5-19-thm1}
There exists $\varphi(\tilde x, y, \xi)\in S^0((1+\abs{\xi}); ((U\times U)\cap(\ol M\times X))\times \R^{2n-1})$ and $g(\tilde x, y, \xi, k)\in S^{2n-1}_{{\rm cl\,}}(1; ((U\times U)\cap(\ol M\times X)\times \R^{2n-1},T^{*0,1}M'\boxtimes(T^{*0,1}M')^*)$  such that $i\varphi_1(x,y,\xi)$ is real, 
\begin{equation}\label{e-gue220328yyd}
\mbox{$i\varphi_1(x,y,\xi)-\abs{\frac{\pr\phi}{\pr\rho}(x)}>C\abs{\xi}$ on $D\times D\times[\frac{\delta}{4},2\delta]$},
\end{equation}
$C>0$ is a constant, 
\begin{equation}\label{e-gue220328yydIz}
\begin{split}
&\mbox{$g^j_\ell( x, y, \xi)=0$ if $\xi_{2n-1}\notin[\frac{\delta}{4},2\delta]$, $j, \ell=0,1,2,\ldots$},\\
&\mbox{$g(\tilde x, y, \xi,k)=g^j(\tilde x, y, \xi)=0$ if $\xi_{2n-1}\notin[\frac{\delta}{4},2\delta]$, $j=0,1,2,\ldots$},
\end{split}
\end{equation}
and the operator $\tilde B^{(1)}_{\tau_{k\delta}, s}: \cC^\infty_c(D,T^{*0,1}M')\To\Omega^{0, 1}(U\cap\ol M)$ defined  in \eqref{e-gue220521yyd} satisfies
\begin{equation}\label{e-gue220328yydI}
\begin{split}
&\Box^{(1)}_{s}\tilde B^{(1)}_{\tau_{k\delta}, s}\equiv0\mod O(k^{-\infty})\ \ \mbox{on $(U\times U)\bigcap(\ol M\times X)$},\\
&\tilde B^{(1)}_{\tau_{k\delta}, s}|_X=B^{(1)}_{\tau_{k\delta}, s},
\end{split}
\end{equation}
where $\varphi(\tilde x, y, \xi)$, $g(\tilde x, y, \xi, k)$,  $\varphi_1(\tilde x, y, \xi)$, $g^j_\ell( x, y, \xi)$, $g^j(\tilde x, y, \xi)$, $j, \ell\in\mathbb N_0$, are as in \eqref{21-6-15-a1} and \eqref{e-gue220323yydq}. 
\end{theorem} 

\begin{proof}
Suppose we have constructed $\tilde B_{\tau_{k\delta}, s}^{(1)}$ such that on $D\times U$ one has
\begin{equation}\label{e-gue220328ycd}
\tilde B_{\tau_{k\delta}, s}^{(1)}(\tilde x, y)=\int_{\R^{2n-1}} e^{ik\varphi(\tilde x,y,\xi)}g(\tilde x, y, \xi, k)d\xi.
\end{equation}
Then
$$\Box_{s}^{(1)}\tilde B_{\tau_{k\delta}, s}^{(1)}(\tilde x, y)=\int_{\R^{2n-1}}\Box_s^{(1)}\left[e^{ik\varphi(\tilde x, y, \xi)}g(\tilde x, y, \xi, k)\right]d\xi,$$
where the Laplacian acts on variables $\tilde x$.
We write $$\Box_{s}^{(1)}\left[e^{ik\varphi(\tilde x, y, \xi)}g(\tilde x, y, \xi, k)\right]=e^{ik\varphi(\tilde x, y, \xi)}h(\tilde x, y, \xi, k)$$
and 
\[h(\tilde x, y, \xi, k)\sim k^{2n+1}h^0(\tilde x, y, \xi)+k^{2n}h^1(\tilde x, y, \xi)+\ldots\] in $S^{2n+1}_{{\rm cl\,}}((1+\abs{\xi})^2; ((U\times U)\cap(\ol M\times X))\times\R^{2n-1},T^{*0,1}M'\boxtimes(T^{*0,1}M')^*)$.
We will construct $g(\tilde x, y, \xi, k)$ such that $h^{j}(\tilde x, y, \xi)=O(\rho^{\infty})$, for all $j$. 
By direct calculation,
\begin{equation}\label{e-gue220329yyd}
\begin{split}
&\left(-\frac{\partial}{\partial x_\alpha}+k\frac{\partial\phi}{\partial x_\alpha}\right)\left(\frac{\partial}{\partial x_\beta}+k\frac{\partial\phi}{\partial x_\beta}\right)[e^{ik\varphi}g]\\
&=e^{ik\varphi}\{k^2(\frac{\partial\varphi}{\partial x_\alpha}\frac{\partial\varphi}{\partial x_\beta}+\frac{\partial\phi}{\partial x_\alpha}\frac{\partial\phi}{\partial x_\beta}+i\frac{\partial\phi}{\partial x_\alpha}\frac{\partial\varphi}{\partial x_\beta}-i\frac{\partial\varphi}{\partial x_\alpha}\frac{\partial\phi}{\partial x_\beta})g^0+O(k)\}.
\end{split}
\end{equation}
From \eqref{21-10-6-a1}, \eqref{e-gue220329yyd} and noticing that 
$\frac{\pr\phi}{\pr x_{2n-1}}=0$,  it is straightforward to check that 
\begin{equation}\label{21-6-15-a2}
\begin{split}
&h_0(\tilde x,y,\xi)=g^0\{\frac{1}{2}(\frac{\partial\varphi}{\partial\rho})^2
+\frac{1}{2}(\frac{\partial\phi}{\partial\rho})^2-\frac{\pr\phi}{\pr\rho}\frac{\pr\varphi}{\pr x_{2n-1}}\\
&+\sum_{\alpha, \beta=1}^{2n-1}C_{\alpha,\ol\beta}(\frac{\partial\varphi}{\partial x_\alpha}\frac{\partial\varphi}{\partial x_\beta}+\frac{\partial\phi}{\partial x_\alpha}\frac{\partial\phi}{\partial x_\beta}+i\frac{\partial\phi}{\partial x_\alpha}\frac{\partial\varphi}{\partial x_\beta}-i\frac{\partial\varphi}{\partial x_\alpha}\frac{\partial\phi}{\partial x_\beta})\\
&+\sum_{\alpha=1}^{2n-1}r_{n,\alpha}(\frac{\partial\varphi}{\partial \rho}\frac{\partial\varphi}{\partial x_\alpha}+\frac{\partial\phi}{\partial \rho}\frac{\partial\phi}{\partial x_\alpha}+
i\frac{\partial\phi}{\partial x_\alpha}\frac{\partial\varphi}{\partial\rho}-i\frac{\partial\varphi}{\partial x_\alpha}\frac{\partial\phi}{\partial\rho})\\
&+\sum_{\beta=1}^{2n-1}\hat r_{n,\beta}(\frac{\partial\varphi}{\partial x_\beta}\frac{\partial\varphi}{\partial \rho}+\frac{\partial\phi}{\partial x_\beta}\frac{\partial\phi}{\partial \rho}+
i\frac{\partial\phi}{\partial\rho}\frac{\partial\varphi}{\partial x_\beta}-i\frac{\partial\varphi}{\partial\rho}\frac{\partial\phi}{\partial x_\beta})\}.
\end{split}
\end{equation}
Claim that we can find $\varphi$ such that for all $\xi\in(\delta/4,2\delta)$, 
\begin{equation}\label{e-gue220330yyd}
\begin{split}
&r(\tilde x,y,\xi):=\{\frac{1}{2}(\frac{\partial\varphi}{\partial\rho})^2
+\frac{1}{2}(\frac{\partial\phi}{\partial\rho})^2-\frac{\pr\phi}{\pr\rho}\frac{\pr\varphi}{\pr x_{2n-1}}\\
&+\sum_{\alpha, \beta=1}^{2n-1}C_{\alpha,\ol\beta}(\frac{\partial\varphi}{\partial x_\alpha}\frac{\partial\varphi}{\partial x_\beta}+\frac{\partial\phi}{\partial x_\alpha}\frac{\partial\phi}{\partial x_\beta}+i\frac{\partial\phi}{\partial x_\alpha}\frac{\partial\varphi}{\partial x_\beta}-i\frac{\partial\varphi}{\partial x_\alpha}\frac{\partial\phi}{\partial x_\beta})\\
&+\sum_{\alpha=1}^{2n-1}r_{n,\alpha}(\frac{\partial\varphi}{\partial \rho}\frac{\partial\varphi}{\partial x_\alpha}+\frac{\partial\phi}{\partial \rho}\frac{\partial\phi}{\partial x_\alpha}+
i\frac{\partial\phi}{\partial x_\alpha}\frac{\partial\varphi}{\partial\rho}-i\frac{\partial\varphi}{\partial x_\alpha}\frac{\partial\phi}{\partial\rho})\\
&+\sum_{\beta=1}^{2n-1}\hat r_{n,\beta}(\frac{\partial\varphi}{\partial x_\beta}\frac{\partial\varphi}{\partial \rho}+\frac{\partial\phi}{\partial x_\beta}\frac{\partial\phi}{\partial \rho}+
i\frac{\partial\phi}{\partial\rho}\frac{\partial\varphi}{\partial x_\beta}-i\frac{\partial\varphi}{\partial\rho}\frac{\partial\phi}{\partial x_\beta})\}\\
&=O(\abs{\rho}^N),\ \ \mbox{for every $N\in\mathbb N$}. 
\end{split}
\end{equation}
Until further notice, we will assume that $\xi\in(\delta/4,2\delta)$.
By Taylor expansion with respect to the variable $\rho$, write 
\[\varphi\sim\varphi_0+\rho\varphi_1+\rho^2\phi_2+\ldots,~~\phi\sim\phi_0+\rho\phi_1+\rho^2\phi_2+\ldots,\ \ r(\tilde x,y,\xi)\sim r_0+\rho r_1+\rho^2r_2+\ldots,\]
where $\varphi_j$, $\phi_j$ and $r_j$ are independent of $\rho$, for every $j\in\mathbb N_0$. 
We can check that 
\begin{equation}\label{e-gue220329yyda}
\begin{split}
r_0=(&\frac{1}{2}\varphi_1^2+\frac{1}{2}\phi_1^2+\sum_{\alpha, \beta=1}^{2n-1}C_{\alpha,\ol\beta}\frac{\partial\varphi_0}{\partial x_\alpha}\frac{\partial\varphi_0}{\partial x_\beta}+\sum_{\alpha, \beta=1}^{2n-1} C_{\alpha,\ol\beta}\frac{\partial\phi_0}{\partial x_\alpha}\frac{\partial\phi_0}{\partial x_\beta}-\phi_1\frac{\pr\varphi_0}{\pr x_{2n-1}}\\
&+i\sum_{\alpha, \beta=1}^{2n-1}C_{\alpha,\ol\beta}\frac{\partial\phi_0}{\partial x_\alpha}\frac{\partial\varphi_0}{\partial x_\beta}
-i\sum_{\alpha, \beta=1}^{2n-1}C_{\alpha,\ol\beta}\frac{\partial\varphi_0}{\partial x_\alpha}\frac{\partial\phi_0}{\partial x_\beta}).
\end{split}
\end{equation}
Notice that ${\rm Re\,}[C_{\alpha,\overline\beta}]>0$
and $d_x\varphi_0$ is non-zero vector on $X$ for $d_x\varphi_0=\xi$ 
with $\xi_{2n-1}\in(\frac{\delta}{4}, 2\delta)$. 
Thus, we have $$\sum_{\alpha, \beta=1}^{2n-1}C_{\alpha,\ol\beta}
\frac{\partial\varphi}{\partial x_\alpha}\frac{\partial\varphi}{\partial x_\beta}>0.$$
On the other hand, we can choose an $\R$-equivariant holomorphic frame
$s$ of $L$ such that $\frac{\pr\phi_0}{\pr x_\alpha}(p)=0$, $\alpha=1,\ldots,2n-1$.
From this observation and \eqref{e-gue220517yyd}, we have 
\begin{equation}
\sum_{\alpha, \beta=1}^{2n-1}C_{\alpha,\ol\beta}(\frac{\partial\varphi_0}{\partial x_\alpha}\frac{\partial\varphi_0}{\partial x_\beta}+\frac{\partial\phi_0}{\partial x_\alpha}\frac{\partial\phi_0}{\partial x_\beta}+i \frac{\partial\phi_0}{\partial x_\alpha}\frac{\partial\varphi_0}{\partial x_\beta}
-i \frac{\partial\varphi_0}{\partial x_\alpha}\frac{\partial\phi_0}{\partial x_\beta})-\phi_1\frac{\pr\varphi_0}{\pr x_{2n-1}}>C\abs{\xi}^2
\end{equation}
when $U$ is sufficiently small, $C>0$ is a constant. 
We choose
\begin{equation}\label{e-gue220426yyd}
\varphi_1=-i\sqrt{\phi_1^2+2\sum_{\alpha, \beta=1}^{2n-1}C_{\alpha,\ol\beta}(\frac{\partial\varphi}{\partial x_\alpha}\frac{\partial\varphi}{\partial x_\beta}+\frac{\partial\phi_0}{\partial x_\alpha}\frac{\partial\phi_0}{\partial x_\beta}+i \frac{\partial\phi_0}{\partial x_\alpha}\frac{\partial\varphi}{\partial x_\beta}
-i \frac{\partial\varphi}{\partial x_\alpha}\frac{\partial\phi_0}{\partial x_\beta})-2\phi_1\frac{\pr\varphi_0}{\pr x_{2n-1}}}
\end{equation}
and hence $r_0(x,y,\xi)=0$. It is not difficult to check that there is a smooth function $a(x,y,\xi)\in\cC^\infty(D\times D\times\R^{2n-1})$ independent of $\rho$ and $\varphi_j$, $j\geq2$, such that
\[r_1=2\varphi_1\varphi_2+a.\]
From the vanishing of this term and since $\varphi_1\neq 0$, we can solve $\varphi_2$. Thus, by the same argument, $\{\varphi_j\}_{j\geq 1}$ are solved and so is $\varphi$.

From \eqref{21-6-15-a2}, we see that $h^0=O(\abs{\rho}^{+\infty})$. By Taylor expansion with respect to the variable $\rho$, for every $j\in\mathbb N_0$, write 
\begin{equation}\label{e-gue220329yydb}
h^{j}(\tilde x, y, \xi)=h^j_0(x,y,\xi)+\rho h^j_1(x,y,\xi)+\rho^2h^j_0(x,y,\xi)+\ldots.
\end{equation}
From the expansion of $\Box_s^{(1)}(e^{ik\varphi} g)$, it is not difficult to see that there is a smooth function $b(x,y,\xi)\in\cC^\infty(D\times D\times\R^{2n-1},T^{*0,1}M'\boxtimes(T^{*0,1}M')^*)$ independent of $\rho$ and $g^0_j$, $j\geq1$, such that
\[h^1_0=-2i\varphi_1g^0_1+b.\]
From the vanishing of this term and since $\varphi_1\neq 0$, we can solve $g^0_1$. Thus, by the same argument, we  can find $\{g^0_j\}_{j\geq 1}$ and $g^0$ such that 
$h^1=O(\abs{\rho}^{+\infty})$. Continuing in this way, we get the theorem.
\end{proof}
\begin{lemma}\label{l-gue220330yyd}
Let $q\in\set{0,1,\ldots.n-1}$ and
$\chi, \hat\chi\in\cC^\infty(X)$ with 
\(\supp\chi\cap\supp\hat\chi=\emptyset.\) Then 
\[\hat\chi F^{(q)}_{\tau_{k\delta}}\chi=O(k^{-\infty}): 
W^\ell(X,T^{*0,q}M'\otimes L^k)\To W^\ell(X,T^{*0,q}M'\otimes L^k),\]
for all $\ell\in\mathbb Z$. 
\end{lemma} 

\begin{proof} 
We may assume that $\chi\in\cC^\infty_c(D)$, $\hat\chi\in\cC^\infty_c(D_1)$, $D, D_1$ are open canonical coordinate patches of $X$, $D\bigcap D_1=\emptyset$, 
and there are $\R$-equivariant CR frames $s$ and $s_1$ defined on $D$ and $D_1$ respectively. 
As before, let 
\[F^{(q)}_{\tau_{k\delta},s,s_1}: \cC^\infty_c(D,T^{*0,q}M')\To\cC^\infty(D_1,T^{*0,q}M')\]
be the localization of $F^{(q)}_{\tau_{k\delta}}$ with respect to $s_1$ and $s$. We write $x=(x_1,\ldots,x_{2n-1})$ to denote canonical local coordinates of $D$ and $D_1$. 
Let $u\in\cC^\infty_c(D,T^{*0,q}M')$. 
It is easy to see that on $D_1$,
\begin{equation}\label{e-gue131217}
\begin{split}
&F^{(q)}_{\tau_{k\delta},s,s_1}u(y)\\
&=\sum_{(m_1,\ldots,m_d)\in\mathbb Z^d}\tau_\delta\Bigr(\frac{\sum^d_{j=1}m_j\beta_j}{k}\Bigr)e^{i(\sum^d_{j=1}m_j\beta_j) y_{2n-1}}\\
&\quad\times\int_{T^d}e^{-(im_1\theta_1+\ldots+im_d\theta_d)}
u((e^{i\theta_1},\ldots,e^{i\theta_d})\circ y')dT_d,
\end{split}
\end{equation}
where $y'=(y_1,\ldots,y_{2n-2},0)$, $dT_d=(2\pi)^{-d}d\theta_1\ldots d\theta_d$ 
and $\beta_1\in\R,\ldots,\beta_d\in\R$ are as in \eqref{e-gue220330ycd}. 
Fix $D'\Subset D_1$ and let $\gamma(y_{2n-1})\in
\cC^\infty_c((-\varepsilon,\varepsilon))$ such that 
$\gamma(y_{2n-1})=1$ for every $(y',y_{2n-1})\in D'$. 
Let $R_{k}:\cC^\infty(X,T^{*0,q}M')\To\cC^\infty(D',T^{*0,q}M')$ 
be the continuous operator given by
\begin{equation}\label{e-gue131217I}
\begin{split}
&(R_kv)(x)=\\
&\frac{1}{2\pi}\sum_{(m_1,\ldots,m_d)\in\mathbb Z^d}\:
\int\limits_{T^d}e^{i\langle x_{2n-1}-y_{2n-1},\xi_{2n-1}\rangle
+i(\sum^d_{j=1}m_j\beta_j)y_{2n-1}-im_1\theta_1
-\ldots-im_d\theta_d}\\
&\times\tau_\delta\Bigr(\frac{\xi_{2n-1}}{k}\Big)
(1-\gamma(y_{2n-1}))v((e^{i\theta_1},\ldots,e^{i\theta_d})\circ x')
dT_dd\xi_{2n-1}dy_{2n-1},
\end{split}
\end{equation}
where $v\in\cC^\infty(X,T^{*0,q}M')$. 
We claim that 
\begin{equation}\label{e-gue170831}
R_k=O(k^{-\infty}): W^s(X,T^{*0,q}M')\To 
W^s_{{\rm loc\,}}(D',T^{*0,q}M'),\ \ \forall s\in\mathbb N_0.
\end{equation}
We only prove the claim \eqref{e-gue170831} for $s=0$. 
For any $s\in\mathbb N$, the proof is similar. 
Let $v\in\cC^\infty(X,T^{*0,q}M')$. 
By using integration by parts with respect to $\xi_{2n-1}$ several times, 
it is straightforward to check that for every $N\in\mathbb N$, 
there is a constant $C_N>0$ independent of $k$ such that 
\begin{equation}\label{e-gue170831I}
\begin{split}
&\int_{D'}\abs{R_kv}^2(x)\\
&\leq C_Nk^{-N}\sum_{(m_1,\ldots,m_d)\in\mathbb Z^d}
\int_{D'}\abs{\int_{T_d}v((e^{i\theta_1},\ldots,e^{i\theta_d})\circ x)
e^{-im_1\theta_1-\ldots-im_d\theta_d}dT_d}^2dv_X(x)\\
&\leq C_Nk^{-N}\norm{v}^2_X.
\end{split}
\end{equation}
From \eqref{e-gue170831I}, we get \eqref{e-gue170831} for $s=0$. 
For any $s\in\N$, the proof is similar. 
Now, we claim that
\begin{equation}\label{e-gue131217III}
\hat\chi R_k=\hat\chi F_{k, \delta,s,s_1}\ \ \mbox{on $\cC^\infty_c(D,T^{*0,q}M')$}.
\end{equation} 
Let $u\in\cC^\infty_c(D,T^{*0,q}M')$. On $D_1$, let 
\begin{equation}\label{e-gue131217IVz}
\begin{split}
&B_{k}u(x)\\
&:=\frac{1}{2\pi}\sum_{(m_1,\ldots,m_d)\in\mathbb Z^d}
\int e^{i\langle x_{2n-1}-y_{2n-1},\xi_{2n-1}\rangle}
\tau_\delta\Big(\frac{\xi_{2n-1}}{k}\Big)\gamma(y_{2n-1})\\
&\quad\times e^{i(\sum^d_{j=1}m_j\beta_j)y_{2n-1}-im_1\theta_1-\ldots-im_d\theta_d}u((e^{i\theta_1},\ldots,e^{i\theta_d})\circ x')dT_ddy_{2n-1}d\xi_{2n-1}.
\end{split}
\end{equation}
From Fourier inversion formula and notice that $D\bigcap D_1=\emptyset$, it is straightforward to see that
\begin{equation}\label{e-gue131217IV}
\begin{split}
&\hat\chi(x)B_{k}u(x)\\
&=\hat\chi(x)\frac{1}{2\pi}\int e^{i\langle x_{2n-1}-y_{2n-1},\xi_{2n-1}\rangle}
\tau_\delta\Big(\frac{\xi_{2n-1}}{k}\Big)\gamma(y_{2n-1})u(x',y_{2n-1}))dT_ddy_{2n-1}d\xi_{2n-1}\\
&=0.
\end{split}
\end{equation}
From \eqref{e-gue131217IVz} and \eqref{e-gue131217I}, we have
\begin{equation}\label{e-gue131217V}
\begin{split}
&(B_{k}+R_{k})u(x)\\
&=\frac{1}{2\pi}\sum_{(m_1,\ldots,m_d)\in\mathbb Z^d}
\int e^{i\langle x_{2n-1}-y_{2n-1},\xi_{2n-1}\rangle}
\tau_\delta\Big(\frac{\xi_{2n-1}}{k}\Big)\\
&\quad\times e^{i(\sum^d_{j=1}m_j\beta_j)y_{2n-1}-im_1\theta_1-\ldots-im_d\theta_d}u((e^{i\theta_1},\ldots,e^{i\theta_d})\circ x')dT_ddy_{2n-1}d\xi_{2n-1}.
\end{split}\end{equation}
Note that the following formula holds for every ${t\in\R}$,
\begin{equation}\label{dm}
\int e^{i{t} y_{2n-1}}e^{-iy_{2n-1}\xi_{2n-1}}dy_{2n-1}=2\pi\delta_{t}(\xi_{2n-1}),
\end{equation}
where the integral is defined as an oscillatory integral and $\delta_{t}$ is the Dirac measure at ${t}$.
Using \eqref{e-gue131217}, \eqref{dm} and the Fourier inversion formula,
 \eqref{e-gue131217V} becomes
\begin{equation}\label{e-gue131217VI}
\begin{split}
&(B_{k}+R_k)u(x)\\
&=\sum_{(m_1,\ldots,m_d)\in\mathbb Z^d}\tau_\delta\Big(\frac{\sum^d_{j=1}m_j\beta_j}{k}\Big)e^{i(\sum^d_{j=1}m_j\beta_j)x_{2n-1}}
\int_{T_d}e^{-im_1\theta_1-\ldots-im_d\theta_d}u((e^{i\theta_1},\ldots,e^{i\theta_d})\circ x')dT_d\\
&=F_{k, \delta,s,s_1}u(x).
\end{split}\end{equation}
From \eqref{e-gue131217IV} and \eqref{e-gue131217VI}, the claim \eqref{e-gue131217III} follows.
From \eqref{e-gue170831}, \eqref{e-gue131217III} and taking adjoint, the lemma follows.
\end{proof}
 
 Recall $U$ is an open subset of $M'$ and $D=U\cap X$. We define 
\begin{equation}\label{22-1-12-a1}
\begin{split}
&B_{\tau_{k\delta}}^{(1)}=e^{k\phi}s^k B_{\tau_{k\delta}, s}^{(1)}(s^{-k}e^{-k\phi}): \cC^\infty_c(D,T^{*0,1}M'\otimes L^k)\rightarrow \cC^\infty(D,T^{*0,1}M'\otimes L^k),\\
&\tilde{B}_{\tau_k\delta}^{(1)}=e^{k\phi}s^k \tilde B_{\tau_{k\delta}, s}^{(1)}(s^{-k}e^{-k\phi}):  \cC^\infty_c(D,T^{*0,1}M'\otimes L^k)\rightarrow \Omega^{0, 1}(U\cap\ol M, L^k).
\end{split}
\end{equation}
Note that $B^{(1)}_{\tau_{k\delta}, s}$ and $\tilde B^{(1)}_{\tau_{k\delta}, s}$ are actually the localizations of $B^{(1)}_{\tau_{k\delta}}$ and $\tilde B^{(1)}_{\tau_{k\delta}}$ with respect to $s$ over $D$ and $U$ respectively. 

\begin{theorem}\label{t-gue220406yyd}
Let $\chi\in\cC^\infty_c(D)$, $\chi_1\in\cC^\infty_c(U\cap\ol M)$ with $\chi_1=1$ near $\supp\chi$. We have 
\begin{equation}\label{e-gue220406yyd}
P_{k}^{(1)}F_{\tau_{k\delta}}^{(1)}\chi-\chi_1\tilde{B}_{\tau_{k\delta}}^{(1)}\chi=O(k^{-\infty}): W^{t}(X, T^{\ast 0, 1}M'\otimes L^{k}) \rightarrow W^{t+\frac{1}{2}}(\overline M, T^{\ast 0, 1}M'\otimes L^{k}), \forall t\in\mathbb Z. 
\end{equation}
\end{theorem}
 
 \begin{proof}
 From \eqref{e-gue190418yyda}, we have 
 \begin{equation}\label{e-gue220406yydI}
 N^{(1)}_k\Box^{(1)}_{f,k}\chi_1\tilde B^{(1)}_{\tau_{k\delta}}\chi+P^{(1)}_k\gamma\chi_1\tilde B^{(1)}_{\tau_{k\delta}}\chi=\chi_1\tilde B^{(1)}_{\tau_{k\delta}}\chi.
 \end{equation} 
 From \eqref{e-gue220328yydI} and the construction of $\tilde B^{(1)}_{\tau_{k\delta}}$, we see that 
 \begin{equation}\label{e-gue220409yyd}
 \begin{split}
 &\mbox{$[\Box^{(1)}_{f,k},\chi_1]\tilde B^{(1)}_{\tau_{k\delta}}\chi\equiv0\mod O(k^{-\infty})$ on $(U\times U)\bigcap(\ol M\times X)$},\\
 &\mbox{$\Box^{(1)}_{f,k}\tilde B^{(1)}_{\tau_{k\delta}}\chi\equiv0\mod O(k^{-\infty})$ on $(U\times U)\bigcap(\ol M\times X)$}.
 \end{split}
 \end{equation}
 Let $\hat\chi_1:=\gamma\chi_1$. From  \eqref{e-gue220314yydII}, \eqref{e-gue220314yydIIa}, Lemma~\ref{l-gue131209}, \eqref{e-gue220406yydI} and \eqref{e-gue220409yyd}, we get 
 \begin{equation}\label{e-gue220409yydII}
 P^{(1)}_k\hat\chi_1F^{(1)}_{\tau_{k\delta}}\chi-\chi_1\tilde B^{(1)}_{\tau_{k\delta}}\chi=O(k^{-\infty}):  W^\ell(X,T^{*0,1}M'\otimes L^k)\To W^{\ell+\frac{1}{2}}(\ol M,T^{*0,1}M'\otimes L^k),
 \end{equation}
 for every $\ell\in\R$. From \eqref{e-gue220314yydIIa}, Lemma~\ref{l-gue220330yyd} and \eqref{e-gue220409yydII}, we get \eqref{e-gue220406yyd}.  
 \end{proof}
 
 We now introduce the operator $R^{(q)}$ acting on $(0,q)$-forms. 

\begin{definition}\label{d-gue220410yyd}
The operator $R^{(q)}$ acting on $(0,q)$-forms is given by 
\begin{equation}\label{e-gue220410yyd}
\begin{split}
R^{(q)}: \Omega^{0,q}(M',L^k)&\To\Omega^{0,q}(M',L^k),\\
u&\mapsto\Bigr(2((\ddbar\rho)\wedge)^*\circ\ddbar_k+2
\ddbar_k\circ((\ddbar\rho)\wedge)^*-iT\Bigr)u.
\end{split}
\end{equation}
\end{definition}

Let 
\[(R^{(q)})^*: \cC^\infty(M',T^{*0,1}M'\otimes L^k)
\To\cC^\infty(M',T^{*0,1}M'\otimes L^k)\]
be the adjoint of $R^{(q)}$ with respect to $(\,\cdot\,|\,\cdot\,)_{M',k}$. 
As before, let $s$ be a $\R$-equivariant holomorphic frame of 
$L$ defined on an open set $U$ of $M'$ with $D:=U\cap X\neq\emptyset$. 
Let $\{e^{\ol j}\}^n_{j=1}$ be an orthonormal basis for 
$T^{*0,1}M'$ with $e^{\ol n}=\frac{\ddbar\rho}{|\ddbar\rho|}$. 
Let $u\in\Omega^{0,1}_c(U,L^k)$. 
Write $u=s^k\otimes\sum^n_j u_je^{\ol j}$. It is straightforward to check that 
\begin{equation}\label{e-gue220410yydI}
\begin{split}
(R^{(1)})^*u&=s^k\otimes\sum^n_{j=1}((iJ(\frac{\pr}{\pr\rho})-\frac{\pr}{\pr\rho}-iT)u_j)e^{\ol j}\\
&+2ks^k\otimes\sum^n_{j=1}\bigr((iJ(\frac{\pr}{\pr\rho})+\frac{\pr}{\pr\rho}-iT)\phi\bigr)u_je^{\ol j}+
s^k\otimes Z_0\Bigr(\sum^n_{j=1}u_je^{\ol j}\Bigr),
\end{split}
\end{equation}
where $Z_0$ is a zero order differential operator, $Z_0$ is independent of $k$. Put 
\begin{equation}\label{e-gue220410yydII}
\mathcal{N}^{(1)}_{\tau_{k\delta}}:=\gamma(R^{(1)})^*P^{(1)}_kF^{(1)}_{\tau_{k\delta}}: \cC^\infty(X,T^{*0,1}M'\otimes L^k)\To\Omega^{0,1}(X,T^{*0,1}M'\otimes L^k).
\end{equation}
Let $\mathcal{N}^{(1)}_{\tau_{k\delta},s}$ be the localization of $\mathcal{N}^{(1)}_{\tau_{k\delta}}$ with respect to $s$. 

\begin{lemma}\label{l-gue220410yyd}
Let $\chi\in\cC^\infty_c(D)$. We have 
\begin{equation}\label{e-gue220412yyd}
\mathcal{N}^{(1)}_{\tau_{k\delta}}\chi=\tilde\chi_1\widehat{\mathcal{N}}^{(1)}_{\tau_{k\delta}}\chi+\varepsilon^{(1)}_k,
\end{equation}
where $\tilde\chi_1\in\cC^\infty_c(D)$, $\tilde\chi_1\equiv1$ near $\supp\chi$, $\widehat{\mathcal{N}}^{(1)}_{\tau_{k\delta}}$ is a continuous operator 
\[\widehat{\mathcal{N}}^{(1)}_{\tau_{k\delta}}: \cC^\infty_c(D,T^{*0,1}M'\otimes L^k)\To{\cC^\infty}(D,T^{*0,1}M'\otimes L^k)\]
 with 
 \begin{equation}\label{e-gue220410yydIII}
 \begin{split}
& \widehat{\mathcal{N}}^{(1)}_{\tau_{k\delta},s}={\rm Op\,}_k(a),\\
 &a\in S^{2n}_{{\rm cl\,}}((1+\abs{\xi});D\times D\times\R^{2n-1},T^{*0,1}M'\boxtimes(T^{*0,1}M')^*),\\
 &\sigma^0_a(x,y,\xi)=(-i\varphi_1(x,\xi)+\frac{\pr\phi}{\pr\rho}(x))\tau_\delta(\xi_{2n-1}),
 \end{split}
 \end{equation}
 where $\varphi_1$ is as in \eqref{e-gue220328yyd}, $\widehat{\mathcal{N}}^{(1)}_{\tau_{k\delta},s}$ is the localization of $\widehat{\mathcal{N}}^{(1)}_{\tau_{k\delta}}$ with respect to $s$ and $\varepsilon^{(1)}_k$ is continuous: 
 \begin{equation}\label{e-gue220411yyd}
 \varepsilon^{(1)}_k=O(k^{-\infty}): W^{\ell}(X,T^{*0,1}M'\otimes L^k)\To W^{\ell-1}(X,T^{*0,1}M'\otimes L^k),\ \ \forall\ell\in\R. 
 \end{equation}
 \end{lemma}
 
 \begin{proof}
 Note that the operator $\gamma(R^{(1)})^*P^{(1)}_k$ has the following property: for every $\ell\in\R$, there is a $n_\ell\in\mathbb N$ such that 
 \begin{equation}\label{e-gue220411yydI}
 \gamma(R^{(1)})^*P^{(1)}_k=O(k^{n_\ell}):  W^\ell(X,T^{*0,1}M'\otimes L^k)\To W^{\ell-1}(X,T^{*0,1}M'\otimes L^k).
 \end{equation}
 From \eqref{e-gue220406yyd}, \eqref{e-gue220410yydI} and \eqref{e-gue220411yydI}, we conclude that 
 $\varepsilon^{(1)}_k:=\widehat{\mathcal{N}}^{(1)}_{\tau_{k\delta}}\chi-\gamma(R^{(1)})^*\chi_1\tilde B^{(1)}_k\chi$ satisfies \eqref{e-gue220411yyd}. 
 From this observation, Theorem~\ref{21-5-19-thm1} and notice that 
 $J(\frac{\pr}{\pr\rho})=T+O(\rho)$, the lemma follows. 
 \end{proof} 
 
 We need 
 
 \begin{lemma}\label{l-gue220411yyd}
 Let $A_k={\rm Op\,}_k(a):\cC^\infty_c(D,T^{*0,1}M')\To\cC^\infty(D,T^{*0,1}M')$ with $a\in S^{2n}_{{\rm cl\,}}((1+\abs{\xi});D\times D\times\R^{2n-1},T^{*0,1}M'\boxtimes(T^{*0,1}M')^*)$. Let $\chi, \tilde\chi, \chi_1\in\cC^\infty_c(D)$, $\tilde\chi=1$ near $\supp\chi$.There is a constant $C>0$ independent of $k$ such that 
 \begin{equation}\label{e-gue220411yyda}
 \norm{\chi_1A_k\chi u}_X\leq C\Bigr(\norm{\tilde\chi u}_{1,X}+k\norm{\tilde\chi u}_X\Bigr), 
 \end{equation}
 for every $u\in\cC^\infty_c(D,T^{*0,1}M')$, $k\gg1$. 
 \end{lemma} 
 
 \begin{proof}
 Let $\tau\in\cC^\infty(\R_+)$, $\tau=1$ on $[\frac{1}{2},+\infty[$, $\tau=0$ 
 on $(-\infty,\frac{1}{4}]$, $0\leq\tau\leq1$. Let $\hat D:=D\times\R$. 
 Let $\hat x=(x_1,\ldots,x_{2n-1},x_{2n})$ be local coordinates of $\hat D$, where $x_{2n}$ is the coordinate of $\R$. Let $\hat A: \cC^\infty_c(\hat D,T^{*0,1}M')\To\cC^\infty(\hat D,T^{*0,1}M')$ be the continuous operator with distribution kernel 
 \[\begin{split}
& \hat A(x,y)=(2\pi)^{-(2n)}\int e^{i\langle\,\hat x-\hat y,\hat\xi\,\rangle}b(\hat x,\hat y,\hat\xi)d\hat\xi,\\
 &b(\hat x,\hat y,\hat\xi):=\xi^{-(2n-1)}_{2n}a(x,y,\frac{\xi}{\xi_{2n}},\xi_{2n})\tau(\xi_{2n})\tau_1(x_{2n})\hat\tau_1(y_{2n}),\end{split}\]
 where $\hat\xi=(\xi_1,\ldots,\xi_{2n})$, $\tau_1, \hat\tau_1\in\cC^\infty_c(\R)$, $\tau_1\equiv1$ near $0$, $\hat\tau_1\equiv1$ near $\supp\tau_1$. 
 We can check that for every $\alpha_1, \alpha_2, \alpha_3\in\mathbb N^{2n}_0$ and every compact set $K\subset\hat D\times\hat D$, there is a constant $C>0$ such that 
 \begin{equation}\label{e-gue220411yydb}
 \abs{\pr^{\alpha_1}_{\hat x}\pr^{\alpha_2}_{\hat y}\pr^{\alpha_3}_{\hat\xi}\Bigr(b(\hat x,\hat y,\hat\xi)(1+\abs{\hat\xi})^{-1}\Bigr)}\leq C,
 \end{equation} 
 for every $(\hat x,\hat y,\hat\xi)\in K\times\R^{2n}$. From \eqref{e-gue220411yydb} and Calderon-Vaillancourt theorem (see~\cite[Chapter XVIII]{Hor85}), we deduce that 
 \begin{equation}\label{e-gue220411yydc}
 \mbox{$\hat A: W^\ell_{{\rm comp\,}}(\hat D,T^{*0,1}M')\To  W^{\ell-1}_{{\rm loc\,}}(\hat D,T^{*0,1}M')$ is continuous},
 \end{equation}
 for every $\ell\in\R$. Let $u\in\cC^\infty_c(D,T^{*0,1}M')$ and put $\hat u:=\tilde\tau_1(x_{2n})e^{ikx_{2n}}\chi(x)u(x)\in\cC^\infty_c(\hat D,T^{*0,1}M')$, 
 where $\tilde\tau_1\in\cC^\infty_c(\R)$, $\tilde\tau_1\equiv1$ near $\supp\hat\tau_1$. From Fourier inversion formula, it is straightforward to check that 
 \begin{equation}\label{e-gue220411ycda}
 \begin{split}
&\hat A(\hat u)(\hat x)=\tau_1(x_{2n})e^{ikx_{2n}}A_k(\chi u)(x)+(2\pi)^{-(2n)}(R_k\hat u)(\hat x), \\
&(R_k\hat u)(\hat x)\\
&=\int  e^{i\langle\,\hat x-\hat y,\hat\xi\,\rangle}\xi^{-(2n-1)}_{2n}a(x,y,\frac{\xi}{\xi_{2n}},\xi_{2n})\tau(\xi_{2n})\tau_1(x_{2n})(1-\hat\tau_1(y_{2n}))e^{iky_{2n}}u(y)d\hat yd\hat\xi.
 \end{split}
 \end{equation} 
 By using integration by parts with respect to $\xi_{2n}$ and $y_{2n}$, we see that for every $N\in\mathbb N$, there is a constant $C_N>0$ independent of $k$ and $u$ such that 
 \begin{equation}\label{e-gue220411ycdb}
 \norm{\chi_1R_k\hat u}_{\hat D}\leq C_Nk^{-N}\norm{\tilde\chi u}_{1,X},
 \end{equation}
where $\norm{\cdot}_{\hat D}$ denotes the $L^2$ norm on $\hat D$ induced by $dv_Xdx_{2n}$. From \eqref{e-gue220411yydc}, \eqref{e-gue220411ycda} and \eqref{e-gue220411ycdb}, we get \eqref{e-gue220411yyda}.  
   \end{proof} 
   
From Lemma~\ref{l-gue220411yyd}, we get 

 \begin{lemma}\label{l-gue220412yyd}
 Let $B_k={\rm Op\,}_k(b):\cC^\infty_c(D,T^{*0,1}M')\To\cC^\infty(D,T^{*0,1}M')$ with $b\in S^{2n-1}_{{\rm cl\,}}((1+\abs{\xi});D\times D\times\R^{2n-1},T^{*0,1}M'\boxtimes(T^{*0,1}M')^*)$. Let $\chi, \tilde\chi, \chi_1\in\cC^\infty_c(D)$, $\tilde\chi=1$ near $\supp\chi$.There is a constant $C>0$ independent of $k$ such that 
 \begin{equation}\label{e-gue220411yydaz}
 \norm{\chi_1B_k\chi u}_X\leq\frac{C}{k}\Bigr(\norm{\tilde\chi u}_{1,X}+k\norm{\tilde\chi u}_X\Bigr), 
 \end{equation}
 for every $u\in\cC^\infty_c(D,T^{*0,1}M')$, $k\gg1$. 
 \end{lemma} 

Now, we change $\tau_\delta$ to $\tau^2_\delta$. 
Recall that $\tau_\delta$ is given by \eqref{e-21-5-16-a2z}.  
\begin{theorem}\label{t-gue220411yyd}
With the notations used above, let $\chi\in\cC^\infty_c(D)$. There is a constant $C_N>0$ independent of $k$ such that 
\begin{equation}\label{e-gue220411yydm}                                                                                                        
(\,\mathcal{N}^{(1)}_{\tau^2_{k\delta}}\chi^2u\,|\,u\,)_{X,k}\leq\frac{C}{k}\Bigr(\norm{u}_{1,X,k}\norm{u}_{X,k}+k\norm{u}^2_{X,k}+\frac{1}{k^2}\norm{u}^2_{1,X,k}\Bigr),
\end{equation}
for every $u\in\cC^\infty(X,T^{*0,1}M'\otimes L^k)$, $k\gg1$.
\end{theorem}

\begin{proof}
From Lemma~\ref{l-gue220410yyd}, we have 
\begin{equation}\label{e-gue220411ycds}
\mathcal{N}^{(1)}_{\tau^2_{k\delta}}\chi=\tilde\chi_1\mathcal{N}^{(1)}_{\tau^2_{k\delta}}\chi+\varepsilon^{(1)}_k,
\end{equation} 
where $\varepsilon^{(1)}_k$ satisfies \eqref{e-gue220411yyd}, $\widehat{\mathcal{N}}^{(1)}_{\tau_{k\delta},s}={\rm Op\,}_k(a)$, 
$\sigma^0_a(x,y,\xi)=(-i\varphi_1(x,\xi)+\frac{\pr\phi}{\pr\rho}(x))\tau^2_\delta(\xi_{2n-1})$. From \eqref{e-gue220328yyd}, we can find 
$h(x,y,\xi)\in S^{0}_{{\rm cl\,}}((1+\abs{\xi});D\times D\times\R^{2n-1},T^{*0,1}M'\boxtimes(T^{*0,1}M')^*)$ such that $(h(x,y,\xi))^2=-\sigma^0_a(x,y,\xi)$. 
From this observation and standard WKB method, we can find $B_k={\rm Op\,}_k(b)$, $b\in S^{2n-\frac{1}{2}}_{{\rm cl\,}}((1+\abs{\xi});D\times D\times\R^{2n-1},T^{*0,1}M'\boxtimes(T^{*0,1}M')^*)$, $B_k$ is properly supported on $D$, such that 
\begin{equation}\label{e-gue220411ycdh}
\chi_1\widehat{\mathcal{N}}^{(1)}_{\tau_{k\delta},s}\chi^2+\chi_1B^*_kB_k\chi^2=\chi_2{\rm Op\,}_k(g), 
\end{equation}
$g\in S^{2n-1}_{{\rm cl\,}}((1+\abs{\xi});D\times D\times\R^{2n-1},T^{*0,1}M'\boxtimes(T^{*0,1}M')^*)$, ${\rm Op\,}_k(g)$ is properly supported on $D$, where $B^*_k$ is the formal adjoint of $B_k$ 
with respect to $(\,\cdot\,|\,\cdot\,)_X$, $\chi_2\in\cC^\infty_c(D)$, $\chi_2\equiv1$ near $\supp\chi_1$. Let $v\in\cC^\infty(D,T^{*0,1}M')$. We have 
\begin{equation}\label{e-gue220411ycdp}
\begin{split}
&(\,\chi_1\widehat{\mathcal{N}}^{(1)}_{\tau_{k\delta},s}\chi^2v\,|\,v)_X=-(\,\chi_1B^*_kB_k\chi^2v\,|\,v\,)_X+(\,\chi_2{\rm Op\,}_k(g)v\,|\,v\,)_X\\
&=-(\,\chi_1\chi B^*_kB_k\chi v\,|\,v\,)_X-(\,\chi_1[B^*_kB_k,\chi]\chi v\,|\,v\,)_X+(\,\chi_2{\rm Op\,}_k(g)v\,|\,v\,)_X\\
&=-\norm{B_k\chi v}^2_X-(\,\chi_1[B^*_kB_k,\chi]\chi v\,|\,v\,)_X+(\,\chi_2{\rm Op\,}_k(g)v\,|\,v\,)_X.
\end{split}
\end{equation}
By using integration by parts, it is not difficult to see that $\chi_1[B^*_kB_k,\chi]=\chi_1{\rm Op\,}_k(h)$, $h\in S^{2n-1}_{{\rm cl\,}}((1+\abs{\xi});D\times D\times\R^{2n-1},T^{*0,1}M'\boxtimes(T^{*0,1}M')^*)$, ${\rm Op\,}_k(h)$ is properly supported on $D$. From Lemma~\ref{l-gue220412yyd}, there is a constant $C>0$ independent of $k$ and $v$ such that 
\begin{equation}\label{e-gue220411ycdq}
\abs{(\,\chi_1[B^*_kB_k,\chi]\chi v\,|\,v\,)_X}+\abs{(\,\chi_2{\rm Op\,}_k(g)v\,|\,v\,)_X}\leq\frac{C}{k}\Bigr(\norm{\tilde\chi v}_{1,X}\norm{\tilde\chi v}_{X}
+k\norm{\tilde\chi v}^2_{X}\Bigr),
\end{equation}
where $\tilde\chi\in\cC^\infty_c(D)$, $\tilde\chi\equiv1$ near 
$\supp\chi\cup\supp\chi_2$. From \eqref{e-gue220411ycdp} and 
\eqref{e-gue220411ycdq}, we conclude that 
\begin{equation}\label{e-gue220411ycdr}
(\,\chi_1\widehat{\mathcal{N}}^{(1)}_{\tau_{k\delta},s}
\chi^2v\,|\,v)_X\leq\frac{C}{k}\Bigr(\norm{\tilde\chi v}_{1,X}\norm{\tilde\chi v}_{X}
+k\norm{\tilde\chi v}^2_{X}\Bigr),
\end{equation} 
for every $v\in\cC^\infty(D,T^{*0,1}M')$. 
From \eqref{e-gue220411ycds} and \eqref{e-gue220411ycdr}, 
we get \eqref{e-gue220411yydm}. The theorem follows. 
\end{proof}

From Theorem~\ref{t-gue220411yyd} and by using integration by parts, 
we get the following. 

\begin{theorem}\label{t-gue220411yydz}
With the notations used above, there is a constant $C>0$ independent of $k$ such that 
\begin{equation}\label{e-gue220411yydmz}                                                                                                        
(\,\mathcal{N}^{(1)}_{\tau^2_{k\delta}}u\,|\,u\,)_{X,k}\leq\frac{C}{k}\Bigr(\norm{u}_{1,X,k}\norm{u}_{X,k}+k\norm{u}^2_{X,k}+\frac{1}{k^2}\norm{u}^2_{1,X,k}\Bigr),
\end{equation}
for every $u\in\cC^\infty(X,T^{*0,1}M'\otimes L^k)$, $k\gg1$.
\end{theorem} 

In the proof of our main result, we need the following result.

\begin{lemma}\label{l-gue220415ycd} 
There is a constant $C>0$ independent of $k$ such that 
\begin{equation}\label{e-gue220415ycdg}
\norm{u}_{1,X,k}\leq C\Bigr(\norm{\mathcal{N}^{(1)}_{\tau^2_{k\delta}}u}_{X,k}+\norm{Tu}_{X,k}+k\norm{u}_{X,k}\Bigr),
\end{equation} 
for every $u\in\cC^\infty(X,T^{*0,1}M'\otimes L^k)$, $k\gg1$. 
\end{lemma}

\begin{proof}
Let $\chi\in\cC^\infty_c(D)$ and let $u\in\cC^\infty(X,T^{*0,1}M'\otimes L^k)$. From Lemma~\ref{l-gue220410yyd}, 
we have that for every $N>0$ there is a constant $C_N>0$ such that 
\begin{equation}\label{e-gue220415ycda}
\norm{\mathcal{N}^{(1)}_{\tau^2_{k\delta}}\chi u}_{X,k}\geq C_N\Bigr(\norm{\tilde\chi_1\widehat{\mathcal{N}}^{(1)}_{\tau^2_{k\delta}}\chi u}_{X,k}-k^{-N}\norm{u}_{1,X,k}\Bigr),
\end{equation} 
for every $u\in\cC^\infty(X,T^{*0,1}M'\otimes L^k)$. Let $u\in\cC^\infty(X,T^{*0,1}M'\otimes L^k)$. On $D$, write $u=s^k\otimes\tilde u$, $\tilde u\in\cC^\infty(D,T^{*0,1}M')$. 
Let $\hat u:=\tilde ue^{-k\phi}$. Let $\widehat{\mathcal{N}}^{(1)}_{\tau^2_{k\delta},s,0}:={\rm Op\,}_k(\sigma^0_a)$, where $a$ is as in \eqref{e-gue220410yydIII}. Then, 
\begin{equation}\label{e-gue220415ycdb}
\widehat{\mathcal{N}}^{(1)}_{\tau^2_{k\delta},s}-\widehat{\mathcal{N}}^{(1)}_{\tau^2_{k\delta},s,0}=:\widehat{\mathcal{N}}^{(1)}_{\tau^2_{k\delta},s,1}={\rm Op\,}_k(b), 
\end{equation}
$b\in S^{2n-1}_{{\rm cl\,}}((1+\abs{\xi}); D\times D\times\R^{2n-1}, T^{*0,1}M'\boxtimes(T^{*0,1}M')^*)$. Note that there exist
$C, C_1>0$ such that for every $(x,y)\in D\times D$, and $\xi$ with
$\abs{\xi'}\geq C_1$, where $\xi'=(\xi_1,\ldots,\xi_{2n-2})$,
$$k\Big|\sigma^0_a\Big(x,y,\frac{\xi}{k}\Big)\Big|\geq C(\abs{\xi'}+k).$$  
From this observation and some elementary computations, it is straightforward to see that there is a constant $C_0>0$ independent of $k$ and $u$ such that 
\begin{equation}\label{e-gue220415ycdc}
\sum^{2n-2}_{j=1}\norm{\pr_{x_j}(\chi\hat u)}^2_X\leq C_0\Bigr(\norm{\tilde\chi_1\widehat{\mathcal{N}}^{(1)}_{\tau^2_{k\delta},s,0}(\chi\hat u)}^2_X+k^2\norm{\tilde\chi_1u}^2_X\Bigr),
\end{equation}
where $\tilde\chi_1\in\cC^\infty_c(D)$, $\tilde\chi_1\equiv1$ near $\supp\chi$. From Lemma~\ref{l-gue220412yyd}, there is a constant $\tilde C_0>0$ independent of $k$ and $u$ such that 
\[\norm{\tilde\chi_1\widehat{\mathcal{N}}^{(1)}_{\tau^2_{k\delta},s,1}\chi\hat u}^2_X\leq\frac{\tilde C_0}{k^2}\Bigr(\norm{\hat\chi_1\hat u}^2_{1,X}+k^2\norm{\hat\chi_1\hat u}^2_X\Bigr),\]
where $\widehat{\mathcal{N}}^{(1)}_{\tau^2_{k\delta},s,1}$ is as in \eqref{e-gue220415ycdb} and $\hat\chi_1\in\cC^\infty_c(D)$, $\hat\chi_1\equiv1$ near $\supp\tilde\chi$. From this observation, \eqref{e-gue220415ycdb} and \eqref{e-gue220415ycdc}, 
there is a constant $\hat C_0>0$ independent of $k$ and $u$ such that 
\begin{equation}\label{e-gue220415ycdd}
\norm{\chi\hat u}^2_{1,X}\leq\hat C_0\Bigr(\norm{\hat\chi_1\widehat{\mathcal{N}}^{(1)}_{\tau^2_{k\delta},s}\chi\hat u}^2_X+\frac{1}{k^2}\norm{\hat\chi_1\hat u}^2_{1,X}
+\norm{T\hat\chi_1\hat u}^2_X+k^2\norm{\hat\chi_1\hat u}^2_X\Bigr).
\end{equation} 
From \eqref{e-gue220415ycda} and \eqref{e-gue220415ycdd}, we have 
\begin{equation}\label{e-gue220415ycde}
\norm{\chi u}^2_{1,X,k}\leq\hat C_1\Bigr(\norm{\mathcal{N}^{(1)}_{\tau^2_{k\delta}}\chi u}^2_{X,k}+\norm{Tu}^2_{X,k}+k^2\norm{u}^2_{X,k}+\frac{1}{k^2}\norm{u}^2_{1,X,k}\Bigr),
\end{equation}
where $\hat C_1>0$ is a constant independent of $k$ and $u$. From \eqref{e-gue220415ycde} and by using partition of unity, we get \eqref{e-gue220415ycdg}. 
\end{proof}

\subsection{Semi-classical behaviour of $P^{(0)}_kF^{(0)}_{\tau_{k\delta}}$}\label{s=gue220422yyd}

We will use the same notations as before. 
Let $s$ be a $\R$-equivariant holomorphic frame of $L$ 
over an open set $U\Subset M'$, $U\bigcap X\neq\emptyset$, 
$\abs{s}^2_{h^L}=e^{-2\phi}$, $\phi\in\cC^\infty(U,\R)$. 
Let $D:=U\cap X$. Let $x=(x_1,\ldots,x_{2n-1})$ 
be local coordinates of $D$ with $T=\frac{\pr}{\pr x_{2n-1}}$ on $D$. 
We extend $x_1,\ldots,x_{2n-1}$ to smooth functions on $U$ so that $\tilde x:=(x_1,\ldots,x_{2n-1},\rho)=(x,\rho)$ are local coordinates on $U$ and $T=\frac{\pr}{\pr x_{2n-1}}$ on $U$. 
Put
\begin{equation}\label{e-gue220422yyd}
\begin{split}
&B^{(0)}_{\tau_{k\delta}, s}(x, y):=\frac{k^{2n-1}}{(2\pi)^{2n-1}}\int_{\R^{2n-1}} e^{ik\langle x-y,~\xi\rangle}\tau^2_\delta(\xi_{2n-1})d\xi,\\
&B^{(0)}_{\tau_{k\delta}, M,s}(x, y):=\frac{k^{2n}}{(2\pi)^{2n}}\int_{\R^{2n}} e^{ik\langle\tilde x-\tilde y,~\tilde\xi\rangle}\tau^2_\delta(\xi_{2n-1})d\tilde\xi,
\end{split}
\end{equation}
where $\tilde\xi=(\xi_1,\ldots,\xi_{2n})$. 

\begin{lemma}\label{l-gue220422yyd}
We have for all $\ell\in\Z$,
\[\begin{split}
&F^{(0)}_{\tau_{k\delta},s}-B^{(0)}_{\tau_{k\delta}, s}
=O(k^{-\infty}):W^\ell_{{\rm comp\,}}(D)\To W^\ell_{{\rm loc\,}}(D),\ \ \\
&F^{(0)}_{\tau_{k\delta},M,s}-B^{(0)}_{\tau_{k\delta}, M,s}
=O(k^{-\infty}):W^\ell_{{\rm comp\,}}(U\cap\ol M)
\To W^\ell_{{\rm loc\,}}(U\cap\ol M). 
\end{split}\] 
\end{lemma}

We can repeat the proof of Lemma~\ref{l-gue220330yyd} and get
the following.

\begin{lemma}\label{l-gue220422yydI}
Let $\chi, \hat\chi\in\cC^\infty(\ol M)$ with 
$\supp\chi\cap\supp\hat\chi=\emptyset$. We have 
\[\hat\chi F^{(0)}_{\tau_{k\delta},M}\chi
=O(k^{-\infty}): W^\ell(\ol M)\To W^\ell(\ol M),\ \ \forall\ell\in\mathbb Z.\]
\end{lemma}
 
 Let $\tilde\tau_{\delta}\in\cC^\infty_c((\delta/4,2\delta))$
 with $\tilde\tau_\delta\equiv1$ near $\supp\tau_\delta$. 
 We can repeat the proof of Theorem~\ref{21-5-19-thm1} 
 with minor changes and deduce the following.

\begin{theorem}\label{t-gue220422yyd}
There exist $\varphi(\tilde x, y, \xi)\in S^0((1+\abs{\xi});((U\times U)\cap(\ol M\times X))\times \R^{2n-1})$ with 
\[\varphi(\tilde x, y, \xi)\sim\langle\,x-y\,,\,\xi\,\rangle+\rho\varphi_1(x,y,\xi)+\ldots+\rho^j\varphi_j(x,y,\xi)+\ldots,\]
$\varphi_j\in S^0((1+\abs{\xi}); D\times D\times \R^{2n-1})$, $j=1,2,\ldots$, $\varphi_1$ is as in \eqref{e-gue220426yyd} and 
\[\begin{split}
&g(\tilde x,y,\xi,k)\in S^{2n-1}_{{\rm cl\,}}(1;((U\times U)\cap(\ol M\times X))\times\R^{2n-1}),\\
&\mbox{$g(\tilde x,y,\xi,k)\sim\sum^{+\infty}_{j=0}k^{2n-1-j}g^j(\tilde x,y,\xi)$ in $S^{2n-1}_{{\rm cl\,}}(1;((U\times U)\cap(\ol M\times X))\times\R^{2n-1})$}\\
&g^j(x,y,\xi,k)\in S^{0}(1;D\times D\times\R^{2n-1}),\ \ j=0,1,\ldots,\\
&\mbox{$g(\tilde x,y,\xi,k)=g^j(\tilde x,y,\xi)=0$ if $\xi_{2n-1}\notin\supp\tilde\tau_\delta$, for every $j=0,1,\ldots$},\\
&g^0(\tilde x,y,\xi)=(2\pi)^{-(2n-1)}\tau_\delta(\xi_{2n-1})+O(\rho),
\end{split}\]
such that if $\tilde B^{(0)}_{\tau_{k\delta},s}:
\cC^\infty_c(D)\To\cC^\infty(U\cap\ol M)$ 
is the continuous operator with distribution kernel 
\[\tilde B^{(0)}_{\tau_{k\delta},s}(\tilde x,y)=
\int e^{ik\varphi(\tilde x,y,\xi)}g(\tilde x,y,\xi,k)d\xi,\]
then 
\[\begin{split}
&\Box^{(0)}_{s}\tilde B^{(0)}_{\tau_{k\delta}, s}\equiv0\mod O(k^{-\infty})\ \ \mbox{on $(U\times U)\cap(\ol M\times X)$},\\
&\tilde B^{(0)}_{\tau_{k\delta}, s}|_X=B^{(0)}_{\tau_{k\delta}, s}. 
\end{split}\]
\end{theorem} 

Define
\begin{equation}\label{e-gue220422yyda}
\begin{split}
&B^{(0)}_{\tau_{k\delta}}:=e^{k\phi}s^kB^{(0)}_{\tau_{k\delta},s}(s^{-k}e^{-k\phi}): \cC^\infty_c(D,L^k)\To\cC^\infty(D,L^k),\\
&\tilde B^{(0)}_{\tau_{k\delta}}:=e^{k\phi}s^k\tilde B^{(0)}_{\tau_{k\delta},s}(s^{-k}e^{-k\phi}): \cC^\infty_c(D,L^k)\To\cC^\infty(U\cap\ol M,L^k),
\end{split}
\end{equation}
where $\tilde B^{(0)}_{\tau_{k\delta},s}$ is as in Theorem~\ref{t-gue220422yyd}. From Lemma~\ref{l-gue220314yydz}, Lemma~\ref{l-gue220422yyd} 
and Lemma~\ref{l-gue220422yydI}, we can repeat the proof of Theorem~\ref{t-gue220406yyd} with minor changes and deduce the following.

\begin{theorem}\label{t-gue220422yydI}
With the notations used above, let $\chi\in\cC^\infty_c(D)$, 
$\chi_1\in\cC^\infty_c(U\cap\ol M)$, $\chi_1\equiv1$ near $\supp\chi$. We have 
\begin{equation}\label{e-gue220422yydb}
P^{(0)}_kF^{(0)}_{\tau_{k\delta}}\chi-F^{(0)}_{\tilde\tau_{k\delta},M}\circ\chi_1\tilde B^{(0)}_{\tau_{k\delta}}\chi=O(k^{-\infty}): W^\ell(X)\To W^{\ell+\frac{1}{2}}(\ol M),\ \ \forall \ell\in\mathbb Z.
\end{equation}
\end{theorem}

Note that $g(\tilde x,y,\xi,k)=0$ if $\xi_{2n-1}\notin\tilde\tau_{\delta}$, where $g(\tilde x,y,\xi,k)$ is as in Theorem~\ref{t-gue220422yyd}. From this observation, Lemma~\ref{l-gue220422yyd}, Lemma~\ref{l-gue220422yydI} and by using 
integration by parts, we can check that 
\begin{equation}\label{e-gue220422yydc}
F^{(0)}_{\tilde\tau_{k\delta},M}\chi_1\tilde B^{(0)}_{\tau_{k\delta}}\chi-\chi_1\tilde B^{(0)}_{\tau_{k\delta}}\chi=O(k^{-\infty}): W^\ell(X,L^k)\To W^\ell(\ol M,L^k),\ \ \forall\ell\in\mathbb Z.
\end{equation}
From \eqref{e-gue220422yydb} and \eqref{e-gue220422yydc}, we obtain:

\begin{theorem}\label{t-gue220422yydII}
With the notations used above, let $\chi\in\cC^\infty_c(D)$, $\chi_1\in\cC^\infty_c(U\bigcap\ol M)$, $\chi_1\equiv1$ near $\supp\chi$. We have 
\begin{equation}\label{e-gue220422yydd}
P^{(0)}_kF^{(0)}_{\tau_{k\delta}}\chi-\chi_1\tilde B^{(0)}_{\tau_{k\delta}}\chi=O(k^{-\infty}): W^\ell(X)\To W^{\ell+\frac{1}{2}}(\ol M),\ \ \forall \ell\in\mathbb Z,
\end{equation}
where $\tilde B^{(0)}_{\tau_{k\delta}}$ is given by \eqref{e-gue220422yyda}. 
\end{theorem}
 
 Let $(P^{(0)}_k)^*: \cC^\infty(\ol M,L^k)\To\cC^\infty(X,L^k)$ be the adjoint of $P^{(0)}_k$. From \eqref{e-gue220422yydd}, we can check that 
for all $\ell\in\Z$ we have
\begin{equation}\label{e-gue220422yyde}
\chi (P^{(0)}_k)^*P^{(0)}_kF^{(0)}_{\tau^2_{k\delta}}\chi-(\chi_1\tilde B^{(0)}_{\tau_{k\delta}}\chi)^*(\chi_1\tilde B^{(0)}_{\tau_{k\delta}}\chi)
=O(k^{-\infty}): W^\ell(X,L^k)\To W^{\ell+1}(X,L^k),
 \end{equation}
where $(\chi_1\tilde B^{(0)}_{\tau_{k\delta}}\chi)^*: \cC^\infty(\ol M,L^k)\To\cC^\infty(X,L^k)$ is the adjoint of $\chi_1\tilde B^{(0)}_{\tau_{k\delta}}\chi$. 
From Theorem~\ref{t-gue220422yyd}, it is straightforward to check that 
\begin{equation}\label{e-gue220423yyd}
\begin{split}
&(\chi_1\tilde B^{(0)}_{\tau_{k\delta},s}\chi)^*(\chi_1\tilde B^{(0)}_{\tau_{k\delta},s}\chi)=\chi E_{k,s}\chi,\\
&E_{k,s}={\rm Op\,}_k(\alpha),\ \ \alpha\in S^{2n-2}_{{\rm cl\,}}((1+\abs{\xi})^{-1}; D\times D\times\R^{2n-1}),\\
&\mbox{$\alpha(x,y,\xi,k)=0$ if $\xi_{2n-1}\notin\supp\tilde\tau_\delta$},\\
&\sigma^0_\alpha(x,x,\xi)=(2\pi)^{-(2n-1)}\tau^2_\delta(\xi_{2n-1})\frac{1}{2i\varphi_1(x,x,\xi)},\ \ \forall (x,x,\xi)\in D\times D\times\R^{2n-1},
\end{split}
\end{equation}
where $\varphi_1$ is given by \eqref{e-gue220426yyd}. Let 
\begin{equation}\label{e-gue220423yydVz}
E_k:=(e^{k\phi}s^k)E_{k,s}(s^{-k}e^{-k\phi}): \cC^\infty_c(D,L^k)\To\cC^\infty(D,L^k).
\end{equation}
From \eqref{e-gue220422yyde} and \eqref{e-gue220423yyd}, we get 
\begin{equation}\label{e-gue220423yydI}
\chi (P^{(0)}_k)^*P^{(0)}_kF^{(0)}_{\tau^2_{k\delta}}\chi-\chi E_k\chi=O(k^{-\infty}): W^\ell(X,L^k)\To W^{\ell+1}(X,L^k),\ \ \forall\ell\in\mathbb Z.
\end{equation}
Let $\hat\chi\in\cC^\infty_c(U\cap\ol M)$ with 
$\supp\hat\chi\cap\supp\chi=\emptyset$. From \eqref{e-gue220422yydd} 
and off-diagonal expansion of $\tilde B^{(0)}_{\tau_{k\delta}}$, we get 
\begin{equation}\label{e-gue220423yydII}
\hat\chi P^{(0)}_kF^{(0)}_{\tau_{k\delta}}\chi=O(k^{-\infty}): W^\ell(X,L^k)\To W^{\ell+\frac{1}{2}}(\ol M,L^k),\ \ \forall\ell\in\mathbb Z
\end{equation}
and hence 
\begin{equation}\label{e-gue220423yydIII}
\chi F^{(0)}_{\tau_{k\delta}}(P^{(0)}_k)^*\hat\chi=O(k^{-\infty}): W^\ell(\ol M,L^k)\To W^{\ell+\frac{1}{2}}(X,L^k),\ \ \forall\ell\in\mathbb Z.
\end{equation}
From \eqref{e-gue220423yydII} and \eqref{e-gue220423yydIII}, we deduce that for every $\chi_1, \chi_2\in\cC^\infty(X)$ with 
\[\supp\chi_1\cap\supp\chi_2=\emptyset,\] we have 
\begin{equation}\label{e-gue220423yydIV}
\chi_1(P^{(0)}_k)^*P^{(0)}_kF^{(0)}_{\tau^2_{k\delta}}\chi_2=O(k^{-\infty}): W^\ell(X,L^k)\To W^{\ell+1}(X,L^k),\ \ \forall\ell\in\mathbb Z.
\end{equation}
From \eqref{e-gue220423yydI} and \eqref{e-gue220423yydIV}, we get 

\begin{lemma}\label{l-gue220423yyd} 
With the notations used above, let $\chi\in\cC^\infty_c(D)$, $\chi_1\in\cC^\infty_c(D)$, $\chi_1\equiv1$ near $\supp\chi$. We have 
\begin{equation}\label{e-gue220423yydV}
(P^{(0)}_k)^*P^{(0)}_kF^{(0)}_{\tau^2_{k\delta}}\chi=\chi_1E_k\chi+\delta_k,
\end{equation}
where $E_k$ is given by \eqref{e-gue220423yydVz} and $\delta_k$ is continuous: 
\[\delta_k=O(k^{-\infty}): W^\ell(X,L^k)\To W^{\ell+1}(X,L^k),\ \ \forall\ell\in\mathbb Z.\]
\end{lemma}

Let $\chi_2, \chi_3\in\cC^\infty_c(D)$, $\chi\equiv1$ near $\supp\chi_2$, $\chi_2\equiv1$ near $\supp\chi_3$. 
Let $\hat\tau_\delta\in\cC^\infty_c((\delta/4,2\delta))$, $0\leq\hat\tau_\delta\leq1$, $\hat\tau_\delta=1$ near $\supp\tau_\delta$, $\tilde\tau_\delta=1$ near 
$\supp\hat\tau_\delta$. From complex stationary phase formula~\cite{MS} and standard WKB method, we can find $H={\rm Op\,}_k(\beta)$, $\beta\in S^{2n}_{{\rm cl\,}}((1+\abs{\xi}); D\times D\times\R^{2n-1})$, $\beta(x,y,\xi,k)=0$ if $\xi_{2n-1}\notin\supp\tilde\tau_\delta$, $\sigma^0_\beta(x,x,\xi)=(2\pi)^{-(2n-1)}\hat\tau^2_\delta(\xi_{2n-1})(2i\varphi_1(x,x,\xi))$, for all $(x,x,\xi)\in D\times D\times\R^{2n-1}$, $\varphi_1$ is as in \eqref{e-gue220426yyd}, such that 
\begin{equation}\label{e-gue220423yydp}
(\chi_1E_{k,s}\chi)(\chi_2H_{k,s}\chi_3)-\chi_1B^{(0)}_{\tau^2_{k\delta},s}\chi_3=O(k^{-\infty}): W^\ell(X)\To W^\ell(X),\ \ \forall\ell\in\mathbb Z.
\end{equation}
Let $H_k:=e^{k\phi}s^kH_{k,s})(e^{-k\phi}s^{-k}): \cC^\infty_c(D,L^k)\To\cC^\infty(D,L^k)$. From Lemma~\ref{l-gue220422yyd}, \eqref{e-gue220423yydV} 
and \eqref{e-gue220423yydp}, we get $(P^{(0)}_k)^*P^{(0)}_kF^{(0)}_{\tau^2_{k\delta}}\chi_2H_k\chi_3=F^{(0)}_{\tau^2_{k\delta}}\chi_3+\hat\delta_k$, where 
$\hat\delta_k$ is continuous: 
\[\hat\delta_k=O(k^{-\infty}): W^\ell(X,L^k)\To W^{\ell}(X,L^k),\ \ \forall\ell\in\mathbb Z.\]
Thus, 
\begin{equation}\label{e-gue220425yyd}
F^{(0)}_{\tau^2_{k\delta}}\chi_2H_k\chi_3=((P^{(0)}_k)^*P^{(0)}_k)^{-1}F^{(0)}_{\tau^2_{k\delta}}\chi_3+(P^{(0)}_k)^*P^{(0)}_k)^{-1}F^{(0)}_{\tilde\tau_{k\delta}}\hat\delta_k.
\end{equation} 
Let $\tau^\dagger_\delta\in\cC^\infty_c((\delta/4,2\delta))$, $0\leq\tau^\dagger_\delta\leq1$, $\tau^\dagger_\delta\equiv1$ on $[\frac{\delta}{2},\delta]$, $\tau\equiv1$ near $\supp\tau^\dagger$. From \eqref{e-gue220425yyd}, we get 
\begin{equation}\label{e-gue220425yydI}
F^{(0)}_{\tau^\dagger_{k\delta}}\chi_2H_k\chi_3=((P^{(0)}_k)^*P^{(0)}_k)^{-1}F^{(0)}_{\tau^\dagger_{k\delta}}\chi_3+((P^{(0)}_k)^*P^{(0)}_k)^{-1}F^{(0)}_{\tau^\dagger_{k\delta}}\hat\delta_k.
\end{equation}
Let $X=\bigcup^N_{j=1}D_j$, $D_j\subset X$ open, $j=1,\ldots,N$. Let $\chi_j\in\cC^\infty_c(D_j)$, $j=1,\ldots,N$, $\sum^N_{j=1}\chi_j\equiv1$ on $X$. From \eqref{e-gue220425yydI} and by using partition of unity, we get 

\begin{equation}\label{e-gue220425yydII}
\sum^N_{j=1}F^{(0)}_{\tau^\dagger_{k\delta}}\tilde\chi_jH_{k,j}\chi_j=((P^{(0)}_k)^*P^{(0)}_k)^{-1}F^{(0)}_{\tau^\dagger_{k\delta}}(I+\Gamma_k), 
\end{equation}
where $\tilde\chi_j\in\cC^\infty(D_j)$, $\tilde\chi_j\equiv1$ near 
$\supp\chi_j$, $j=1,\ldots,N$, $H_{k,j}$ is as in \eqref{e-gue220425yydI} 
and $\Gamma_k$ is continuous: 
\[\Gamma_k=O(k^{-\infty}): W^\ell(X,L^k)\To W^\ell(X,L^k),\ \ 
\forall\ell\in\mathbb Z.\] 

\begin{lemma}\label{l-gue220425yyd}
For every $\ell\in\mathbb Z$, there is a constant $C_\ell>0$
independent of $k$ such that 
\[\norm{((P^{(0)}_k)^*P^{(0)}_k)^{-1}
F^{(0)}_{\tau^\dagger_{k\delta}}u}_{\ell-1,X,k}\leq 
C_\ell\Bigr(\norm{u}_{\ell,X,k}+k\norm{u}_{\ell-1,X,k}\Bigr),\]
for every $u\in W^\ell(X,L^k)$, $k\gg1$.
\end{lemma}

\begin{proof}
For $k$ large, the Neumann series 
$\sum^\infty_{j=0}(-\Gamma_k)^j$ converges to $W^\ell(X,L^k)$ and there is a constant $C>1$ independent of $k$ 
such that 
\begin{equation}\label{e-gue220425yydIII}
\frac{1}{C}\norm{u}_{\ell,X,k}\leq\norm{\sum^\infty_{j=0}(-\Gamma_k)^ju}_{\ell,X,k}\leq C\norm{u}_{\ell,X,k},
\end{equation}
for every $u\in W^\ell(X,L^k)$. Let $u\in W^\ell(X,L^k)$. From the proof of Lemma~\ref{l-gue220411yyd}, we see that 
\begin{equation}\label{e-gue220425yydIV}
\norm{\sum^N_{j=1}F^{(0)}_{\tau^\dagger_{k\delta}}\tilde\chi_jH_{k,j}\chi_ju}_{\ell-1,X,k}\leq C\Bigr(\norm{u}_{\ell,X,k}+k\norm{u}_{\ell-1,X,k}\Bigr),
\end{equation}
where $C>0$ is a constant independent of $k$ and $u$. From \eqref{e-gue220425yydII}, \eqref{e-gue220425yydIII} and \eqref{e-gue220425yydIV}, we get 
\[\begin{split}
&\norm{((P^{(0)}_k)^*P^{(0)}_k)^{-1}F^{(0)}_{\tau^\dagger_{k\delta}}u}_{\ell-1,X,k}=\norm{((P^{(0)}_k)^*P^{(0)}_k)^{-1}F^{(0)}_{\tau^\dagger_{k\delta}}(I+\Gamma_k)\sum^\infty_{j=1}(-\Gamma_k)^ju}_{\ell-1,X,k}\\
&\leq C\Bigr(\norm{\sum^\infty_{j=0}(-\Gamma_k)^ju}_{\ell,X,k}+k\norm{\sum^\infty_{j=0}(-\Gamma_k)^ju}_{\ell-1,X,k}\Bigr)\\
&\leq\hat C\Bigr(\norm{u}_{\ell,X,k}+k\norm{u}_{\ell-1,X,k}\Bigr),
\end{split}\]
for all $u\in W^\ell(X,L^k)$, for all $k\gg1$, where $C>0$, $\hat C>0$ are constants independent of $k$. The lemma follows. 
\end{proof}

We return to the equation \eqref{e-gue220425yyd}. From Lemma~\ref{l-gue220422yyd} and Lemma~\ref{l-gue220422yydI}, we can check that 
\begin{equation}\label{e-gue220425yyda}
F^{(0)}_{\tau^2_{k\delta}}\chi_2H_k\chi_3-\tilde\chi_2\tilde H_k\chi_3=O(k^{-\infty}): W^\ell(X,L^k)\To W^\ell(X,L^k),\ \ \forall\ell\in\mathbb Z, 
\end{equation}
where $\tilde H_k:\cC^\infty_c(D,L^k)\To\cC^\infty(D,L^k)$ is the continuous operator with $\tilde H_{k,s}={\rm Op\,}_k(\tilde\beta)$, $\tilde\beta\in S^{2n}_{{\rm cl\,}}((1+\abs{\xi});D\times D\times\mathbb  R^{2n-1})$, $\tilde\beta(x,y,\xi,k)=0$ if $\xi_{2n-1}\notin\supp\tilde\tau_\delta$, $\sigma^0_{\tilde\beta}(x,x,\xi)=(2\pi)^{-(2n-1)}\tau^2_\delta(\xi_{2n-1})(2i\varphi_1(x,x,\xi))$ where $\tilde H_{k,s}$ denotes the localization of $\tilde H_k$ with respect to $s$. From Lemma~\ref{l-gue220425yyd}, we get \begin{equation}\label{e-gue220425yydb}
((P^{(0)}_k)^*P^{(0)}_k)^{-1}F^{(0)}_{\tilde\tau_{k\delta}}\hat\delta_k=O(k^{-\infty}): W^\ell(X,L^k)\To W^{\ell-1}(X,L^k),\ \ \forall\ell\in\mathbb Z. 
\end{equation}
From \eqref{e-gue220425yyd}, \eqref{e-gue220425yyda} and \eqref{e-gue220425yydb}, we get 

\begin{lemma}\label{l-gue220425yydI}
With the notations used above, let $\chi\in\cC^\infty_c(D)$. W ehave 
\begin{equation}\label{e-gue220425yydc}
((P^{(0)}_k)^*P^{(0)}_k)^{-1}F^{(0)}_{\tilde\tau_{k\delta}}\chi=\chi_1\tilde H_k\chi+\tilde\delta_k, 
\end{equation}
where $\chi_1\in\cC^\infty_c(D)$, $\chi_1\equiv1$ near $\supp\chi$, $\tilde H_k: \cC^\infty_c(D,L^k)\To\cC^\infty(D,L^k)$ is the continuous operator with $\tilde H_{k,s}={\rm Op\,}_k(\tilde\beta)$, $\tilde\beta\in S^{2n}_{{\rm cl\,}}((1+\abs{\xi});D\times D\times\mathbb  R^{2n-1})$, $\tilde\beta(x,y,\xi,k)=0$ if $\xi_{2n-1}\notin\supp\tilde\tau_\delta$, $\sigma^0_{\tilde\beta}(x,x,\xi)=(2\pi)^{-(2n-1)}\tilde\tau^2_\delta(\xi_{2n-1})(2i\varphi_1(x,x,\xi))$ where $\tilde H_{k,s}$ denotes the localization of $\tilde H_k$ with respect to $s$, $\varphi_1$ is as in \eqref{e-gue220426yyd} and $\tilde\delta_k$ is continuous: 
\[\tilde\delta_k=O(k^{-\infty}): W^\ell(X,L^k)\To W^{\ell-1}(X,L^k),\ \ \forall\ell\in\mathbb Z.\]
\end{lemma}

\section{Bergman kernel asymptotics}\label{s-gue220430yyd}

In this section, we will prove our main results. 

\subsection{Weighted Szeg\H{o}  kernels on the boundary}
\label{s-gue220430yydI} 

We consider $X$ the boundary of $M$. $X$ is a compact CR manifold with natural CR 
structure \[T^{1,0}X:=\mathbb CTX\cap T^{1,0}M'.\] 
Recall that $L$ is a positive $\R$-equivariant
CR line bundle over $X$. 
We consider the 1-form on $M'$,
\begin{equation}\label{e-gue190312scdqII}
\omega_0=d\rho\circ J=i(\partial\rho-\ddbar\rho).
\end{equation}
Then $\omega_0|_X\in\cC^\infty(X,TX)$ and  
$\omega_0(T)=-1$ on $X$, 
$\langle\,\omega_0\,,\,u\,\rangle=0$, for all $u\in T^{1,0}X\oplus T^{0,1}X$. 
We recall the following definitions (see~\cite[Definition 2.11]{HHL17}) 

\begin{definition}\label{d-gue220507yyd}
The Levi form of $X$ is the Hermitian qandratic form 
$\mathcal{L}$ on $T^{1,0}X$  defined by 
\[\mathcal{L}_p(U,\ol V)=-\frac{1}{2i}d\omega_0(p)
(U,\overline V)=\partial\ddbar\rho(p)(U,\ol{V}),\ \ U, V\in T^{1,0}_pX,\ \ p\in X.\]
\end{definition}

\begin{definition}\label{d-gue220507yydI}
Let $s$ be a $\R$-equivariant CR frame of $L$ on an open set 
$D$ of $X$, and let $\phi$ be the corresponding weight, 
$\abs{s}^2_{h^L}=e^{-2\phi}$. 
The CR curvature of $(L,h^L)$ is the Hermitian quadratic 
form $R^L_X$ on $T^{1,0}X$ 
defined by 
\[R^L_{X,p}(U,\ol V)=d(\ddbar_b\phi-\pr_b\phi)(p)
(U,\ol{V}),\ \ U, V\in T^{1,0}_pX,\ \ p\in D.\]
\end{definition}

It was shown in~\cite[Proposition 4.2]{HM12} that the definition 
of $R^L_X$ is independent of the choices of $\R$-equivariant CR 
frames of $L$. 

We have the following relation between the CR curvature of $L$
and the restriction of the Chern curvature $R^L$ of $L$
over $M'$.

\begin{lemma}\label{L:RX}
For every $p\in X$ we have
\begin{equation}\label{e-gue220711yydIa}
R^L_{X,p}=R^L|_{T^{1,0}_pX}+2(JT)(\phi)(p)\mathcal{L}_p\,.
\end{equation}
\end{lemma}
\begin{proof}
Let $p\in X$ and holomorphic 
coordinates $z=(z_1,\ldots,z_n)$ so that $z(p)=0$ and 
\[\begin{split}
&\rho(z)=x_{2n}+\sum^{n-1}_{j=1}\lambda_j\abs{z_j}^2+O(\abs{z}^3),\\
&\phi(z)=\beta x_{2n}+\phi_1(z),
    \end{split}\]
    where $\beta=-(JT)(\phi)(p)$ and $\phi_1(z)=O(\abs{z}^2)$. We have 
    \begin{equation}\label{e-gue220711yyda}
    \phi|_X=-\beta\sum^{n-1}_{j=1}\lambda_j\abs{z_j}^2|_X+\phi_1(z)|_X+O(\abs{z}^3).
    \end{equation} 
    From \eqref{e-gue220711yyda}, it is not difficult to see that 
    \eqref{e-gue220711yydIa} holds. 

\end{proof}

Let 
\(\ddbar_b:\cC^\infty(X,L^k)\To\Omega^{0,1}(X,L^k)\)
be the tangential Cauchy-Riemann operator with values in $L^k$. We 
extend $\ddbar_b$ to the $L^2$ space,
\[\ddbar_b:\Dom\ddbar_b\subset L^2(X,L^k)\To L^2_{(0,1)}(X,L^k),\]
where $\Dom\ddbar_b=\set{u\in L^2(X,L^k);\, 
\ddbar_bu\in L^2_{(0,1)}(X,L^k)}$. 
Let 
\[\ddbar^*_b: \Dom\ddbar^*_b\subset L^2_{(0,1)}(X,L^k)\To 
L^2(X,L^k)\] be the $L^2$ adjoint 
of $\ddbar_b$. Put 
\[
\begin{split}
\Box_b:&=\ddbar^*_b\,\ddbar_b: \Dom\Box_b\subset L^2(X,L^k)
\To L^2(X,L^k),\\ 
\Dom\Box_b&=\set{u\in L^2(X,L^k);\, 
u\in\Dom\ddbar_b, \ddbar_bu\in\Dom\ddbar^*_b}. 
\end{split}
\]
Let 
\[S_k: L^2(X,L^k)\To {\rm Ker\,}\Box_b\]
be the orthogonal projection with respect to $(\,\cdot\,|\,\cdot\,)_{X,k}$. 
As before, let $\tau_{\delta}\in\cC^\infty_c((\delta/4,2\delta))$, 
$\tau_{\delta}\equiv1$ on $[\frac{\delta}{2},\delta]$. Put 
\[S_{\tau_{k\delta}}:=S_k\circ F^{(0)}_{\tau_{k\delta}}: 
L^2(X,L^k)\To {\rm Ker\,}\Box_b.\]
Let $s$ be a $\R$-equivariant CR frame of $L$ on an open set 
$D$ of $X$, $\abs{s}^2_{h^L}=e^{-2\phi}$, and let $x=(x_1,\ldots,x_{2n-1})$ be 
canonical coordinates of $X$ defined on $D$ 
with $T=\frac{\partial}{\partial x_{2n-1}}$ on $D$. 
Let $S_{\tau_{k\delta},s}: \cC^\infty_c(D)\To\cC^\infty(D)$ 
be the localization of $S_{\tau_{k\delta}}$ with respect to $s$ and let $S_{\tau_{k\delta},s}(x,y)\in\cC^\infty(D\times D)$ be the distribution kernel of $S_{\tau_{k\delta},s}$. The following was established in~\cite[Theorem 1.1]{HLM},~\cite[Theorem 1.1]{HHL17}. 

\begin{theorem}\label{t-gue220430yyd}
 With the notations used above, fix $\delta>0$ be a small constant. Let $\chi_1, \chi\in\cC^\infty(D)$. We have 
 \begin{equation}\label{e-gue220502yyd}
 (\chi_1S_{\tau_{k\delta},s}\chi)(x,y)=\int_{\R}e^{ik\Phi(x,y,t)}g(x,y,t,k)dt+O(k^{-\infty}),
 \end{equation}
 where $\Phi(x,y,t)\in\cC^\infty(D\times D\times(\delta/4,2\delta))$ is a phase function such that for some constant $c>0$, we have 
 \begin{equation}\label{e-gue220430yyd}
 \begin{split}
&d_x\Phi(x,y,t)|_{x=y}=-2{\rm Im\,}\ddbar_b\phi(x)-t\omega_0(x),\ \ 
d_y\Phi(x,y,t)|_{x=y}=2{\rm Im\,}\ddbar_b\phi(x)+t\omega_0(x),\\
&{\rm Im\,}\Phi(x,y,t)\geq C\abs{x'-y'}^2,\ \ x'=(x_1,\ldots,x_{2n}), y'=(y_1,\ldots,y_{2n}),\\
&{\rm Im\,}\Phi(x,y,t)+\abs{\frac{\partial\Phi}{\pr t}(x,y,t)}\geq C\abs{x-y}^2,\\
&\mbox{$\Phi(x,y,t)=0$, $\frac{\partial\Phi}{\pr t}(x,y,t)=0$ if and only if $x=y$ },
 \end{split}
 \end{equation}
 for every $(x,y,t)\in D\times D\times(\delta/4,2\delta)$ and $g(x,y,t,k)\in S^n(1;D\times D\times(\delta/4,2\delta))\cap\cC^\infty_c(D\times D\times(\delta/4,2\delta))$ is a symbol with expansion 
 \begin{equation}\label{e-gue220430yydI}
 \begin{split}
 &\mbox{$g(x,y,t,k)\sim\sum^{+\infty}_{j=0}g_j(x,y,t)k^{n-j}$ in $S^n(1;D\times D\times(\delta/4,2\delta))$},\\
 &g_j(x,y,t)\in S^0(1;D\times D\times(\delta/4,2\delta))\cap\cC^\infty_c(D\times D\times(\delta/4,2\delta)),\ \ j=0,1,\ldots,\\
 &g_0(x,x,t)=(2\pi)^{-n}\chi_1(x)\chi(x)\abs{\det(R^L_{X,x}+2t\mathcal{L}_x)}\abs{\tau_\delta(t)}^2,\ \ \forall (x,y,t)\in D\times D\times(\delta/4,2\delta),
 \end{split}
 \end{equation}
 where $\abs{\det(R^L_{X,x}+2t\mathcal{L}_x)}=\abs{\lambda_1(x,t)}\ldots\abs{\lambda_{n-1}(x,t)}$, $\lambda_j(x,t)$, $j=1,\ldots,n-1$, are the eigenvalues of the Hermitian quadratic form $R^L_{X,x}+2t\mathcal{L}_x$ with respect to $\langle\,\cdot\,|\,\cdot\,\rangle$. 
\end{theorem}

We refer the reader to~\cite[Section 3.3]{HLM} for more properties of the phase function $\Phi(x,y,t)$ in Theorem~\ref{t-gue220430yyd}. Under the same notations used in 
Theorem~\ref{t-gue220430yyd}, from Malgrange preparation theorem, we can take $\Phi(x,y,t)$ so that 
\begin{equation}\label{e-gue220504ycdm}
\Phi(x,y,t)=x_{2n-1}t+\hat\Phi(x',y,t),
\end{equation}
where $\hat\Phi(x',y,t)\in\cC^\infty(D\times D\times(\delta/4,2\delta))$, $x'=(x_1,\ldots,x_{2n-2})$. From now on, we assume that $\Phi(x,y,t)$ has the form 
\eqref{e-gue220504ycdm}. 

We need the following off-diagonal expansion. 

\begin{theorem}\label{t-gue220501yyd}
Let $\hat\chi, \chi\in\cC^\infty(X)$ with $\supp\hat\chi\cap\supp\chi=\emptyset$. Then, 
\begin{equation}\label{e-gue220502yydb}
\hat\chi S_{\tau_{k\delta}}\chi\equiv0\mod O(k^{-\infty})\ \ \mbox{on $X\times X$}. 
\end{equation}
\end{theorem}

\begin{proof}
Let $s$ be a $\R$-equivariant CR frame defined on an open set $D$ of $X$, 
$\abs{s}^2_{h^L}=e^{-2\phi}$, and let $x=(x_1,\ldots,x_{2n-1})$ be 
canonical coordinates of $X$ defined on $D$ 
with $T=\frac{\partial}{\partial x_{2n-1}}$ on $D$. Let $\chi_1\in\cC^\infty_c(D)$ and let $\tilde\tau_\delta\in\cC^\infty_c((\delta/4,2\delta))$, $0\leq\tilde\tau_\delta\leq1$, 
$\tilde\tau_\delta\equiv1$ near $\supp\tau_\delta$. From~\cite[Theorem 6.5]{Hsiao18}, we can find continuous 
operators $G_{k,s}, \hat S_{k,s}: \cC^\infty_c(D)\To\cC^\infty(D)$ such that $G_{k,s}$, $\hat S_{k,s}$ are properly supported on $D$,  the distribution kernel of $\hat S_{k,s}$ satisfies \eqref{e-gue220502yyd} and \begin{equation}\label{e-gue220502yydI}
\begin{split}
&\Box_{b,s}\hat S_{k,s}=O(k^{-\infty})\ \ \mbox{on $D\times D$},\\
&\Box_{b,s}G_{k,s}+\hat S_{k,s}=B^{(0)}_{\tilde\tau_{k\delta},s}\chi_1,
\end{split}
\end{equation}
where $\Box_{b,s}$ denotes the localization of $\Box_b$ with respect to $s$.
Note that 
\begin{equation}\label{e-gue220502yydII}
\hat\chi_2\hat S_{k,s}\hat\chi_1=O(k^{-\infty}),
\end{equation}
for every $\hat\chi_1, \hat\chi_2\in\cC^\infty_c(D)$, $\supp\hat\chi_1\cap\supp\hat\chi_2=\emptyset$. Let $\chi\in\cC^\infty_c(D)$, $\chi_1\equiv1$ near $\supp\chi$. Since $G_{k,s}$, $\hat S_{k,s}$ are properly supported on $D$, we can find $\tilde\chi\in\cC^\infty_c(D)$ 
such that $\hat S_{k,s}\chi=\tilde\chi\hat S_{k,s}\chi$, 
$G_{k,s}\chi=\tilde\chi G_{k,s}\chi$. From this observation and \eqref{e-gue220502yydI}, we get 
\begin{equation}\label{e-gue220502yydIII}
\Box_{b,s}(\tilde\chi G_{k,s}\chi)+\tilde\chi\hat S_{k,s}\chi=\tilde B^{(0)}_{\tilde\tau_{k\delta},s}\chi.
\end{equation}
Let 
\[\begin{split}
&G_{k,D}:=e^{k\phi}s^kG_{k,s}(s^{-k}e^{-k\phi}): \cC^\infty_c(D,L^k)\To\cC^\infty(D,L^k),\\
&\hat S_{k,D}:=e^{k\phi}s^k\hat S_{k,s}(s^{-k}e^{-k\phi}): \cC^\infty_c(D,L^k)\To\cC^\infty(D,L^k).
\end{split}
\]
From Lemma~\ref{l-gue220330yyd}, Lemma~\ref{l-gue220422yyd} and \eqref{e-gue220502yydIII}, we get 
\begin{equation}\label{e-gue220502yydIV}
\begin{split}
&\mbox{$\Box_b(\tilde\chi\hat S_{k,D}\chi)\equiv0\mod O(k^{-\infty})$ on $X\times X$},\\
&\Box_b(\tilde\chi G_{k,D}\chi)+\tilde\chi\hat S_{k,D}\chi=F^{(0)}_{\tilde\tau_{k\delta}}\chi+R_{k,D},
\end{split}
\end{equation}
where $R_{k,D}$ is continuous: 
\[R_{k,D}=O(k^{-\infty}):W^\ell(X,L^k)\To W^\ell(X,L^k),\]
for all $\ell\in\mathbb Z$. Assume that $X=\bigcup^N_{j=1}D_j$, where $D_j\subset X$ is an open set and there is a local CR frame $s_j$ defined on $D_j$, $\abs{s_j}^2_{h^L}=e^{-2\phi_j}$, $j=1,\ldots,N$. 
Let $\chi_j\in\cC^\infty_c(D_j)$, $j=1,\ldots,N$, $\sum^N_{j=1}\chi_j=1$ on $X$. Put $\hat S_k:=\sum^N_{j=1}\tilde\chi_j\hat S_{k,D_j}\chi_j$, $\hat G_k:=\sum^N_{j=1}\tilde\chi_j\hat G_{k,D_j}\chi_j$, 
where $\tilde\chi_j\in\cC^\infty_c(D_j)$ is as in the discussion before \eqref{e-gue220502yydIII}, $j=1,\ldots,N$. 
From \eqref{e-gue220502yydIV}, we have 
\begin{equation}\label{e-gue220502yydV}
\begin{split}
&\mbox{$\Box_b\hat S_{k}\equiv0\mod O(k^{-\infty})$ on $X\times X$},\\
&\Box_b\hat G_{k}+\hat S_{k}=F^{(0)}_{\tilde\tau_{k\delta}}+R_{k},
\end{split}
\end{equation}
where $R_{k}$ is continuous: 
\[R_{k}=O(k^{-\infty}):W^\ell(X,L^k)\To W^\ell(X,L^k),\]
for all $\ell\in\mathbb Z$. From \eqref{e-gue220502yydV}, we have $S_{\tau_{k\delta}}\hat S_k=S_{\tau_{k\delta}}+S_{\tau_{k\delta}}R_k$ and hence 
\begin{equation}\label{e-gue220502yydVI}
\mbox{$S_{\tau_{k\delta}}\hat S_k-S_{\tau_{k\delta}}\equiv0\mod O(k^{-\infty})$ on $X\times X$}.
\end{equation}
It was known that~\cite[Theorem 1.5]{Hsiao18} there is a continuous operator $H_k: L^2(X,L^k)\To\Dom\Box_b$ such that 
\[H_{k}=O(k^{n_\ell}):W^\ell(X,L^k)\To W^{\ell+1}(X,L^k),\]
for all $\ell\in\mathbb Z$, where $n_\ell\in\mathbb N$ is independent of $k$ and 
\begin{equation}\label{e-gue220502yydVII}
H_k\Box_b+S_{\tau_{k\delta}}=F^{(0)}_{\tau_{k\delta}}.
\end{equation}
From \eqref{e-gue220502yydVII}, we have 
\begin{equation}\label{e-gue220502yydVIII}
\mbox{$S_{\tau_{k\delta}}\hat S_k-F^{(0)}_{\tau_{k\delta}}\hat S_k\equiv0\mod O(k^{-\infty})$ on $X\times X$}.
\end{equation}
From \eqref{e-gue220502yydVI} and \eqref{e-gue220502yydVIII}, we get 
\begin{equation}\label{e-gue220502yyda}
\mbox{$S_{\tau_{k\delta}}-F^{(0)}_{\tau_{k\delta}}\hat S_k\equiv0\mod O(k^{-\infty})$ on$X\times X$}.
\end{equation} 
From Lemma~\ref{l-gue220330yyd} and the construction of $\hat S_k$, we see that $F^{(0)}_{\tau_{k\delta}}\hat S_k$ satisfies \eqref{e-gue220502yydb}. From this observation and \eqref{e-gue220502yyda}, we get \eqref{e-gue220502yydb}. 
\end{proof}

Let $A_k, B_k: \cC^\infty_c(D)\To\cC^\infty(D)$
be properly supported continuous operators with distribution kernels 
\[A_k(x,y)=\int_{\R}e^{ik\Phi(x,y,t)}a(x,y,t,k)dt,\ \ B_k(x,y)=\int_{\R}e^{ik\Phi(x,y,t)}b(x,y,t,k)dt,\] 
where $\Phi(x,y,t)\in\cC^\infty(D\times D\times(\delta/4,2\delta))$ is as in Theorem~\ref{t-gue220430yyd} and $a(x,y,t,k), b(x,y,t,k)\in S^n(1;D\times D\times(\delta/4,2\delta))\cap\cC^\infty_c(D\times D\times(\delta/4,2\delta))$. From complex stationary phase formula, we can repeat the procedure in~\cite[p.64--p.65]{Hsiao18} and get 
\begin{equation}\label{e-gue220502yydc}
(A_k\circ B_k)(x,y)=\int_{\R}e^{ik\Phi_1(x,y,t)}c(x,y,t,k)dt+O(k^{-\infty}), 
\end{equation}
where $\Phi_1(x,y,t)\in\cC^\infty(D\times D\times(\delta/4,2\delta))$, ${\rm Im\,}\Phi_1\geq0$, and for every $(x,x,t)\in D\times D\times(\delta/4,2\delta)$, 
we have
\[\Phi(x,x,t)=0,\ \ d_x\Phi_1(x,y,t)|_{x=y}=d_x\Phi(x,y,t)|_{x=y},\ \ d_y\Phi_1(x,y,t)|_{x=y}=d_y\Phi(x,y,t)|_{x=y}.\]
Recall that $\Phi(x,y,t)$ has the form \eqref{e-gue220504ycdm}. It is not difficult to see that we can take $\Phi_1(x,y,t)$ so that 
\begin{equation}\label{e-gue220504yydn}
\Phi_1(x,y,t)=x_{2n-1}t+\hat\Phi_1(x',y,t),\ \ \hat\Phi_1(x',y,t)\in\cC^\infty(D\times D\times(\delta/4,2\delta)).
\end{equation}
From now on, we assume that $\Phi_1(x,y,t)$ has the form \eqref{e-gue220504yydn}. In the proof of Theorem~\ref{t-gue220512yyd}, we need 

\begin{lemma}\label{l-gue220502yyd}
With the notations used above, $\Phi(x,y,t)-\Phi_1(x,y,t)$ vanishes to infinite order at $x=y$, for every $(x,x,t)\in D\times D\times(\delta/4,2\delta)$.
\end{lemma}

\begin{proof}
Fix $x_0\in D$ and $t_0\in(\delta/4,2\delta)$. We may assume that $x(p)=0$ and $\tau_\delta(t_0)=1$. Let $\chi\in\cC^\infty_c(D)$, $\chi\equiv1$ near $0$ and let $\chi_1, \tilde\chi_1\in\cC^\infty_c(D)$, $\chi_1\equiv1$ near $\supp\chi$, $\tilde\chi_1\equiv1$ near $\supp\chi_1$. From Theorem~\ref{t-gue220501yyd}, we can check that 
\begin{equation}\label{e-gue220503yydII}
(\chi_1S_{\tau_{k\delta},s}\tilde\chi_1)\circ(\chi_1S_{\tau_{k\delta},s}\chi)=\chi_1S_{\tau^2_{k\delta},s}\chi+O(k^{-\infty}).
\end{equation} 
From Theorem~\ref{t-gue220430yyd}, we have 
\begin{equation}\label{e-gue220503yydIII}
\begin{split}
&(\chi_1S_{\tau_{k\delta},s}\tilde\chi_1)\circ(\chi_1S_{\tau_{k\delta},s}\chi)=\int_{\R}e^{ik\Phi_1(x,y,t)}a(x,y,t,k)dt+O(k^{-\infty}),\\
&a(x,y,t,k)\in S^n(1;D\times D\times(\delta/4,2\delta))\cap\cC^\infty_c(D\times D\times(\delta/4,2\delta)),\\
&\mbox{$a(x,y,t,k)\sim\sum^{+\infty}_{j=0}a_j(x,y,t)k^{n-j}$ in $S^n(1;D\times D\times(\delta/4,2\delta))$},\\
 &a_j(x,y,t)\in S^0(1;D\times D\times(\delta/4,2\delta))\cap\cC^\infty_c(D\times D\times(\delta/4,2\delta)),\ \ j=0,1,\ldots,\\
 &a_0(x_0,x_0,t_0)\neq0. 
\end{split}
\end{equation} 
From Theorem~\ref{t-gue220430yyd}, we have 
\begin{equation}\label{e-gue220503yydIV}
\begin{split}
&(\chi_1S_{\tau^2_{k\delta},s}\chi)(x,y)=\int_{\R}e^{ik\Phi(x,y,t)}b(x,y,t,k)dt+O(k^{-\infty}),\\
&b(x,y,t,k)\in S^n(1;D\times D\times(\delta/4,2\delta))\cap\cC^\infty_c(D\times D\times(\delta/4,2\delta)),\\
&\mbox{$b(x,y,t,k)\sim\sum^{+\infty}_{j=0}b_j(x,y,t)k^{n-j}$ in $S^n(1;D\times D\times(\delta/4,2\delta))$},\\
 &b_j(x,y,t)\in S^0(1;D\times D\times(\delta/4,2\delta))\cap\cC^\infty_c(D\times D\times(\delta/4,2\delta)),\ \ j=0,1,\ldots,\\
 &b_0(x_0,x_0,t_0)\neq0. 
\end{split}
\end{equation} 
Thus, 
\begin{equation}\label{e-gue220503yydV}
\int_{\R}e^{ik\Phi_1(x,y,t)}a(x,y,t,k)dt=
\int_{\R}e^{ik\Phi(x,y,t)}b(x,y,t,k)dt+O(k^{-\infty}).
\end{equation}
Without loss of generality, assume that $D=D'\times(-\varepsilon,\varepsilon)$, 
$D'$ is an open set of $0$ in $\R^{2n-2}$, $\varepsilon>0$. 
Let $\chi_0\in\cC^\infty_c((-\varepsilon,\varepsilon))$, $\chi_0(0)=1$. 
Consider the continuous operators: 
\[\begin{split}
\hat A_k: \cC^\infty_c(D)&\To\cC^\infty_c(D'),\\
u&\To\int e^{-ikx_{2n-1}t_0}\chi_0(x_{2n-1})
e^{ik\Phi_1(x,y,t)}a(x,y,t,k)u(y)dydtdx_{2n-1},
\end{split}\]
\[\begin{split}
\hat B_k: \cC^\infty_c(D)&\To\cC^\infty_c(D'),\\
u&\To\int e^{-ikx_{2n-1}t_0}\chi_0(x_{2n-1})e^{ik\Phi(x,y,t)}b(x,y,t,k)u(y)dydtdx_{2n-1}.\end{split}\]
Recall that $\Phi$ and $\Phi_1$ have the forms 
\eqref{e-gue220504ycdm} and \eqref{e-gue220504yydn}. 
From this observation and complex stationary phase formula, 
we get 
\[\hat A_k(x',y)=e^{ik\hat\Phi_1(x',y,t_0)}\alpha(x',y,k)+
O(k^{-\infty}),\ \ \hat B_k(x',y)=e^{ik\hat\Phi(x',y,t_0)}
\beta(x',y,k)+O(k^{-\infty}),\]
$\alpha, \beta\in S^{n-1}(1; D'\times D)$, $\alpha\sim\sum^{+\infty}_{j=0}\alpha_j(x',y)k^{n-1-j}$ in $S^{n-1}(1;D'\times D)$, 
$\beta\sim\sum^{+\infty}_{j=0}\beta_j(x',y)k^{n-1-j}$ in 
$S^{n-1}(1;D'\times D)$, $\alpha_j, \beta_j\in\cC^\infty(D'\times D)$, 
$j=0,1,\ldots$, 
$\alpha_0(0,0)\neq0$, $\beta_0(0,0)\neq0$, where $\hat A_k(x',y)$, 
$\hat B_k(x',y)$ denote the distribution kernels of $\hat A_k$ and 
$\hat B_k$ respectively. 
From \eqref{e-gue220503yydV}, we can check that 
$\hat A_k-\hat B_k=O(k^{-\infty})$. From this observation and 
$\alpha_0(0,0)\neq0$, $\beta_0(0,0)\neq0$, we can repeat the 
procedure in~\cite[p.150]{HM17} and deduce that 
$\hat\Phi(x',y,t_0)-\hat\Phi_1(x',y,t_0)$ vanishes 
to infinite order at $(x_0,x_0)$. The lemma folllows. 
\end{proof}

\subsection{Bergman kernel asymptotic expansion}\label{s-gue220503yyd}

In this section, we will prove our main results. Let $s$ 
be a $\R$-equivariant holomorphic frame of $L$ defined 
on an open set $U$ of $M'$ 
with $U\cap X=:D\neq\emptyset$. On $U$, let 
$\set{e^{\ol j}}^n_{j=1}$ be an orthonormal basis for 
$T^{*0,1}M'$ with $e^{\ol n}=\frac{\ddbar\rho}{\abs{\ddbar\rho}}$. 
Let $\set{L_{\ol j}}^n_{j=1}\subset T^{0,1}M'$ 
be the dual basis of $\set{e^{\ol j}}^n_{j=1}$. 
Then, $L_{\ol n}=\frac{\frac{\pr}{\pr\rho}+
iJ(\frac{\pr}{\pr\rho})}{\abs{\frac{\pr}{\pr\rho}+
iJ(\frac{\pr}{\pr\rho})}}$. On $\Omega^{0,1}(U,L^k)$, 
we have $\ddbar_k=\sum^n_{j=1}(e^{\ol j})\wedge L_{\ol j}+
P_0$ and 
\begin{equation}\label{e-gue220503yyda} 
\ddbar^*_{f,k}=\sum^n_{j=1}(e^{\ol j}\wedge)^*L^*_{\ol j}+P^*_0,
\end{equation}
where $L^*_{\ol j}$ is the adjoint of $L_{\ol j}$ 
with respect to $(\,\cdot\,|\,\cdot\,)_{M',k}$, $j=1,\ldots,n$, 
$P_0$ is a $k$-independent zero order differential operator, 
$P^*_0$ is the adjoint of $P_0$ with respect to 
$(\,\cdot\,|\,\cdot\,)_{M',k}$. From \eqref{e-gue220410yydI}, 
we can check that 
\begin{equation}\label{e-gue220503ycd}
\ddbar^*_{f,k}=\sum^{n-1}_{j=1}(e^{\ol j}\wedge)^*L^*_{\ol j}+
\Bigr(((\ddbar\rho)\wedge)^*\circ((R^{(1)})^*+iT)\Bigr)+kZ_1+Z_0,
\end{equation}
where $Z_1$ is a $k$-independent first order differential operator 
with $Z_1|_X=0$ and $Z_0$ is a $k$-independent zero order differential operator. 
Let $\tilde\tau_\delta, \tau_\delta\in\cC^\infty_c((\delta/4,2\delta))$, 
$0\leq\tilde\tau_\delta, \tau_\delta\leq1$, $\tilde\tau_\delta\equiv1$
near $\supp\tau_\delta$, $\tau_\delta\equiv1$ on $[\frac{\delta}{2},\delta]$. 
From \eqref{e-gue220410yydII} and \eqref{e-gue220503ycd}, we get 
\begin{equation}\label{e-gue220503ycdI}
\gamma\ol{\pr}^*_{f,k}P^{(1)}_kF^{(1)}_{\tilde\tau^2_{k\delta}}
=\sum^{n-1}_{j=1}(e^{\ol j})^*\circ L^*_j+
\Bigr(((\ddbar\rho)\wedge)^*\circ
(\mathcal{N}^{(1)}_{\tilde\tau^2_{k\delta}}+iT)\Bigr)+Z_0.
\end{equation}
Now, consider $\gamma\ddbar_kP^{(0)}_kS_{\tau^2_{k\delta}}$. From $\ddbar_bS_{\tau^2_{k\delta}}=0$, we can check that 
\[\gamma\ddbar_kP^{(0)}_kS_{\tau^2_{k\delta}}=2((\ddbar\rho)\wedge)\circ((\ddbar\rho)\wedge)^*\gamma\ddbar P^{(0)}_kS_{\tau^2_{k\delta}}.\]
From this observation and \eqref{e-gue220503ycdI}, we get 
\begin{equation}\label{e-gue220503ycdII}
\begin{split}
0&=2((\ddbar\rho)\wedge)\circ\gamma\ol{\pr}^*_{f,k}\,\ddbar P^{(0)}_kS_{\tau^2_{k\delta}}\\
&=2((\ddbar\rho)\wedge)\circ\gamma\ol{\pr}^*_{f,k}\,P^{(1)}_kF^{(1)}_{\tilde\tau^2_{k\delta}}\gamma\ddbar P^{(0)}_kS_{\tau^2_{k\delta}}\\
&=(\mathcal{N}^{(1)}_{\tilde\tau^2_{k\delta}}+iT+\hat Z_0)\gamma\ddbar P^{(0)}_kS_{\tau^2_{k\delta}},
\end{split}
\end{equation}
where $\hat Z_0$ is a $k$-independent zero order differential operator. We need 

\begin{lemma}\label{l-gue220503ycd}
There is a constant $C>0$ independent of $k$ such that 
\begin{equation}\label{e-gue220503ycdIII}
\norm{\gamma\ddbar_kP^{(0)}_kS_{\tau^2_{k\delta}}u}_{1,X,k}\leq C\Bigr(k\norm{\gamma\ddbar_kP^{(0)}_kS_{\tau^2_{k\delta}}u}_{X,k}+\norm{\gamma\ddbar P^{(0)}_kS_{\tau^2_{k\delta}}u}_{X,k}\Bigr),
\end{equation}
for every $u\in\cC^\infty(X,L^k)$. 
\end{lemma}

\begin{proof}
Let $u\in\cC^\infty(X,L^k)$. From \eqref{e-gue220503ycdII}, we have 
\begin{equation}\label{e-gue220503ycdIV}
\mathcal{N}^{(1)}_{\tilde\tau^2_{k\delta}}\gamma\ddbar_kP^{(0)}_kS_{\tau^2_{k\delta}}u=(-iT-\hat Z_0)\gamma\ddbar_kP^{(0)}_kS_{\tau^2_{k\delta}}u.
\end{equation}
From \eqref{e-gue220415ycdg} and \eqref{e-gue220503ycdIV}, we get \eqref{e-gue220503ycdIII}. 
\end{proof}

We can now prove 

\begin{theorem}\label{t-gue220503ycd}
There is a $k_0\in\mathbb N$ such that for all $k\geq k_0$, $\ddbar_kP^{(0)}_kS_{\tau^2_{k\delta}}=0$. 
\end{theorem}

\begin{proof}
Let $u\in\cC^\infty(X,L^k)$ and put $v:=\gamma\ddbar_kP^{(0)}_kS_{\tau^2_{k\delta}}u\in\cC^\infty(X,T^{*0,1}M'\otimes L^k)$. 
From \eqref{e-gue220411yydmz} and \eqref{e-gue220503ycdII}, we have 
\begin{equation}\label{e-gue220504yyd}
\begin{split}
(\,-iTv\,|\,v\,)_{X,k}=(\,\mathcal{N}^{(1)}_{\tilde\tau^2_{k\delta}}v\,|\,v\,)_{X,k}+(\,\hat Z_0v\,|\,v\,)_{X,k}\\
\leq\frac{C}{k}\Bigr(\norm{v}_{1,X,k}\norm{v}_{X,k}+\norm{v}^2_{X,k}+k\norm{v}^2_{X,k}+\frac{1}{k^2}\norm{v}^2_{1,X,k}\Bigr),
\end{split}
\end{equation}
where $C>0$ is a constant independent of $k$ and $v$. From \eqref{e-gue220503ycdIII}, there is a constant $\hat C>0$ independent of $k$ and $v$ such that 
\begin{equation}\label{e-gue220504yydI}
\norm{v}_{1,X,k}\leq\hat C\Bigr(k\norm{v}_{X,k}+\norm{v}_{X,k}\Bigr).
\end{equation} 
From \eqref{e-gue220504yyd} and \eqref{e-gue220504yydI}, we get 
\begin{equation}\label{e-gue220504yydII}
(\,-iTv\,|\,v\,)_{X,k}\leq\tilde C\norm{v}^2_{X,k},
\end{equation}
where $\tilde C>0$ is a constant independent of $k$ and $v$. Since 
\[v\in\oplus_{\alpha\in{\rm Spec\,}(-iT), \alpha\in]\frac{k\delta}{4},2k\delta[}L^2_\alpha(X,L^k),\] 
we have 
\begin{equation}\label{e-gue220504yydIII}
(\,-iTv\,|\,v\,)_{X,k}\geq\frac{k\delta}{4}\norm{v}^2_{X,k}. 
\end{equation}
From \eqref{e-gue220504yydII} and \eqref{e-gue220504yydIII}, we conclude that if $\frac{k\delta}{4}>\tilde C$, then $v=0$. The theorem follows. 
\end{proof} 

\begin{proof}[Proof of Theorem~\ref{t-gue220513yyd}]
We take $\hat\tau_{k\delta}$ so that $\supp\hat\tau_{k\delta}\subset I_k$. From Theorem~\ref{t-gue220503ycd}, we obtain that 
\begin{equation}\label{e-gue220517ycds}
\int_X S_{\hat\tau^2_{k\delta}}(x,x)dv_X(x)\leq \dim H^0_{I_k}(\ol M,L^k)\leq\int _XS_{\tau^2_{k\delta}}(x,x)dv_X(x), 
\end{equation}
where $S_{\hat\tau^2_{k\delta}}(x,x):=S_{\hat\tau^2_{k\delta},s}(x,x)\in\cC^\infty(X)$, for every $\R$-equivariant CR frame, and similar for $S_{\tau^2_{k\delta}}(x,x)$.
From Theorem~\ref{t-gue220430yyd} and \eqref{e-gue220517ycds}, we get Theorem~\ref{t-gue220513yyd}. 
\end{proof}

From now on, we assume that $k$ is large enough so that $\ddbar_kP^{(0)}_kS_{\tau^2_{k\delta}}=0$.

Put 
\begin{equation}\label{e-gue220504yydIV}
A_{\tau^2_{k\delta}}:=S_{\tau^2_{k\delta}}((P^{(0)}_k)^*P^{(0)}_k)^{-1}: \cC^\infty(X,L^k)\To\cC^\infty(X,L^k).
\end{equation}  
From Lemma~\ref{l-gue220425yydI}, Theorem~\ref{t-gue220430yyd}, Theorem~\ref{t-gue220501yyd} and by using complex stationary phase formula, we get 

\begin{lemma}\label{l-gue220504yyd}
Let $\tilde\chi, \hat\chi\in\cC^\infty(X)$ with $\supp\tilde\chi\cap\supp\hat\chi=\emptyset$. Then, $\tilde\chi A_{\tau^2_{k\delta}}\hat\chi\equiv0\mod O(k^{-\infty})$ on $X\times X$. Let $s$ be a $\R$-equivariant CR frame of $L$ on an open set 
$D$ of $X$, $\abs{s}^2_{h^L}=e^{-2\phi}$. Let $A_{\tau^2_{k\delta},s}$ be the localization of $A_{\tau^2_{k\delta}}$ with respect to $s$ and let 
$A_{\tau^2_{k\delta},s}(x,y)\in\cC^\infty(D\times D)$ be the distribution kernel of $A_{\tau^2_{k\delta},s}$. Let $\chi, \chi_1\in\cC^\infty_c(D)$, $\chi_1\equiv1$ near $\supp\chi$. We have 
\begin{equation}\label{e-gue220505yyd}
(\chi_1A_{\tau^2_{k\delta},s}\chi)(x,y)=\int_{\R}e^{ik\Phi(x,y,t)}a(x,y,t,k)dt+O(k^{-\infty}),
\end{equation}
where $\Phi\in\cC^\infty(D\times D\times(\delta/4,2\delta))$ is as in Theorem~\ref{t-gue220430yyd}, 
\begin{equation}\label{e-gue220505yydI}
\begin{split}
&a(x,y,t,k)\in S^{n+1}(1;D\times D\times(\delta/4,2\delta))\cap\cC^\infty_c(D\times D\times(\delta/4,2\delta)),\\
&\mbox{$a(x,y,t,k)\sim\sum^{+\infty}_{j=0}a_j(x,y,t)k^{n+1-j}$ in $S^{n+1}(1;D\times D\times(\delta/4,2\delta))$},\\
 &a_j(x,y,t)\in S^0(1;D\times D\times(\delta/4,2\delta))\cap\cC^\infty_c(D\times D\times(\delta/4,2\delta)),\ \ j=0,1,\ldots,\\
 &a_0(x,x,t)=(2\pi)^{-n}\chi(x)\abs{\det(R^L_{X,x}+2t\mathcal{L}_x)}\abs{\tau^2_\delta(t)}^2(2i\varphi_1(x,x,d_x\Phi(x,x,t)),\\ 
 &\mbox{for all $(x,x,t)\in D\times D\times(\delta/4,2\delta)$}, 
\end{split}
\end{equation} 
where $\varphi_1$ is as in \eqref{e-gue220426yyd}. 
\end{lemma}

We now consider 
\begin{equation}\label{e-gue220509yyd}
B_{\tau^2_{k\delta}}:=P^{(0)}_kA_{\tau^2_{k\delta}}(P^{(0)}_k)^*=P^{(0)}_kS_{\tau^2_{k\delta}}((P^{(0)}_k)^*P^{(0)}_k)^{-1}(P^{(0)}_k)^*: \cC^\infty(\ol M,L^k)\To\cC^\infty(\ol M,L^k). 
\end{equation}
From Lemma~\ref{l-gue220314yydz} and Lemma~\ref{l-gue220504yyd}, we can repeat the proof of Theorem~\ref{t-gue220406yyd} with minor changes and deduce 

\begin{lemma}\label{l-gue220505yyd}
Let $\tilde\chi, \hat\chi\in\cC^\infty(\ol M)$ with $\supp\tilde\chi\cap\supp\hat\chi=\emptyset$. Then, 
\begin{equation}\label{e-gue220510ycd}
\mbox{$\tilde\chi B_{\tau^2_{k\delta}}\hat\chi\equiv0\mod O(k^{-\infty})$ on $\ol M\times\ol M$}. 
\end{equation} 
Let $s$ be a $\R$-equivariant holomorphic frame of $L$ on an open set 
$U$ of $M'$, $\abs{s}^2_{h^L}=e^{-2\phi}$. If $U\cap X=\emptyset$. Then, 
\[\mbox{$B_{\tau^2_{k\delta}}\equiv0\mod O(k^{-\infty})$ on $(U\times U)\cap(\ol M\times\ol M)$}.\]
Suppose that $D:=U\cap X\neq\emptyset$. Let $x=(x_1,\ldots,x_{2n-1})$ be canonical coordinates of $X$ on $D$ with $T=\frac{\pr}{\pr x_{2n-1}}$ on $D$ and $\tilde x=(x_1,\ldots,x_{2n-1},x_{2n})$, $x_{2n}=\rho$, are local coordinates of $U$. Let $B_{\tau^2_{k\delta},s}$ be the localization of $B_{\tau^2_{k\delta}}$ with respect to $s$. 
Let $\chi, \chi_1\in\cC^\infty_c(U\cap\ol M)$, $\chi_1\equiv1$ near $\supp\chi$. Then, 
\begin{equation}\label{e-gue220505yydII}
\mbox{$(\chi_1B_{\tau^2_{k\delta},s}\chi)(\tilde x,\tilde y)\equiv\int_{\R}e^{ik\Psi(\tilde x,\tilde y,t)}\beta(\tilde x,\tilde y,t,k)dt\mod O(k^{-\infty})$ on $(U\times U)\cap(\ol M\times\ol M)$},
\end{equation}
where $(\chi_1B_{\tau^2_{k\delta},s}\chi)(\tilde x,\tilde y)$
is the distribution kernel of $\chi_1B_{\tau^2_{k\delta},s}\chi$, 
$\Psi(\tilde x,\tilde y, t)\in\cC^\infty(((U\times U)\cap(\ol M\times\ol M))
\times(\delta/4,2\delta))$, $\Psi|_{D\times D}=\Phi$, 
where $\Phi$ is as in Theorem~\ref{t-gue220430yyd},
\begin{equation}\label{e-gue220505yydIII}
\begin{split}
&\Psi(\tilde x,\tilde y, t)=\Phi(x,y,t)+x_{2n}\varphi_1(x,y,d_x\Phi(x,y,t))-y_{2n}\ol{\varphi_1(y,x,-d_y\ol\Phi(x,y,t))}\\
&\quad+O(\abs{(x_{2n},y_{2n})}^2),\\
&{\rm Im\,}\Psi(\tilde x,\tilde y, t)\geq C(\abs{x_{2n}}+\abs{y_{2n}}),
\end{split}
\end{equation}
for every $(\tilde x, \tilde y,t)\in((U\times U)\cap(\ol M\times\ol M))
\times(\delta/4,2\delta)$, where $C>0$ is a constant, 
\begin{equation}\label{e-gue220505yydIV}
\begin{split}
&\beta(\tilde x,\tilde y,t,k)\\
&\in S^{n+1}(1;((U\times U)\cap(\ol M\times\ol M))
\times(\delta/4,2\delta))\cap\cC^\infty_c(((U\times U)
\cap(\ol M\times\ol M))\times(\delta/4,2\delta)),\\
&\mbox{$\beta(\tilde x,\tilde y,t,k)\sim
\sum^{+\infty}_{j=0}\beta_j(\tilde x,\tilde y,t)k^{n+1-j}$ 
in $S^{n+1}(1;((U\times U)\cap(\ol M\times\ol M))
\times(\delta/4,2\delta))$},\\
 &\beta_j(\tilde x,\tilde y,t)\in S^0(1;((U\times U)
 \cap(\ol M\times\ol M))\times(\delta/4,2\delta))
 \cap\cC^\infty_c(((U\times U)\cap(\ol M\times\ol M))
 \times(\delta/4,2\delta)),\\ 
 &j=0,1,\ldots,
 \end{split}\end{equation}
 \begin{equation}\label{e-gue220505yydV}
 \beta_0(x,x,t)=(2\pi)^{-n}\chi(x)|\det(R^L_{X,x}+2t\mathcal{L}_x)|
 |\tau^2_\delta(t)|^2(2i\varphi_1(x,x,d_x\Phi(x,x,t)),
\end{equation} 
for all $(x,x,t)\in D\times D\times(\delta/4,2\delta)$, where $\varphi_1$ is as in \eqref{e-gue220426yyd}.
\end{lemma}

As before, let 
\(\Pi_k: L^2(M,L^k)\To{\rm Ker\,}\ddbar_k\)
be the orthogonal projection (Bergman projection) and put 
\[\Pi_{\tau_{k\delta}}:=\Pi_k\circ F^{(0)}_{\tau_{k\delta},M}:
L^2(M,L^k)\To{\rm Ker\,}\ddbar_k.\]
Let $\tau^\dagger_\delta\in\cC^\infty((\delta/4,2\delta))$, 
$0\leq\tau^\dagger_\delta\leq1$, $\tau^\dagger_\delta\equiv1$ on 
$[\delta/2,\delta]$,
$\tau_\delta\equiv1$ near $\supp\tau^\dagger_\delta$. 
We consider $\Pi_{\tau^\dagger_{k\delta}}$. 
Since $\ddbar_b\gamma\Pi_{\tau^\dagger_{k\delta}}=0$ and 
\[\gamma\Pi_{\tau^\dagger_{k\delta}}
=((P^{(0)}_k)^*P^{(0)}_k)^{-1}(P^{(0)}_k)^*
\Pi_{\tau^\dagger_{k\delta}},\]
we have $B_{\tau^2_{k\delta}}\Pi_{\tau^\dagger_{k\delta}}
=\Pi_{\tau^\dagger_{k\delta}}$ and hence 
\begin{equation}\label{e-gue220505yydVI}
\Pi_{\tau^\dagger_{k\delta}}=
\Pi_{\tau^\dagger_{k\delta}}B^*_{\tau^2_{k\delta}}
=\Pi_{\tau^\dagger_{k\delta}}R_{\tau^2_{k\delta}}+
\Pi_{\tau^\dagger_{k\delta}}B_{\tau^2_{k\delta}},
\end{equation}
where $B^*_{\tau^2_{k\delta}}$ is the adjoint of 
$B_{\tau^2_{k\delta}}$ with respect to $(\,\cdot\,|\,\cdot\,)_{M,k}$ 
and $R_{\tau^2_{k\delta}}:=
B^*_{\tau^2_{k\delta}}-B_{\tau^2_{k\delta}}$. 
From Theorem~\ref{t-gue220503ycd}, we can check that 
\[\Pi_{\tau^\dagger_{k\delta}}B_{\tau^2_{k\delta}}=
F^{(0)}_{\tau^\dagger_{k\delta},M}B_{\tau^2_{k\delta}}.\]
From this observation and \eqref{e-gue220505yydVI}, we get 
\[\Pi_{\tau^\dagger_{k\delta}}(I-R_{\tau^2_{k\delta}})
=F^{(0)}_{\tau^\dagger_{k\delta},M}B_{\tau^2_{k\delta}}\]
and hence 
\begin{equation}\label{e-gue220505yydVII}
\Pi_{\tau^\dagger_{k\delta}}(I-R^N_{\tau^2_{k\delta}})
=F^{(0)}_{\tau^\dagger_{k\delta},M}
B_{\tau^2_{k\delta}}(I+R_{\tau^2_{k\delta}}+
\ldots+R^{N-1}_{\tau^2_{k\delta}}),
\end{equation}
for every $N\in\mathbb N$. 
By\eqref{e-gue220505yydVII} 
$\Pi_{\tau^\dagger_{k\delta}}$ satisfies \eqref{e-gue220510ycd}. 
Let $s$ be an $\R$-equivariant 
holomorphic frame of $L$ on an open set 
$U$ of $M'$, $\abs{s}^2_{h^L}=e^{-2\phi}$. 
If $U\cap X=\emptyset$. From Lemma~\ref{l-gue220505yyd} 
and \eqref{e-gue220505yydVII}, we conclude that
\[\mbox{$\Pi_{\tau^\dagger_{k\delta}}\equiv0\mod O(k^{-\infty})$
on $(U\times U)\cap(\ol M\times\ol M)$}.\]
Suppose that $D:=U\cap X\neq\emptyset$. 
Let $x=(x_1,\ldots,x_{2n-1})$ be canonical coordinates of 
$X$ on $D$ with $T=\frac{\pr}{\pr x_{2n-1}}$ on $D$ and 
$\tilde x=(x_1,\ldots,x_{2n-1},x_{2n})$, $x_{2n}=\rho$, 
are local coordinates of $U$. Lemma~\ref{l-gue220505yyd} implies 
\begin{equation}\label{e-gue220510ycda}
\mbox{$R_{\tau^2_{k\delta},s}(\tilde x,\tilde y)\equiv
\int_{\R}e^{ik\Psi(\tilde x,\tilde y,t)}r(\tilde x,\tilde y,t,k)dt
\mod O(k^{-\infty})$ on $(U\times U)\cap(\ol M\times\ol M)$},
\end{equation}
where $R_{\tau^2_{k\delta},s}$ is the localization of 
$R_{\tau^2_{k\delta}}$ with respect to $s$, 
$R_{\tau^2_{k\delta},s}(\tilde x,\tilde y)$ 
is the distribution kernel of $R_{\tau^2_{k\delta},s}$, 
$\Psi(\tilde x,\tilde y, t)\in\cC^\infty(((U\times U)\cap(\ol M\times\ol M))
\times(\delta/4,2\delta))$ is as in Lemma~\ref{l-gue220505yyd}, 
\begin{equation}\label{e-gue220505yydm}
\begin{split}
&r(\tilde x,\tilde y,t,k)\\
&\in S^{n+1}(1;((U\times U)\cap(\ol M\times\ol M))\times(\delta/4,2\delta))\cap\cC^\infty_c(((U\times U)\cap(\ol M\times\ol M))\times(\delta/4,2\delta)),\\
&\mbox{$r(\tilde x,\tilde y,t,k)\sim\sum^{+\infty}_{j=0}r_j(\tilde x,\tilde y,t)k^{n+1-j}$ in $S^{n+1}(1;((U\times U)\cap(\ol M\times\ol M))\times(\delta/4,2\delta))$},\\
 &r_j(\tilde x,\tilde y,t)\in S^0(1;((U\times U)\cap(\ol M\times\ol M))\times(\delta/4,2\delta)),\ \ j=0,1,\ldots,\\
 &\mbox{$r_j(\tilde x,\tilde y,t)=0$ if $t\notin\supp\tau_\delta$, $j=0,1,\ldots$},\\
 &\mbox{$r_0(x,x,t)=0$, for every $(x,x,t)\in D\times D\times(\delta/4,2\delta)$}.
 \end{split}\end{equation}
 For every $\ell\in\mathbb N_0$, let 
 $(F^{(0)}_{\tau^\dagger_{k\delta},M}B_{\tau^2_{k\delta}}R^\ell_{\tau^2_{k\delta}})_s$ be the localization of 
 $F^{(0)}_{\tau^\dagger_{k\delta},M}
 B_{\tau^2_{k\delta}}R^\ell_{\tau^2_{k\delta}}$ 
 with respect to $s$ and let $(F^{(0)}_{\tau^\dagger_{k\delta},M}
 B_{\tau^2_{k\delta}}R^\ell_{\tau^2_{k\delta}})_s(\tilde x,\tilde y)$ 
 be the distribution kernel of $(F^{(0)}_{\tau^\dagger_{k\delta},M}
 B_{\tau^2_{k\delta}}R^\ell_{\tau^2_{k\delta}})_s$. 
 From the complex stationary phase formula, \eqref{e-gue220422yydb}, 
 \eqref{e-gue220422yydc}, 
Lemma~\ref{l-gue220502yyd}, Lemma~\ref{l-gue220505yyd} 
and \eqref{e-gue220510ycda}, we have 
\begin{equation}\label{e-gue220510ycdb}
\mbox{$(F^{(0)}_{\tau^\dagger_{k\delta},M}
B_{\tau^2_{k\delta}}R^\ell_{\tau^2_{k\delta}})_s
(\tilde x,\tilde y)\equiv\int_{\R}e^{ik\Psi(\tilde x,\tilde y,t)}
\alpha^\ell(\tilde x,\tilde y,t,k)dt\mod O(k^{-\infty})$ 
on $(U\times U)\cap(\ol M\times\ol M)$},
\end{equation}
\begin{equation}\label{e-gue220510ycdbz}
\begin{split}
&\alpha^\ell(\tilde x,\tilde y,t,k)\\
&\in S^{n+1}(1;((U\times U)\cap(\ol M\times\ol M))\times(\delta/4,2\delta))\cap\cC^\infty_c(((U\times U)\cap(\ol M\times\ol M))\times(\delta/4,2\delta)),\\
&\mbox{$\alpha^\ell(\tilde x,\tilde y,t,k)\sim\sum^{+\infty}_{j=0}\alpha^\ell_j(\tilde x,\tilde y,t)k^{n+1-j}$ in $S^{n+1}(1;((U\times U)\cap(\ol M\times\ol M))\times(\delta/4,2\delta))$},\\
 &\alpha^\ell_j(\tilde x,\tilde y,t)\in S^0(1;((U\times U)\cap(\ol M\times\ol M))\times(\delta/4,2\delta)),\ \ j=0,1,\ldots,\\
 &\mbox{$\alpha^\ell_j(\tilde x,\tilde y,t)=0$ if $t\notin\supp\tau_\delta$, $j=0,1,\ldots$}. 
 \end{split}\end{equation}
Since $r_0(x,x,t)=0$, for every $(x,x,t)\in D\times D\times(\delta/4,2\delta)$, for every $\ell_1, \ell_2\in\mathbb N$, there is a $N_{\ell_1,\ell_2}>0$ 
such that for every $N\geq N_{\ell_1,\ell_2}$, we have 
\begin{equation}\label{e-gue220505yydn}
\norm{(\Pi_{\tau^\dagger_{k\delta}}R^N_{\tau^2_{k\delta}})_s(\tilde x,\tilde y)}_{\cC^{\ell_1}((U\times U)\cap(\ol M\times\ol M))}\leq Ck^{n+1-\ell_2},
\end{equation}
where $C>0$ is a constant independent of $k$, $(\Pi_{\tau^\dagger_{k\delta}}R^N_{\tau^2_{k\delta}})_s$ is the localization of $\Pi_{\tau^\dagger_{k\delta}}R^N_{\tau^2_{k\delta}}$ with respect to $s$, $(\Pi_{\tau^\dagger_{k\delta}}R^N_{\tau^2_{k\delta}})_s(\tilde x,\tilde y)$ is the distribution kernel of $(\Pi_{\tau^\dagger_{k\delta}}R^N_{\tau^2_{k\delta}})_s$. 

Since  $r_0(x,x,t)=0$, for every $(x,x,t)\in D\times D\times(\delta/4,2\delta)$, from complex stationary phase formula, we can check that for every $m\in\mathbb N$, there is a $N_m\in\mathbb N$, such that for all $N\geq N_m$, 
\begin{equation}\label{e-gue220505ycd}
\pr^{\alpha}_{\tilde x}\pr^\beta_{\tilde y}\alpha^N_j(\tilde x,\tilde y,t)|_{\tilde x=x, \tilde y=y}=0,
\end{equation}
for every $(x,x,t)\in D\times D\times(\delta/4,2\delta)$, for all $j=0,1,\ldots,m$, for all $\alpha, \beta\in\mathbb N^{2n}_0$, $\abs{\alpha}+\abs{\beta}\leq m$. From \eqref{e-gue220505ycd}, we can apply Borel construction and find 
$b_j(\tilde x,\tilde y,t)\in S^0(1;((U\times U)\cap(\ol M\times\ol M))\times(\delta/4,2\delta))$, $b_j(\tilde x,\tilde y,t)=0$ if $t\notin\supp\tau_\delta$, $j=0,1,\ldots$, such that for any $j=0,1\ldots$, every $\ell\in\mathbb N$, there is a $N_\ell\in\mathbb N$ such that for all $N\geq N_\ell$,  
\begin{equation}\label{e-gue220510ycdp}
\pr^{\alpha}_{\tilde x}\pr^\beta_{\tilde y}\Bigr(b_j(\tilde x,\tilde y,t)-\sum^N_{\mu=0}\alpha^\mu_j(\tilde x,\tilde y,t)\Bigr)|_{\tilde x=x, \tilde y=y}=0,
\end{equation}
for every $(x,x,t)\in D\times D\times(\delta/4,2\delta)$, for all $\alpha, \beta\in\mathbb N^{2n}_0$, $\abs{\alpha}+\abs{\beta}\leq\ell$. Let 
$b(\tilde x,\tilde y,t,k)\in S^{n+1}(1;((U\times U)\cap(\ol M\times\ol M))\times(\delta/4,2\delta))$ with
$b(\tilde x,\tilde y,t,k)=0$ if $t\notin\supp\tau_{\delta}$ and 
\[\mbox{$b(\tilde x,\tilde y,t,k)\sim
\sum^{+\infty}_{j=0}b_j(\tilde x,\tilde y,t)k^{n+1-j}$ 
in $S^{n+1}(1;((U\times U)\cap(\ol M\times\ol M))\times(\delta/4,2\delta))$}.\]
From \eqref{e-gue220505yydVII} and \eqref{e-gue220505yydn}, 
we deduce that 
\[\Pi_{\tau^\dagger_{k\delta},s}(\tilde x,\tilde y)\equiv
\int_{\R}e^{ik\Psi(\tilde x,\tilde y,t)}b(\tilde x,\tilde y,t,k)dt 
\mod O(k^{-\infty})\  \ \mbox{on $(U\times U)\cap(\ol M\times\ol M)$},\]
where $\Pi_{\tau^\dagger_{k\delta},s}$ is the localization of 
$\Pi_{\tau^\dagger_{k\delta}}$ with respect to $s$, 
$\Pi_{\tau^\dagger_{k\delta},s}(\tilde x,\tilde y)$ 
is the distribution kernel of $\Pi_{\tau^\dagger_{k\delta},s}$. 
Summing up, we get Theorem~\ref{t-gue220512yyd}. 

\section{Proof of Theorem~\ref{t-gue230718yyd}
on smooth extension up to the boundary}\label{s-gue230718yydI}

In this section, we will prove Theorem~\ref{t-gue230718yyd}. We will use the same notations 
and assumptions as Theorem~\ref{t-gue230718yyd}. Fix $j=1, 2$. Let $L_j$ be the trivial line bundle over $M_j$ with $\abs{1}^2_{h^{L_j}}=e^{-2\phi}$ and let $\Pi_{\tau_{k\delta},j}$ be the partial Bergman kernel as \eqref{e-gue220512yydb} with respect to the $\mathbb R$-action $\eta_j$, flat Hermitian metric on $\mathbb C^n$ and $e^{-2k\phi}$. Since $F$ is $\mathbb R$-equivariant, it is straightforward to check that 
\begin{equation}\label{e-gue230718yyda}
\Pi_{\tau_{k\delta},1}(x,y)=\Pi_{\tau_{k\delta},2}(F(x),F(y)){\rm det\,}F'(x)\ol{{\rm det\,}F'(y)}\ \ \mbox{on $M_1\times M_1$}
\end{equation}
and $\det F'(x)$ is $\R$-invariant. Since $\det F'(x)$ 
is $\R$-invariant, we can repeat the proof of~\cite[Lemmas 8.2, 8.3]{HHLS}
and deduce that ${\rm det\,}F'(x)\in\cC^\infty(\ol M_1)$, 
${\rm det\,}(F^{-1})'(x)\in\cC^\infty(\ol M_2)$. 
\begin{lemma}\label{l-gue230723yyd}
Fix $q\in\pr M_2$. Let $z=(z_1,\ldots,z_n)$ be holomorphic coordinates 
defined near $q$ and let 
$s$ be a $\R$-equivariant holomorphic
frame of $L_2$ defined near $q$, 
$\abs{s}^2_{h^{L_2}}=e^{-2\phi}$ so that $z(q)=0$, 
$\phi(z)=\beta x_{2n}+O(\abs{z}^2)$, and
the defining function 
$\rho_2$ of $M_2$ has the form
\[\rho_2(z)=\sqrt{2}{\rm Im\,}z_n+
\sum^n_{j,\ell=1}\mu_{j,\ell}z_j\ol z_\ell+O(\abs{z}^3),\]
where $\beta\in\mathbb R$, $\mu_{j,\ell}\in\mathbb C$, 
$j, \ell=1,\ldots,n$. Then, for every $k\gg1$, we can find $a_{1,k},\ldots,a_{n,k}\in M_2$ 
so that for $k\gg1$, 
\begin{equation}\label{e-gue230723yyd}
\det\left(\frac{\pr}{\pr z_j}\Bigr(
\Pi_{\tau_{k\delta},2}(x,a_{\ell,k})\Bigr)\right)^{n}_{j,\ell=1}\Big|_{x=q}\neq0.
\end{equation}
\end{lemma}

\begin{proof}
We fix $\varepsilon>0$, $\varepsilon\ll1$. 
For every $j=1,\ldots,n$, let $a_{j,k}=(\alpha_1,\ldots,\alpha_{2n})$, $\alpha_u=0$ if $u\notin\set{2j-1,2n}$, $\alpha_{2j-1}=\frac{1}{k^{1+\varepsilon}}$, $\alpha_{2n}=-\frac{1}{k^{1+\frac{\varepsilon}{2}}}$. We can check that if $k\gg1$, then $a_{j,k}\in M_2$, $j=1,\ldots,n$.
From Theorem~\ref{t-gue220512yyd} and \eqref{e-gue220430yyd}, we can check that for all $j, \ell=1,\ldots,n-1$, 
\begin{equation}\label{e-gue230724yyd}
\begin{split}
&\lim_{k\To+\infty}k^{-n-1+\varepsilon}\frac{\pr}{\pr z_j}\Bigr(
\Pi_{\tau_{k\delta},2}(x,a_{\ell,k})\Bigr)|_{x=q}=0,\ \ \mbox{if $j\neq\ell$},\\
&\lim_{k\To+\infty}k^{-n-1+\varepsilon}\frac{\pr}{\pr z_j}\Bigr(
\Pi_{\tau_{k\delta},2}(x,a_{j,k})\Bigr)|_{x=q}\neq0,
\end{split}
\end{equation}
and 
\begin{equation}\label{e-gue230724yyda}
\lim_{k\To+\infty}k^{-n-2}\frac{\pr}{\pr z_n}\Bigr(
\Pi_{\tau_{k\delta},2}(x,a_{n,k})\Bigr)|_{x=q}\neq0.
\end{equation}
From \eqref{e-gue230724yyd} and \eqref{e-gue230724yyda}, we get \eqref{e-gue230723yyd}. 
\end{proof}

\begin{proof}[Proof of Theorem~\ref{t-gue230718yyd}]
Write $F(x)=(F_1(x),\ldots,F_n(x))\in\mathbb C^n$, $x\in M_1$. 
 We claim that 
 \begin{equation}\label{e-gue230718yydb}
\mbox{the matrix $\left(\frac{\pr F_j}{\pr z_\ell}\right)^n_{j,\ell=1}$ is bounded on $M_1$}.
 \end{equation}
Suppose that the claim \eqref{e-gue230718yydb} is not true.  Then, there are $\set{x_s}^{+\infty}_{s=1}\subset M_1$, $\lim_{s\To+\infty}x_s=p\in\pr M_1$, such that 
\begin{equation}\label{e-gue230725ycd}
\lim_{s\To+\infty}\abs{{\rm det\,}F'(x_s)}=+\infty.
\end{equation} 
Assume $\lim_{s\To+\infty}F(x_s)=q\in\pr M_2$. Let $z=(z_1,\ldots,z_n)$ be holomorphic coordinates defined near $q$ and let 
$s$ be a $\R$-equivariant holomorphic frame of $L_2$ defined near $q$, $\abs{s}^2_{h^{L_2}}=e^{-2\phi}$ so that $z(q)=0$, $\phi(z)=\beta x_{2n}+O(\abs{z}^2)$, 
\[\rho_2(z)=\sqrt{2}{\rm Im\,}z_n+\sum^n_{j,\ell=1}\mu_{j,\ell}z_j\ol z_\ell+O(\abs{z}^3),\]
where $\beta\in\mathbb R$, $\mu_{j,\ell}\in\mathbb C$, $j, \ell=1,\ldots,n$, $\rho_2$ denotes the defining function of $M_2$. From \eqref{e-gue230718yyda}, it is 
straightforward to check that 
\begin{equation}\label{e-gue230724yydh}
\begin{split}
&\left(\frac{\pr}{\pr z_j}\Bigr(\Pi_{\tau_{k\delta},1}(x,F^{-1}(a_{\ell,k}))({\rm det\,}F'(x))^{-1}\ol{({\rm det\,}F'(F^{-1}(a_{\ell,k})))^{-1}}\Bigr)\right)^n_{j,\ell=1}\\
&=\left(\frac{\pr F_\ell}{\pr z_j}(x)\right)^n_{j,\ell=1}\circ\left(\frac{\pr}{\pr z_j}\Bigr(\Pi_{\tau_{k\delta},2}(F(x),a_{\ell,k})\Bigr)\right)^n_{j,\ell=1},
\end{split}
\end{equation}
where $a_{\ell,k}\in M_2$, $\ell=1,\ldots,n$, are as in Lemma~\ref{l-gue230723yyd}. 
From \eqref{e-gue230725ycd} and \eqref{e-gue230724yydh}, we get 
\begin{equation}\label{e-gue230725ycdI}
{\rm det\,}\left(\frac{\pr}{\pr z_j}\Bigr(
\Pi_{\tau_{k\delta},2}(x,a_{\ell,k})\Bigr)\right)^{n}_{j,\ell=1}|_{x=q}=0. 
\end{equation}
From \eqref{e-gue230723yyd} and \eqref{e-gue230725ycdI}, we get a contradiction. 
The claim follows. 

Similarly, we can take derivatives of both side of \eqref{e-gue230724yydh} and repeat the argument above with minor changes and show that all derivatives of $F$ are bounded in $M_1$. 
Theorem~\ref{t-gue230718yyd} follows. 
\end{proof} 

\section{Application to pseudoconcave manifolds and examples}
\label{S:psconc}

Recall that a domain $M$ in a complex manifold $M'$
is said to have smooth boundary if there exists a smooth
function $\rho:M'\to\R$ (called defining function of $M$)
such that $M=\{p\in M':\rho(p)<0\}$ and $d\rho\neq0$
in a neighborhood of $\partial M$.
 
A relatively compact domain $M$ with smooth boundary
in a complex manifold $M'$ is called strictly pseudoconcave
if for some (and hence all) defining function $\rho:M'\to\R$
of $M$ the Levi form $\mathcal{L}_x(\rho)$ is negative definite
on the complex tangent space $\C T_x(\partial M)$ for all $x\in\partial M$.
If $M$ is strictly pseudoconcave we can find a defining function
$\rho$ such that $-\rho$ is strictly plurisubharmonic on a neighborhood $U$ of
$\partial M$ (see e.\,g.\ \cite[Proposition 2.14]{Ran}). 
In this case, for $c$ sufficiently close to zero, the domains
$M'_c:=\{p\in M':\rho(p)<c\}$ have smooth boundary and
are strictly pseudoconcave. Hence a neighborhood of $M$ is
1-concave and $M$ is a sublevel set of this 1-concave manifold.
Recall that a complex manifold $M'$ is called  
$1$-concave in the sense of Andreotti--Grauert \cite{AG} 
if there exists a smooth function $\varphi:X\longrightarrow (a,b\,]$
where  
$a=\inf\varphi\in\{-\infty\}\cup\R$, $b\in\R$, such that 
$M'_c:=\{\varphi>c\}\Subset M'$ for  
all $c\in (a,b\,]$ and $\imath\partial\overline{\partial}\varphi$ 
is strictly plurisubharmonic outside an exceptional compact set $K\Subset M'$. 

\begin{proof}[Proof of Corollary \ref{C:Moi}]
(a) 
Since $H^0_{I_k}(\ol M,L^k)\subset H^0(M,L^k)$ 
we have by \eqref{e-gue220513yydI}
$$\liminf_{k\to\infty}k^{-n}\dim H^0(M,L^k)>0.$$
Hence there exists $C_1>0$ and $k_0$ such that
for all $k\geq k_0$ we have $\dim H^0(M,L^k)\geq C_1k^n$.
On the other hand there exists $C_2>0$
such that $\dim H^0(M,L^k)\leq C_2k^{\varrho_k}$ 
for all $k$, where $\varrho_k$ is the maximal rank
of the Kodaira map $\Phi_k:M\dashrightarrow\mathbb{P}(H^0(M,L^k)^*)$.
It follows that for $k\geq k_0$ we have $\varrho_k=n$,
and hence the Kodaira dimension of $L\to M$ is maximal
and equals the dimension of $M$. 

(b) Since $\dim M\geq3$ and $M$ is 1-concave, it follows from 
\cite[Theorem 5.2]{M96} that $M$ admits a Moishezon compactification
$\widetilde{M}$.
\end{proof}

In the end of this section, we give some simple but non-trivial examples. For $z=(z_1,\ldots,z_N)\in\C^N$
we denote by $[z]=[z_1;\ldots;z_N]\in\C\mathbb{P}^{N-1}$
the class of $z$ in the projective space.
We denote by $|z|$ the Euclidean norm of $z$.
Let 
%

\begin{equation}\label{e-gue220616ycdm}
M:=\left\{[z]\in\mathbb C\mathbb P^{N-1};
\frac{1}{|z|^2}\Bigl(-\sum_{j=1}^\ell\lambda_j\abs{z_j}^2
+\sum_{j=\ell+1}^N\lambda_{j}\abs{z_j}^2\Bigr)
+\varepsilon f([z])<0\right\},
\end{equation}
where $\lambda_j>0$, $j=1,\ldots,N$, $\ell\in\mathbb N$, $\ell\geq1$, 
$\varepsilon>0$ is a small constant and 
$f\in\cC^\infty(\mathbb C\mathbb P^{N-1})$ is a non-negative 
real valued function with $f[z_1,z_2,z_3,\ldots,z_N]=0$ if 
$z_1=z_{\ell+1}$, $z_2=z_3=\ldots z_{\ell}=z_{\ell+2}=\ldots=z_N=0$, 
the support $f$ is in some small open neighborhood $U$ of 
$z_1=z_{\ell+1}$, $z_2=z_3=\ldots z_{\ell}=z_{\ell+2}=\ldots=z_N=0$. 
Moreover, we assume that 
$$f([z])=(1-\chi(\frac{|z_1|^2+\cdots+
|z_l|^2}{|z|^2}))g(|z_1|, |z_2|, \cdots, |z_l|),$$ where 
$\chi\in \cC_c^\infty(\mathbb R)$ and $\chi\equiv1$ in a neighborhood $0$. Here, 
$g\in\cC^\infty({\mathbb R^\ell\setminus\{0\}},\mathbb R)$, 
$g(x_1, x_2, \ldots, x_l)=0$ if $x_2=0$ and 
$g(\lambda x_1,\ldots,\lambda x_\ell)=g(x_1,\ldots,x_\ell)$, 
for all $\lambda>0$.
We take  $\varepsilon$ small enough so that the boundary $X$ 
is a smooth manifold. Moreover, when $\varepsilon$ is small, 
$X$ is a CR manifold and the Levi form of $X$ is non-degenerate of 
constant signature $(n_-,n_+)$, $n_-=\ell-1$, $n_+=N-\ell-1$. 
Let $(\mu_{\ell+1},\ldots,\mu_N)\in\R^{N-\ell}_+$. 
Suppose that $(\mu_{\ell+1},\ldots,\mu_N)$ are linear independent over 
$\mathbb Q$. $X$ admits a $\R$-action: 
\[\begin{split}
\R\times X&\To X,\\
(\eta,[z_1,\ldots,z_N])&\To\eta\circ[z_1,\ldots,z_N]:=
[z_1,\ldots,z_{\ell},e^{i\mu_{\ell+1}\eta}z_{\ell+1},\ldots,
e^{i\mu_N\eta}z_{N}]\in X.\end{split}\]
The $\R$-action is CR. We claim that the $\R$-action is transversal 
if $\varepsilon$ is small. 
Let $T$ be the vector field induced by the  $\R$-action. 
We can check that 
\[T=(i\mu_{\ell+1}z_{\ell+1}\frac{\pr}{\pr z_{\ell+1}}
-i\mu_{\ell+1}\ol{z_{\ell+1}}\frac{\pr}{\pr\ol{z_{\ell+1}}})
+\ldots+(i\mu_{N}z_{N}\frac{\pr}{\pr z_{N}}
-i\mu_{N}\ol{z_{N}}\frac{\pr}{\pr\ol {z_{N}}}).\]
Let $J$ the complex structure map of $T\mathbb C\mathbb P^{N-1}$. 
Let 
$$\rho([z_1,\ldots,z_N)):=\frac{1}{\abs{z}^2}
\Bigr(-\lambda_1\abs{z_1}^2-\ldots
-\lambda_\ell\abs{z_\ell}^2+\lambda_{\ell+1}\abs{z_{\ell+1}}^2
+\ldots+\lambda_N\abs{z_N}^2\Bigr)+\varepsilon f([z_1,\ldots,z_N]).$$
We can check that on $X$, 
\begin{equation}\label{e-gue220616ycdn}
d\rho(JT)\big|_X=
\frac{1}{\abs{z}^2}\Bigr(-2\mu_{\ell+1}
\lambda_{\ell+1}\abs{z_{\ell+1}}^2-
\ldots-2\mu_{N}\lambda_N\abs{z_{N}}^2\Bigr)\Big|_X<0.
\end{equation}
Thus, the  $\R$-action is transversal and CR. 

Let $U_j:=\set{[z_1,z_2,z_3,\ldots,z_N])\in
\mathbb C\mathbb P^{N-1};\, z_j\neq0}$, $j=1,\ldots,N$.
Since $f$ is non-negative, $(z_1,\ldots,z_\ell)\neq0$ if 
$[z_1,z_2,z_3,\ldots,z_N]\in\ol M$. 
Let 
\[M':=\set{[z_1,z_2,z_3,\ldots,z_N])\in\mathbb C\mathbb P^{N-1};\, 
(z_1,\ldots,z_\ell)\neq0}.\] Then, $\ol M\subset M'$. 
Let $L\To M'$ be the holomorphic line bundle given by the transition functions 
$s_k=\frac{z_k}{z_j}s_j$, where $s_j:=1$ on $U_j$, $j, k=1,\ldots,\ell$. 
Thus, $L$ is a $\R$-invaraint holomorphic line bundle over $M'$. 
Let $h^L$ be the $\R$-invariant Hermitian metric of $L$ given by $\abs{s_j}^2_{h^L}:=\abs{\frac{z_j}{z}}^2$, $j=1,\ldots,\ell$. 
For simplicity, we work on $U_1$ and let 
$(z_1,\ldots,z_{N-1})\in\mathbb C^{N-1}\To [1,z_2,z_3,\ldots,z_N]\in U_1$ 
be local holomorphic coordinates of $U_1$. On $U_1$, 
we can check that $\phi=\frac{1}{2}\log(1+\abs{z'}^2)$ and 
\begin{equation}\label{e-gue220713yyd}
JT(\phi)=-\frac{\sum^N_{j=\ell+1}\mu_j\abs{z_j}^2}{1+\abs{z'}^2}, 
\end{equation} 
where $z'=(z_2,\ldots,z_N)$. From \eqref{e-gue220713yyd}, 
we see that if $\lambda_j$, $j=1,\ldots,N$, are small enough, 
then conditions \eqref{e-gue220710yyd} hold.


\end{document}